\documentclass{amsart}
\usepackage{graphicx}

\usepackage{mathtools}
\DeclarePairedDelimiter{\ceil}{\lceil}{\rceil}
\DeclarePairedDelimiter{\floor}{\lfloor}{\rfloor}

\usepackage{amsmath,amsfonts,amsthm,amssymb,graphics,color,datetime}
\usepackage{hyperref}

\newtheorem{example}{Example}
\newtheorem*{acknowledgements}{Acknowledgements}
\newtheorem{theorem}{Theorem}
\newtheorem{prop}{Proposition}
\newtheorem{lemma}{Lemma}
\newtheorem{cor}{Corollary}

\newtheorem{claim}{Claim}

\newtheorem{conj}{Conjecture}

\newtheorem{remark}{Remark}
\newtheorem{defn}{Definition}

\numberwithin{equation}{section}

        \definecolor{pink}{rgb}{1,0,1}
        \definecolor{purple}{rgb}{0.4,0.2,1}

\newcommand{\supp}{\operatorname{Supp}}
\newcommand{\re}{\operatorname{Re}}
\newcommand{\im}{\operatorname{Im}}

\newcommand{\Ker}{\operatorname{Ker}}

\newcommand{\pa}{\partial}
\newcommand{\eps}{\varepsilon}

\newcommand{\calV}{\mathcal{V}}
\newcommand{\del}{\partial}

\newcommand{\tr}{{\rm Tr}}
\newcommand{\Tr}{{\rm Tr}}

\newcommand{\cD}{{\mathcal{D}}}
\newcommand{\cN}{{\mathcal{N}}}
\newcommand{\cM}{{\mathcal{M}}}
\newcommand{\cT}{{\mathcal{T}}}

\newcommand{\cR}{{\mathcal{R}}}

\newcommand{\cA}{{\mathcal{A}}}
\newcommand{\cB}{{\mathcal{B}}}

\newcommand{\N}{\mathbb{N}}
\newcommand{\R}{\mathbb{R}}
\newcommand{\D}{\mathbb{D}}
\newcommand{\cV}{\mathcal{V}}
\newcommand{\C}{\mathbb{C}}
\newcommand{\Z}{\mathbb{Z}}

\newcommand{\cC}{\mathcal{C}}
\newcommand{\cL}{\mathcal{L}}

\DeclareMathOperator{\Dom}{Dom}
\newcommand{\wt}[1]{\widetilde{#1}}

\title{A Polyakov formula for sectors}  
\author{Clara L. Aldana} 
\author{Julie Rowlett}
\address{University of Luxembourg\\ 6, avenue de la Fonte\\ L-4364 Esch-sur-Alzette, Luxembourg}
\email{clara.aldana@uni.lu}
\address{Department of Mathematics \\ Chalmers University of Technology and the University of Gothenburg \\ 41296 Gothenburg, Sweden}
\email{julie.rowlett@chalmers.se}

\thanks{C.L. Aldana was supported by ANR grant ACG: ANR-10-BLAN 0105 and by the Fonds National de la Recherche, Luxembourg 7926179}
\begin{document}

\begin{abstract}
We consider finite area convex Euclidean circular sectors.  We prove a variational Polyakov formula which shows how the zeta-regularized determinant of the Laplacian varies with respect to the opening angle. Varying the angle corresponds to a conformal deformation in the direction of a conformal factor with a logarithmic singularity at the origin.  We compute explicitly all the contributions to this formula coming from the different parts of the sector. In the process, we obtain an explicit expression for the heat kernel on an infinite area sector using Carslaw-Sommerfeld's heat kernel. We also compute the zeta-regularized determinant of rectangular domains of unit area and prove that it is uniquely maximized by the square. 
\end{abstract}

\maketitle

\section{Introduction}
Polyakov's formula expresses a difference of zeta-regularized determinants of
Laplace operators, an anomaly of global quantities, in terms of simple local quantities. 
The main applications of Polyakov's formula are in differential geometry and mathematical physics. In mathematical physics, this formula arose in the study of the quantum theory of strings \cite{polyakov} and has been used in connection to conformal quantum field theory \cite{bpz} and Feynmann path integrals \cite{hawk}.  

In differential geometry, Polyakov's formula was used in the work of Osgood, Phillips and Sarnak \cite{OPS} to prove that under certain restrictions on the Riemannian metric, the determinant is maximized at the uniform metric inside a conformal class. Their result holds for smooth closed surfaces and for surfaces with smooth boundary. This result was generalized to surfaces with cusps and funnel ends in \cite{AAR}. The techniques used in this article are similar to the ones used by the first author in \cite{Aldana-ae} to prove a Polyakov formula for the relative determinant for surfaces with cusps.

We expect that the formula of Polyakov we shall demonstrate here will have applications to differential geometry in the spirit of \cite{OPS}. Our formula is a step towards answering some of the many open questions for domains with corners such as polygonal domains and surfaces with conical singularities:  what are suitable restrictions to have an extremal of the determinant in a conformal class as in \cite{OPS}? Will it be unique? Does the regular $n$-gon maximize the determinant on all $n$-gons of fixed area?  What happens to the determinant on a family of $n$-gons which collapses to a segment?  

\subsection{The zeta regularized determinant of the Laplacian} 
Consider a smooth $n$-dimensional manifold $M$ with Riemannian metric $g$. 
We denote by $\Delta_g$ the Laplace operator associated to the metric $g$. We consider the positive Laplacian $\Delta_g\geq 0$.
If $M$ is compact and without boundary, or if $M$ has non-empty boundary and suitable boundary conditions are imposed, then the eigenvalues of the Laplace operator form an increasing,
discrete subset of $\R^+$,
$$0 \leq \lambda_1 \leq \lambda_2 \leq \lambda_3 \leq \ldots.$$
These eigenvalues tend toward infinity according to Weyl's law \cite{weyl},
$$\lambda_k^{\frac n 2} \sim \frac{(2 \pi)^n k}{\omega_n \textrm{Vol}(M)}, \quad \textrm{ as } k \to \infty,$$
where $\omega_n$ is the volume of the unit ball in $\R^n$.

Ray and Singer generalized the notion of determinant of matrices to the Laplace-de Rham operator on forms using an associated zeta function
\cite{rays}. The spectral zeta function associated to the Laplace operator is defined for $s\in \C$ with $\re(s)>\frac{n}{2}$ by 
$$\zeta(s) := \sum_{\lambda_k >0} \lambda_k ^{-s}.$$
By Weyl's law, the zeta function is holomorphic on the half-plane $\{\re(s) > n/2\}$, and it is well known that the heat equation can be used to prove that the zeta function admits a meromorphic extension to $\C$ which is holomorphic at $s=0$ \cite{rays}.
Consequently, the zeta-regularized determinant of the Laplace operator may be defined as 
\begin{equation} \label{zdet} \det(\Delta) := e^{-\zeta'(0)}.\end{equation}
In this way, the determinant of the Laplacian is a number that depends only on the spectrum; it is a spectral invariant.  Furthermore, it is also a global invariant, meaning that in general it can not be expressed as an integral over the manifold of local quantities.

\subsection{Polyakov's formula for smooth surfaces} 
Let $(M,g)$ be a smooth Riemannian surface. Let $g_t=e^{2\sigma(t)}g$ be a one-parameter family of metrics in the conformal class of $g$ depending smoothly on $t \in (-\epsilon, \epsilon)$ for some $\epsilon > 0$. Assume that each conformal factor $\sigma(t)$ is a smooth function on $M$. The Laplacian for the metric $g_t$ relates to the Laplacian of the metric $g$ via 
$$\Delta_{g_t} = e^{-2\sigma(t)} \Delta_g.$$
The variation of the Laplacian for the metric $g_t$ with respect to the parameter $t$ is 
\begin{equation} \left . \partial_t \Delta_{g_t}\right|_{t=0} = -2 \sigma'(0) \Delta_{g_0}, \quad g_0 = e^{2\sigma(0)} g.\label{eq:varXidir}\end{equation}

In this setting, Polyakov's formula gives the variation of the determinant of the family of conformal Laplacians $\Delta_{g_t}$ with respect to the parameter $t$ of the \em conformal factor \em  $\sigma(t)$, \cite{KoKo} and \cite{AAR}, 
\begin{equation}
\pa_t \log \det(\Delta_{g_t}) = {-\frac{1}{24\pi}} \int_{M} \sigma'(t) \ \text{Scal}_t \ dA_{g_t}
+ \pa_t \log \text{Area}(M,g_t),
\label{polclosed-v1}\end{equation}
where $\text{Scal}_t $ denotes the scalar curvature of the metric $g_t$. This is the type of formula that we demonstrate here and may refer to it as either the differentiated or variational Polyakov formula or simply Polyakov's formula. The classical form of Polyakov's formula is the \lq\lq integrated form\rq\rq which expresses the determinant as an anomaly; for a surface 
$M$ with smooth boundary it was first proven by Alvarez \cite{alv}; see also \cite{OPS}.
There are two main difficulties which distinguish our work from the case of closed surfaces:  (1) the presence of a geometric singularity in the domain or surface and (2) the presence of an analytic singularity in the conformal factor.  
\subsection{Conical singularities}
Analytically and geometrically, the presence of even the simplest conical singularity, a corner in a Euclidean domain, has a profound impact on the Laplace operator. As in the case of a manifold with boundary, the Laplace operator is not essentially self-adjoint.  It has many self adjoint extensions, and the spectrum depends on the choice of self-adjoint extension. Thus, the zeta-regularized determinant of the Laplacian also depends upon this choice \cite{gm-zeta}. In addition, conical singularities add regularity problems that do not appear when the boundary of the domain or manifold is smooth. 

In recent years there has been progress towards understanding the behavior of the determinant of certain self-adjoint extensions of the Laplace operator, most notably the Friedrichs extension, on  surfaces with conical singularities. This progress represents different aspects that have been studied by Kokotov \cite{k}, Hillairet and Kokotov \cite{hk}, Loya et al \cite{LoMcDPa}, Spreafico \cite{Spre}, and Sher \cite{Sher}. In particular, the results by Aurell and Salomonson in \cite{AuSa} inspired our present work.  Using heuristic arguments they computed a formula for the contribution of the corners to the variation of the determinant on a polygon \cite[eqn (51)]{AuSa}. Here we use different techniques to rigorously prove the differentiated Polyakov formula for an angular sector.  Our work is complementary to those mentioned above since the dependence of the determinant of the Friedrichs extension of the Laplacian with respect to changes of the cone angle has not been addressed previously.  In addition, our formula can be related to a variational principle.  

\subsection{Organization and main results}
In \S \ref{s:prelim}, we present the framework of this article and develop the requisite geometric and analytic tools needed to prove our first main result, Theorem \ref{var-det} below. In \S \ref{ss:stae} and \S \ref{Scarslaw} we prove the following theorem which is a key ingredient in the proof of Theorem \ref{var-det}.  

\begin{theorem}
Let $\mathcal M_f$ denote the multiplication operator by the function $f$, so that for a function $\phi$, 
$$\mathcal M_f : \phi \mapsto f \phi.$$
Let $S_\alpha$ denote a finite circular sector of opening angle $\alpha \in (0,\pi)$, and let $e^{-t \Delta_\alpha}$ denote the heat operator associated to the Dirichlet extension of the Laplacian. Then, the operator ${\mathcal M}_{\left( 1 + \log(r) \right)}e^{-t\Delta_{\alpha}}$ on $S_{\alpha}$ is trace class and its trace admits an  
asymptotic expansion at $t\to 0$ of the form 
\begin{equation} \label{exp-exists} \tr_{S_\alpha}  \left( {\mathcal M}_{\left( 1 + \log(r) \right)}e^{-t\Delta_{\alpha}} \right) \sim  a_0 t^{-1} + a_{1}t^{-\frac12} + a_{2, 0} \log(t) + a_{2, 1} + O(t^{1/2}). \end{equation} 
\label{th-exp-exists}
\end{theorem}

The trace in Theorem \ref{th-exp-exists} can be rewritten as the following integral:
$$\tr_{S_\alpha}  \left( {\mathcal M}_{\left( 1 + \log(r) \right)}e^{-t\Delta_{\alpha}} \right) = \int_{S_\alpha} (1+\log(r)) H_{S_{\alpha}}(t,r,\phi,r,\phi) r dr d\phi,$$
where $H_{S_\alpha}$ denotes the Schwartz kernel of $e^{-t \Delta_\alpha}$, also called the heat kernel.
Our next theorem is a preliminary variational Polyakov formula.

\begin{theorem}\label{var-det} Let $\{S_{\gamma}\}_{\gamma \in (0, \pi)}$ be a family of finite circular 
sectors in $\R^2$, where $S_\gamma$ has opening angle $\gamma$ and unit radius. Let $\Delta_{\gamma}$ be the Euclidean Dirichlet Laplacian on $S_\gamma$. Then for any $\alpha \in (0, \pi)$ 
\begin{equation} \label{variation}
\left.\frac{\pa}{\pa \gamma} \big(-\log(\det(\Delta_{\gamma}))\big)\right|_{\gamma=\alpha} = 
\frac{2}{\alpha} \left( -\gamma_{e} a_{2,0} + a_{2,1} \right).  
 \end{equation} 
Above, $\gamma_{e}$ is the Euler constant, and $a_{2,0}$ is the coefficient of $\log(t)$ and $a_{2,1}$ is the constant coefficient in the asymptotic expansion as $t \to 0$ given in equation (\ref{exp-exists}).  
 
If the radial direction is multiplied by a factor of $R$, which is equivalent to scaling the metrics by $R^2$, the determinant of the
Laplacian transforms as
$$\det(\Delta_{\alpha}) \mapsto R^{-2 \zeta_{\Delta_{\alpha}}(0)} \det(\Delta_{\alpha}).$$
\end{theorem}

The proof of the preceding results comprises \S \ref{s:prelim} and \S \ref{s:formula}.   In \S \ref{s:qcrd} we prove the following theorem.  Its proof not only illustrates the method we shall use to compute the general case of a sector of opening angle $\alpha \in (0, \pi)$ but also shall be used in the proof of the general case.

\begin{theorem} \label{ctpio2}  Let $S_{\pi/2}\subset \R^2$ be a circular sector of opening angle $\pi/2$ and radius one. Then the variational Polyakov formula is 
$$
\left.\frac{\pa}{\pa \gamma} \big(-\log(\det(\Delta_{S_\gamma}))\big)\right|_{\gamma=\pi/2} = 
\frac{-\gamma_e}{4\pi} + \frac{5}{12 \pi}, $$ 
where $\gamma_e$ is the Euler-Mascheroni constant.
\end{theorem}

In \S \ref{Scarslaw} we determine an explicit formula for Sommerfeld-Carslaw's heat kernel for an infinite sector with opening angle $\alpha$.  This allows us to compute the contribution of the corner at the origin to the variational Polyakov formula, completing the proof of Theorem \ref{th-exp-exists}.  Moreover, these calculations allow us to refine the preliminary variational Polyakov formula by determining an explicit formula. 

\begin{theorem} \label{allsectors} Assume the same hypotheses as in Theorem \ref{var-det}.  Let 
$$ k_{min} = \ceil*{ \frac{-\pi}{2\alpha} }, \textrm{ and } k_{max} = \floor*{\frac{\pi}{2\alpha}} \textrm{ if } \frac{\pi}{2\alpha} \not\in \Z, \textrm{ otherwise } k_{max} = \frac{\pi}{2\alpha} - 1, $$ 
and 
$$W_{\alpha} =\left\{  k \in \left( \Z \bigcap \left[k_{min},  k_{max}\right]\right) \setminus \left\{ \frac{\ell\pi}{\alpha} \right\}_{\ell \in \Z} \right\} .$$
Then 
\begin{multline*}
\begin{split}
\frac{\pa}{\pa \gamma} &\left. \big(-\log(\det(\Delta_{\gamma}))\big)\right|_{\gamma=\alpha} = \frac{\pi}{12\alpha^2} + \frac{1}{12\pi}\\
& +  \sum_{k\in W_{\alpha}} \frac{-2\gamma_e + \log(2) - \log\left({1-\cos(2k\alpha)}\right) }{4 \pi (1-\cos(2k\alpha))} \\
& - (1-\delta_{\alpha, \frac{\pi}{n}})\ \frac{2}{\alpha} \sin(\pi^2/\alpha) \int_{-\infty} ^\infty 
\frac{\gamma_e + \log(2) - \log(1+\cosh(s))}{16\pi(1+\cosh(s))(\cosh(\pi s /\alpha) - \cos(\pi^2/\alpha))} ds, 
\end{split}
\end{multline*}
where $n\in \N$ is arbitrary and $\delta_{\alpha, \frac{\pi}{n}}$ denotes the Kronecker delta.
\end{theorem}

Here is a short list of examples. Let us denote 
$${\mathcal{S}}(\alpha):=\left.\frac{\pa}{\pa \gamma} \big(-\log(\det(\Delta_{\gamma}))\big)\right|_{\gamma=\alpha}.$$ 
Then ${\mathcal{S}}(\alpha)$ and the set $W_{\alpha}$ have the following values:  
\begin{enumerate}
\item $\alpha=\frac{\pi}{4}$, $W_{\frac{\pi}{4}}= \{-2,\pm1,\}$, ${\mathcal{S}}(\frac{\pi}{4})= \frac{-5\gamma_e}{4\pi} + \frac{\log(2)}{4\pi}+\frac{17}{12 \pi}\sim 0.2764$
\item $\alpha=\frac{\pi}{3}$, $W_{\frac{\pi}{3}}= \{-1,1\}$, 
${\mathcal{S}}(\frac{\pi}{3})= \frac{-\gamma_e}{2\pi} + \frac{\log(2)}{2\pi}+\frac{5}{6 \pi}\sim 0.2837$
\item $\alpha=\frac{\pi}{2}$, $W_{\frac{\pi}{2}}=\{-1\}$, ${\mathcal{S}}(\frac{\pi}{2})= \frac{-\gamma_e}{4\pi} + \frac{5}{12 \pi}\sim 0.0867$
\item For $\alpha \in ]\frac{\pi}{2},\pi[$, $W_{\alpha}=\emptyset$, but $\sin(\pi^2/\alpha)\neq 0$. Thus, the integral in Theorem \ref{allsectors} determines ${\mathcal{S}}(\alpha)$.  For example, with $\alpha=\frac{2\pi}{3}$, the integral converges rapidly, and a numerical computation gives an approximate  value of $0.0075015$. Hence  ${\mathcal{S}}(\frac{2\pi}{3}) \sim 0.0933723$.
\end{enumerate}

Generalizing our Polyakov formula to Euclidean polygons shall require additional considerations because one cannot change the angles independently.  We expect that the results obtained here will help us to achieve these generalizations with the eventual goal of computing closed formulas for the determinant on planar sectors and Euclidean polygons.  In the latter setting one naturally expects the following:

\begin{conj} 
  Amongst all convex $n$-gons of fixed area, the regular one maximizes the determinant.  
  \end{conj}

We conclude this work by proving in \S \ref{s:rect} the following result which shows that for the case of rectangular domains, the conjecture holds.

 \begin{theorem} \label{thm:mdr} 
  Let $R$ be a rectangle of dimensions $L \times L^{-1}$.  Then the zeta regularized determinant is uniquely maximized for $L=1$, and tends to $0$ as $L \to 0$ or equivalently as $L \to \infty$.  
\end{theorem} 

\begin{acknowledgements}
The authors are grateful to Gilles Carron and Rafe Mazzeo for their interest in the project and many useful conversations. The second author is also grateful to Lashi Bandara for productive discussions.  We specially acknowledge Werner M\"uller's support and remarks.  Finally, we thank strongly the anonymous referee for insightful comments and constructive criticism.  
\end{acknowledgements}

\section{Geometric and analytic settings}  \label{s:prelim}
In this section we present the framework of this article and fix the geometric and analytic settings required to proof Theorem \ref{var-det}.  
\subsection{The determinant and Polyakov's formula} \label{standardarg} 
Let us describe briefly the classical deduction of Polyakov's formula, since we will use the same argument. 
Let $(M,g)$ be a smooth Riemannian surface with or without boundary. If $\pa M \neq \emptyset$, we consider the Dirichlet boundary condition, in which case $\Ker(\Delta_g)=\{0\}$.  

Let $H_g(t, z, z')$ denote the heat kernel associated to $\Delta_g$. It is the fundamental solution to the heat equation on $M$
\begin{eqnarray*}
(\Delta_g + \pa_t) H_g(t, z, z') &=& 0 \quad (t>0),\\
H_g(0, z, z') &=& \delta(z-z').
\end{eqnarray*}
The heat operator, $e^{-t\Delta_g}$ for $t>0$, is trace class, and the trace is given by
$$\tr(e^{-t{\Delta_g}}) = \int_M H_g(t, z, z) dz = \sum_{\lambda_k \geq 0} e^{-\lambda_k t}.$$
The zeta function and the heat trace are related by the Mellin transform
\begin{equation} \label{mellin} \zeta_{\Delta_g}(s) = \frac{1}{\Gamma(s)} \int_0 ^\infty t^{s-1} \tr(e^{-t{\Delta_g}}-P_{\Ker(\Delta_g)}) dt,  
\end{equation}
where $P_{\Ker(\Delta_g)}$ denotes the projection on the kernel of $\Delta_g$.

It is well known that the heat trace has an asymptotic expansion for small values of $t$ \cite{Gilkey}. 
This expansion has the form
$$\tr(e^{-t\Delta_g}) = a_0 t^{-1} + a_1 t^{\frac{-1}{2}} + a_2  + O(t^{\frac12}).$$
The coefficients $a_j$ are known as the heat invariants. They are given in terms of the curvature tensor and its derivatives as well as the geodesic curvature of the boundary in case of boundary.  By (\ref{mellin}) and the short time asymptotic expansion of the heat trace
$$\zeta_{\Delta_g} (s) = \frac{1}{\Gamma(s)} \left\{ \frac{a_0}{s-1} + \frac{a_1}{s-\frac{1}{2}}
+ \frac{a_2-\dim(\Ker(\Delta_g))}{s}+ e(s) \right\},$$
where $e(s)$ is an analytic function on $\re(s)> -1$. 
The regularity of $\zeta_{\Delta_g}$ at $s=0$ and hence the fact that the zeta regularized determinant of the Laplacian is well defined by (\ref{zdet}) both follow from the above expansion together with the fact that $\Gamma(s)$ has simple pole at $s=0$. 

Let $\{\sigma(\tau), \ \tau \in (- \epsilon, \epsilon)\}$ be a family of smooth conformal factors which depend on the parameter $\tau$ for some $\epsilon > 0$. Consider the corresponding family of conformal metrics $\{h_\tau = e^{2\sigma(\tau)} g,  \ \tau \in (- \epsilon, \epsilon)\}$.  To prove Polyakov's formula one first differentiates the spectral zeta function $\zeta_{\Delta_{h_\tau}}(s)$ with respect to $\tau$. This requires differentiating the trace of the heat operator. Then, after integrating by parts, one obtains
\begin{equation*} \label{dzeta}  \pa_\tau \zeta_{\Delta_{h_\tau}}(s) = \frac{s}{\Gamma(s)} \int_0 ^\infty  t^{s-1} \tr \big(2\cM_{\sigma'(\tau)}(e^{-t\Delta_{h_\tau}}-P_{\Ker(\Delta_{h_\tau})})\big) \, dt , \end{equation*}
where $\cM_{\sigma'(\tau)}$ denotes the operator multiplication by the function $\sigma'(\tau)$.  The integration by parts is again facilitated by the pole of $\Gamma(s)$ at $s=0$.

If the ma\-ni\-fold is compact, and the metrics and the conformal factors are smooth, then the operator $\cM_{\sigma'(\tau)} e^{-t\Delta_{h_\tau}}$ is trace class, and the trace behaves well for $t$ large.  As
$t\to 0$ the trace also has an asymptotic expansion of the form
\begin{multline*}
\Tr (\cM_{\sigma'(\tau)}e^{-t\Delta_{h_\tau}} ) \sim  a_0(\sigma'(\tau),h_\tau) t^{-1} + a_1(\sigma'(\tau),h_\tau)  t^{-\frac12}\\ 
+ a_2(\sigma'(\tau),h_\tau) -\dim(\Ker(\Delta_{h_\tau})) + O(t^{\frac12})
\end{multline*} 
The notation $a_j(\sigma'(\tau),h_\tau)$ is meant to show that these are the coefficients of the given trace, which depend on $\sigma'(\tau)$ and on the metric $h_\tau$. The dependence on the metric is through its associated heat operator. 

Therefore, the derivative of $\zeta_{\Delta_{h_\tau}}'(0)$ at $\tau = 0$ is simply given by
\begin{equation*} 
\left. \pa_{\tau} \zeta'_{\Delta_{h_\tau}}(0)\right|_{\tau=0} = 2 \ \big(a_2(\sigma'(0) ,h_0)-\dim(\Ker(\Delta_{h_0}))\big).
\end{equation*}
Polyakov's formula in (\ref{polclosed-v1}) is exactly this equation. 

\subsection{Euclidean sectors} 
Let $S_\gamma\subset \R^2$ be a finite circular sector with opening angle $\gamma\in (0,\pi)$ and radius $R$.
The Laplace operator $\Delta_\gamma$ with respect to the Euclidean metric is a priori defined on smooth functions with compact support within the open sector. It is well known that the Laplacian is not an essentially self
adjoint operator since it has many self adjoint extensions;  see e.g. \cite{domains} and \cite{les}.  The largest of these is the extension to
$$\Dom_{max}(\Delta_{\gamma}) = \{u\in L^2(S_\gamma) \vert \Delta_\gamma u \in L^2(S_\gamma) \}$$

For several reasons the most natural or standard self adjoint extension is the Friedrichs extension whose domain, $\Dom_F(\Delta_\gamma)$, is
defined to be the completion of
$$C_0^{\infty}(S_\gamma) \text{ w.r.t the norm } \|\nabla f\|_{L^2}$$
intersected with $\Dom_{max}$.  For a smooth domain $\Omega \subset \R^2$, it is well known that
$$\Dom_F(\Delta_{\Omega}) = H^1_0(\Omega) \cap H^2(\Omega).$$
The same is true if the sector is convex which we shall assume; see \cite[Theorem 2.2.3]{Grisvard} and
\cite[Chapter 3, Lema 8.1]{Ladyz-Ural}.

\begin{remark} \label{indep}  Let $S = S_{\gamma,R}$ be a planar circular sector of opening angle $
\gamma \in (0, \pi)$, radius $R>0$, and $S' = S_{\gamma', R'}$ be a circular sector of opening angle $\gamma' \in (0, \pi)$ and
radius $R'> 0$.
Then map $\Upsilon : S \to S'$ defined by $\Upsilon (\rho,\theta) = \left(\frac{R' \rho}{R},\frac{\gamma' \theta}{\gamma} \right) = (r,\phi)$
induces a bijection
$$\Upsilon^* : C_c ^\infty (S') \stackrel{\cong}{\longrightarrow} C_c ^\infty (S), \quad f \mapsto \Upsilon^*f := f \circ \Upsilon.$$
This bijection extends to the domains of the Friedrichs extensions of the corresponding Laplace operator.
Furthermore, under this map, the corresponding $L^2$ norms are equivalent, i.e., there exist constants $c, C>0$ such that
for any $f\in L^2(S')$,
$$c \| f \|_{L^2(S')} \leq \|\Upsilon^*f \|_{L^2(S)} \leq \ C \| f\|_{L^2(S')}.$$

The same holds for the norms on the corresponding Sobolev spaces $H^k$ for $k \geq 0$. In spite of inducing an equivalence between the different domains, this map is not useful for our purposes since it does not produce a conformal transformation of the Euclidean metric.
\end{remark}

To understand how the determinant of the Laplacian changes when the angle of the sector varies requires differentiating the spectral zeta function with respect to the angle \begin{equation}
\frac{\pa}{\pa \gamma} \ \zeta_{S_\gamma} (s) =
\frac{\pa}{\pa \gamma} \frac{1}{\Gamma(s)} \int_{0}^{\infty} t^{s-1} \Tr_{L^{2}(S_\gamma, g)}(e^{-t{\Delta_\gamma}}-P_{\Ker(\Delta_\gamma)}) dt.
\label{eq:derzetang}
\end{equation}

In order to do that we use conformal transformations. Varying the sector is
equivalent to varying a conformal family of metrics with singular conformal factors on a fixed domain. 

\subsubsection{Conformal transformation from one sector to another} \label{sect:confsectors}
Let $(r, \phi)$ denote polar coordinates on the sector $S_\gamma$.  We assume that the radii of all sectors are equal to one. Let $\alpha \in (0,\pi)$ be the angle at which we shall compute the derivative and $Q = S_\beta$ be a sector with opening angle $\beta \leq \alpha$. We use $(\rho, \theta)$ to denote polar coordinates on $Q$.

Consider the map
\begin{equation}
\Psi_\gamma  : Q \to S_{\gamma}, \quad (\rho, \theta) \mapsto \left( \rho^{\gamma/\beta}, \frac{\gamma \theta}{\beta}  \right)
= (r, \phi) \label{eq:cts} \end{equation}
The pull-back metric with respect to $\Psi_\gamma$ of the Euclidean metric $g$ on $S_{\gamma}$ is
\begin{eqnarray} h_{\gamma} := \Psi_\gamma ^* g &=& \left( \frac{\gamma}{\beta} \right)^2  \rho^{2 \gamma/\beta -2}
\left( d\rho^2 + \rho^2 d\theta^2 \right) = e^{ 2 \sigma_{\gamma}} \left( d\rho^2 + \rho^2 d\theta^2 \right) , \label{eq:pbmet}\\
\sigma_{\gamma} (\rho, \theta) &=& \log \left( \frac{\gamma}{\beta} \rho^{\gamma/\beta - 1} \right)
= \log\left( \frac{\gamma}{\beta} \right) + \left(\frac{\gamma}{\beta}-1\right) \log \rho
\label{eq:conf-factor}\end{eqnarray}
We will consider the family of metrics
$$\{h_\gamma, \gamma \in [\beta, \pi)\}$$
defined by (\ref{eq:pbmet}) on the fixed sector $Q = S_\beta$.

The area element on $Q$ with respect to the metric $h_\gamma$ is
\begin{equation} \label{Ahg} dA_{h_\gamma} = e^{2\sigma_\gamma} \rho d\rho d\theta = e^{2\sigma_\gamma} dA_g, \end{equation}
and the Laplace operator $\Delta_{h_\gamma}$ associated to the metric $h_\gamma$ is formally given by
\begin{equation} \label{Dhg} \Delta_{h_\gamma} = 
-\left( \frac{\beta}{\gamma} \right)^2  \rho^{-2 \gamma/\beta +2}
\left( \pa_\rho ^2 + \rho^{-1} \pa_\rho + \rho^{-2} \pa_\theta ^2 \right) =
e^{-2\sigma_\gamma} \Delta,  \end{equation}
where $\Delta := \Delta_{\beta}=  -\pa_\rho ^2 -\rho^{-1} \pa_\rho - \rho^{-2} \pa_\theta ^2 $ \ is the Laplacian on $(Q,g)$.

The transformation $\Psi_{\gamma}$ induces a map between the function spaces
$$ \Psi_{\gamma}^* : C_{c}^{\infty}(S_{\gamma}) \to C_{c}^{\infty}(Q), \quad f\mapsto \Psi_{\gamma}^*f := f\circ \Psi_{\gamma}.$$

\begin{prop} \label{cor-s2} For $\gamma\geq \beta$, the map $\Psi_{\gamma}^*$ is an isometry between the Friedrichs domain of $\Delta_{h_\gamma}$ on $Q$ 
and the domain of the Friedrichs extension of $\Delta_{\gamma}$ on the sector $S_{\gamma}$. Moreover, 
$$\Psi_{\gamma}^* ( \Dom(\Delta_{\gamma})) = \Dom(\Delta_{h_\gamma}) = H^2 (Q, h_\gamma) \cap H^{1}_0 (Q,h_{\gamma}),$$
with $\Delta_{h_\gamma}= e^{-2\sigma_{\gamma}}\Delta_{\beta}$.
\end{prop}

This proposition is a direct consequence of the following two Lemmas.

\begin{lemma} The map $\Psi_{\gamma}$ defined by equation (\ref{eq:cts}) is an isometry $\Psi_{\gamma}^*$ between the Sobolev spaces
$H^{1}_0 (Q,h_{\gamma})$ and $H^{1}_0 (S_{\gamma},g_{\gamma})$. 
\label{lemma:H1hg}
\end{lemma}

\begin{proof}
As before, let $r,\phi$ denote the coordinates in $S_\gamma$, and let $\rho, \theta$ denote the coordinates in $Q$.
The volume element in $Q$ and the Laplacian for the metric $h_{\gamma}$ are given in (\ref{Ahg}) and (\ref{Dhg}), respectively.

The transformation $\Psi_\gamma^*$ extends to the $L^2$ spaces.
The fact that $\Psi_\gamma^*$ is an isometry between $L^{2}(S_\gamma,g)$ and $L^{2}(Q,h_\gamma)$ follows from a standard change of variables computation.  
For $f:S_\gamma\to \R$, we compute that the $L^2$ norms of $f \in L^{2}(S_\gamma,g)$ and $\Psi_\gamma^* f$ on $L^{2}(Q,h_\gamma)$ coincide: 
$$\int_{S_\gamma} |f(r,\phi)|^2 rdrd\phi = \int_{Q} |f\circ \Psi_\gamma|^2 \left(\frac{\gamma}{\beta}\right)^2 \
\rho^{2\frac{\gamma}{\beta}-1} d\rho d\theta = \int_{Q} |\Psi_\gamma^*f|^2 e^{2 \sigma_{\gamma}} \rho d\rho d\theta.$$

Next let $f\in H^{1}_0 (S_\gamma,g)$.  To prove that $\Psi_\gamma^*f \in H^{1}_0 (Q,h_\gamma)$ we show that the $L^2$-norms $\Vert df \Vert_{L^2(S_\gamma,g)}$ and  $\Vert df \circ d\Psi_\gamma \Vert_{L^2(Q,h_\gamma)}$ are identical.  Since $|df|^{2}_g = \vert \nabla_g f \vert^{2} = g^{lj}(\partial_l f)(\partial_j f)$, 
$$\int_{S_\gamma} \vert \nabla_g f \vert^{2} dA_g = \int_Q \left( \left(
\left(\frac{\partial f}{\partial r}\right)^2 + \frac{1}{r^2}\left(\frac{\partial f}{\partial \phi}\right)^2\right) \circ \Psi_\gamma (\rho, \theta)\right) e^{2\sigma_{\gamma}} \rho d\rho d\theta.$$

Using $\Psi_\gamma^*f=f\circ \Psi_\gamma (\rho, \theta)$ we have 
$$\frac{\partial f}{\partial r}(\Psi_\gamma(\rho,\theta)) = \frac{\beta}{\gamma} \rho^{1-\gamma/\beta}
\frac{\partial \Psi_\gamma^*f}{\partial \rho}, \quad \frac{\partial f}{\partial \phi}(\Psi_\gamma(\rho,\theta))
= \frac{\beta}{\gamma} \frac{\partial \Psi_\gamma^*f}{\partial \theta}.$$
Substituting above, we obtain

\begin{eqnarray*}\int_{S_\gamma} \vert \nabla_g f \vert^{2} dA_g
&=&  \int_Q \left(
\left( \frac{\beta}{\gamma} \rho^{1-\gamma/\beta} \frac{\partial \Psi_\gamma^*f}{\partial \rho}\right)^2
 + \rho^{-2\gamma/\beta} \left(\frac{\beta}{\gamma} \frac{\partial \Psi_\gamma^*f}{\partial \theta}\right)^2
\right) e^{2\sigma_{\gamma}}\rho d\rho d\theta \\
&=&  \int_Q \left(\frac{\beta}{\gamma} \rho^{1-\gamma/\beta}\right)^2 \left(
\left(\frac{\partial \Psi_\gamma^*f}{\partial \rho}\right)^2
 + \frac{1}{\rho^2}\left(\frac{\beta}{\gamma} \frac{\partial \Psi_\gamma^*f}{\partial \theta}\right)^2
\right) e^{2\sigma_{\gamma}}\rho d\rho d\theta \\
&=& \int_Q
e^{-2\sigma_{\gamma}}\left(\left(\frac{\partial \Psi_\gamma^*f}{\partial \rho}\right)^2 + \frac{1}{\rho^2}
\left(\frac{\partial \Psi_\gamma^*f}{\partial \theta}\right)^2\right)
e^{2 \sigma_{\gamma}} \rho d\rho d\theta\\
&=& \int_{Q}  \vert \nabla_{h_\gamma} \Psi^*f \vert^{2} dA_{h_\gamma}.
\end{eqnarray*}
This completes the proof.
\end{proof}


\begin{lemma} \label{le:3} The map $\Psi_{\gamma}^*$ is an isometry between the Sobolev spaces
$H^{2}(Q,h_{\gamma})$ and $H^{2}(S_{\gamma},g )$. 
A function $f \in H^{2}(Q,h_{\gamma})$ if
and only if $\Psi^*f \in H^{2}(S_\gamma,g)$.
\end{lemma}

\begin{proof}
Let $f \in H^{2}(Q,h_\gamma)$. By definition $\Psi_\gamma^*f = (f\circ \Psi_\gamma) (\rho, \theta)$, so 
$$|\Delta_{h_\gamma} \Psi_\gamma^*f|^2 = \left(\frac{\beta}{\gamma}\right)^2 \rho^{-\frac{4\gamma}{\beta}+4}
\left(\frac{\partial^2 \Psi_\gamma^*f}{\partial \rho^2} + \frac{1}{\rho}\frac{\partial \Psi_\gamma^*f}{\partial \rho}
+ \frac{1}{\rho^2} \frac{\pa^2\Psi_\gamma^*f}{\pa \theta^2}\right)^2.$$
Since
$$\frac{\partial^2 \Psi_\gamma^*f}{\partial \rho^2}
= \left( \frac{\gamma}{\beta} \right)^2 \rho^{2 \frac{\gamma}{\beta}-2} \frac{\partial^2 f}{\partial r^2} (\Psi_\gamma(\rho, \theta)) +
\frac{\gamma}{\beta} \left(\frac{\gamma}{\beta} -1\right) \rho^{\frac{\gamma}{\beta}-2} \frac{\partial f}{\partial r} (\Psi_\gamma(\rho, \theta)), $$
it is easy to see that
$$\int_{Q}|\Delta_{h_\gamma} \Psi_\gamma^*f|^2 dA_{h_\gamma} = \int_{Q} (|\Delta_g f|^2 \circ \Psi_\gamma)(\rho,\theta) e^{2\sigma_{\gamma}} dA_g
= \int_{S_\gamma} |\Delta_g f|^2 dA_g $$
where the last equality follows from the standard change of variables, and $g$ denotes the Euclidean metric on both $Q$ and
$S_\gamma$.
\end{proof}

\begin{example} Let $\gamma\in [\beta,\pi)$, and $h_\gamma$ be as above. Let $\varphi(\rho,\theta):= \rho^x \sin(k\pi \theta/\beta)$. It is easy to see that  \begin{itemize} \item $\varphi \in L^{2}(Q,h_\gamma) \Leftrightarrow  x> -\gamma/\beta$ \item $\varphi \in H^{1}(Q,h_\gamma) \Leftrightarrow x>0$\item $\varphi \in H^{2}(Q,h_\gamma) \Leftrightarrow x > \frac{\gamma}{\beta}$.\end{itemize} \end{example}

The example above shows that the domain of the Laplacian depends on the angle, and in particular, it will be different for different angles.  As a consequence several problems appear here that distinguish this case from the classical smooth case and force us to go into the details of the differentiation process.

\subsubsection{Domains of the Laplace operators} 
Even though the description of the domains of the family of Laplace operators $\{\Delta_{h_\gamma}, \gamma\geq \beta\}$ given in the previous section is useful for our purposes, it is not enough. 
Unlike the smooth case, this family do not act on a single fixed Hilbert space when $\gamma$ varies but instead we will demonstrate below that they act on a nested family of weighted, so-called ``b''-Sobolev spaces.

\begin{defn} The $b$-vector fields on $(S_\gamma,g)$, denoted by $\calV_b$, are the $\cC^\infty$ span of the vector fields 
$$\calV_b := \cC^{\infty} \, \textrm{ span of } \, \{r\del_r, \del_\phi\},$$
where $\cC^{\infty}$ means that the coefficient functions are smooth up to the boundary.
For $m \in \N$, the $b$-Sobolev space is defined as
$$H^m _b := \left\{ f \ |\ V_1 \ldots V_j f \in L^2(S_\gamma,g) \, \forall j \leq m, \, \forall \, V_1, \ldots, V_j \in \calV_b \right\},$$
and $H^0 _b = L^2(S,g)$.  The weighted $b$-Sobolev spaces are 
$$r^x H^m _b = \{ f \ |\ \exists v \in  H^m _b, \quad f = r^x v \}.$$
\end{defn}

We first apply results due to several authors, including but not limited to, Mazzeo \cite{etdeo1} Theorem 7.14 and Lesch \cite{les}  Proposition 1.3.11.  

\begin{prop} \label{prop:domdesc2} The Friedrichs domain of the Laplace operator $\Delta_{\gamma}$ on the sector $S_\gamma$ with Dirichlet boundary condition is
$$\Dom(\Delta_{\gamma}) = r^2 H^2 _b \cap H^1 _0(S_\gamma,g).$$ 
\end{prop} 

\begin{proof} 
By equation (19) in \cite{mm} and Theorem 7.14 \cite{etdeo1} (c.f. \cite{les} Proposition 1.3.11), any element in the domain of the Friedrichs extension of Laplacian $\Delta_{\gamma}$ has a partial expansion near $r=0$ of the form
$$\sum_{\gamma_j \in ]-n/2, -n/2+2]} c_j r^{\gamma_j} \psi_j (\phi) + w, \quad w \in r^2 H^2 _b.$$
In our case the dimension $n=2$, and the indicial roots $\gamma_j$ are given by
$$\gamma_j = \pm \sqrt{\mu_j},$$
where $\mu_j$ is an eigenvalue of the Laplacian on the link of the singularity, and $\psi_j$ is the eigenfunction with eigenvalue $\mu_j$.  The link is in this case $[0, \gamma]$ with Dirichlet boundary condition.  These eigenvalues are therefore $\mu_j = \frac{j^2 \pi^2}{\gamma^2}$ with $j \in \N$, $j\geq 1$.  In particular, there are \em no \em indicial roots in the critical interval $]-1, 1]$, because $\gamma < \pi$.  Taking into account the Dirichlet boundary condition away from the singularity, it follows that the domain of the Laplace operator is precisely given by 
$$r^2 H^2 _b (S_\gamma) \cap H^1 _0 (S_\gamma,g).$$
\end{proof} 

The operators $\Delta_{h_\gamma}$, albeit each defined on functions on $Q$, have domains which are defined in terms of $L^2(Q, dA_{h_\gamma})$.  In particular, the area forms depend on $\gamma$.  Consequently, in order to fix a single Hilbert space on which our operators act, we use the following maps
\begin{eqnarray}
\Phi_\gamma &:& L^2(Q, dA_{h_\gamma}) \to L^2 (Q, dA), \quad f \mapsto e^{\sigma_\gamma} f = \frac{\gamma}{\beta} \rho^{\gamma/\beta -1} f; \label{eq:t2domains}\\
\Phi_\gamma ^{-1} &:& L^2 (Q, dA) \to L^2 (Q, dA_{h_\gamma}), \quad f \mapsto e^{-\sigma_\gamma} f = \frac{\beta}{\gamma}\rho^{-\gamma/\beta + 1}f. \notag
\end{eqnarray}
Each $\Phi_\gamma$ is an isometry of $L^2(Q, dA_{h_\gamma})$ and $L^2 (Q, dA)$, since
$$\int_Q f^2 dA_{h_\gamma} = \int_Q f^2 e^{2\sigma_\gamma} dA = \int_Q (\Phi_\gamma f)^2 dA.$$

\begin{prop} \label{prop-nest} For all $\gamma \in [\beta, \pi)$, we have
$$\Phi_\gamma \left(\Dom (\Delta_{h_\gamma}) \right) \subseteq \rho^{2\gamma/\beta} H^2 _b (Q, dA) \cap H^1 _0 (Q, dA).$$
Moreover, 
$$\Phi_\gamma \left( \Dom (\Delta_{h_\gamma}) \right) \subset \Phi_{\gamma'} \left( \Dom (\Delta_{h_{\gamma'}})\right), \quad \gamma' < \gamma.$$
\end{prop} 

\begin{proof} 
Let us start by comparing the $H^2 _b$ spaces.  To do this, we first compute 
$$r = \rho^{\gamma/\beta} \implies \rho \partial_\rho = \frac{\gamma}{\beta} r \partial_r; \, \partial_\theta = \frac{\gamma}{\beta} \partial_\phi \implies \cC^\infty \langle \rho \pa_\rho , \pa_\theta \rangle = \cC^\infty \langle r \pa_r, \pa_\phi \rangle.$$
Now, let $f\in r^2 H^2 _b (S_\gamma)$, so by definition $f(r,\phi)=r^2 u(r,\phi)$ with $u\in H^2 _b (S_\gamma)$. Then
$$(\Psi_\gamma ^* f)(\rho,\theta) = f(\rho^{\gamma/\beta},\gamma \theta/\beta) = \rho^{2\gamma/\beta}(\Psi_\gamma ^*u)(\rho,\theta).$$
Consequently,  
$$
\Psi_\gamma ^* (H^2 _b (S_\gamma) ) = H^2 _b (Q, dA_{h_\gamma}) =  \rho^{-\gamma/\beta +1} H^2 _b(Q ,dA),$$
$$\Psi_{\gamma}^* (r^2 H^2_b (S_{\gamma}) ) = \rho^{2 \gamma/\beta} H^2 _b (Q, dA_{h_\gamma}) = \rho^{\gamma/\beta +1} H^2 _b(Q ,dA),$$ 
and 
\begin{eqnarray*}
\Phi_\gamma (\Psi_\gamma ^* (r^2 H^2 _b (S_\gamma) ) &=& \Phi_\gamma(\rho^{\gamma/\beta + 1} H^2 _b (Q, dA))\\
&=&  \rho^{2\gamma/\beta} H^2 _b (Q, dA) \subseteq \rho^2 H^2 _b (Q, dA),
\end{eqnarray*}
for $\gamma\in [\beta,\pi)$.  Moreover, we have 
$$\Dom(\Delta_{h_\gamma}) = \rho^{2\gamma/\beta} H^2 _b (Q, dA) \cap H_0 ^1 (Q, dA).$$
It is straightforward to see that 
$$ \gamma' < \gamma \implies \rho^{2\gamma/\beta} H^2 _b (Q, dA) \subset \rho^{2\gamma'/\beta} H^2 _b (Q, dA).$$

Now, we claim that
$$\Phi_\gamma \left(H^1 _0 (Q, dA_{h_\gamma})\cap \rho^{2 \gamma/\beta} H^2 _b (Q, dA_{h_\gamma})\right) \subseteq H^1 _0 (Q, dA).$$
Note that $\cC^\infty _0 (Q)$ is independent of $h_\gamma$.  
Then, it is enough to show that for any $f\in \Dom(\Delta_{h_\gamma})$ the $L^2(Q, dA)$-norms of $\Phi_{\gamma} f$ and $\nabla (\Phi_{\gamma} f)$,  can be estimated using the fact that $f \in H^1 _0 (Q, dA_{h_\gamma})\cap \rho^{\gamma/\beta +1} H^2 _b (Q, dA)$.  
By definition, $\Phi_\gamma$ is an isometry of $L^2 (Q, dA_{h_\gamma} )$ and $L^2 (Q, dA)$. So we only need to prove that 
$\nabla (\Phi_{\gamma} f)\in L^2(Q, dA)$.
To do this, we compute
$$\int_Q |\nabla_{h_\gamma} f|^2 dA_{h_\gamma} = \int_Q e^{-2\sigma_{\gamma}}
\left( (\pa_\rho f)^2  + \rho^{-2} (\pa_\theta f)^2 \right) e^{2\sigma_\gamma}dA
= \int_Q |\nabla f|^2 dA.$$
Next we compute
\begin{multline*}
\int_{Q} |\nabla \Phi_{\gamma} f|^2 dA = \int_{Q} \left( (\pa_\rho e^{\sigma_\gamma} f)^2  + \rho^{-2} (\pa_\theta e^{\sigma_\gamma} f)^2  \right) dA\\
= \int_{Q} \left\{e^{2\sigma_\gamma} \left((\pa_\rho  f)^2  + \rho^{-2} (\pa_\theta f)^2 \right) + (\pa_{\rho }e^{\sigma_\gamma})^2f^2 +
2(\pa_{\rho} e^{\sigma_\gamma})e^{\sigma_\gamma} f (\pa_{\rho}f)\right\} dA.
\end{multline*}

The first term,
$$\int_{Q} e^{2\sigma_\gamma} \left((\pa_\rho  f)^2  + \rho^{-2} (\pa_\theta f)^2 \right) dA
= \int_{Q} |\nabla f|^2 \rho^{2\frac{\gamma}{\beta}-2}\frac{\gamma^2}{\beta^2} dA \leq 
\frac{\gamma^2}{\beta^2} \int_{Q} |\nabla_{h_{\gamma}} f|^2 dA_{h_\gamma}$$
since $\frac{\gamma}{\beta} \geq 1$, $\rho^{2\frac{\gamma}{\beta}-2}\leq 1$ on $Q$.

To estimate the second term, we use that $f\in \rho^{\gamma/\beta + 1} H^2_b(Q, dA)$, therefore %
$$\int_{Q} (\pa_{\rho }e^{\sigma_\gamma})^2f^2 dA = c \int_{Q} f^2 \rho^{2\frac{\gamma}{\beta} -4} dA
\leq \int_{Q} f^2 \rho^{-\frac{\gamma}{\beta} -1} dA < \infty$$
where $c=\frac{\gamma^2}{\beta^2} \frac{(\gamma-\beta)^2}{\beta^2}$ and we have used again that $\gamma \geq \beta$.
For the third term we compute
\begin{multline*}
\int_Q (\pa_{\rho} e^{\sigma_\gamma})e^{\sigma_\gamma} f (\pa_{\rho}f) dA 
= c \int_Q  \rho^{2\frac{\gamma}{\beta} -3} f (\pa_{\rho}f) dA \\
\leq c \left(\int_Q f^2 \rho^{2\frac{\gamma}{\beta} -4} dA \right)^{1/2} 
\left(\int_Q (\rho \pa_{\rho}f)^2 \rho^{2\frac{\gamma}{\beta}-4} dA \right)^{1/2}. 
\end{multline*}
Since $f\in \rho^{\gamma/\beta + 1} H^2_b(Q, dA)$, write $f = \rho^{\frac{\gamma}{\beta}+1}u$ with $u\in H^2_b(Q, dA)$.  Then
$$\int_Q f^2 \rho^{2\frac{\gamma}{\beta} -4} dA = \int_Q u^2 \rho^{\frac{2\gamma}{\beta}+2}\rho^{2\frac{\gamma}{\beta} -4} dA < \infty$$
since $\gamma \geq \beta$, and $u \in H^2 _b (Q, dA) \subset L^2 (Q, dA)$. 

Now, for the integral $\int_Q (\rho \pa_{\rho}f)^2 \rho^{2\frac{\gamma}{\beta}-4} dA$ we compute

$$(\rho \pa_{\rho}f)^2 
= \big(\frac{\gamma}{\beta}+1\big)^2\rho^{\frac{2\gamma}{\beta}+2} u^2 + 
2(\frac{\gamma}{\beta}+1)\rho^{2\frac{\gamma}{\beta}+2} u (\rho \pa_{\rho}) u
+\rho^{\frac{2\gamma}{\beta}+2} ((\rho \pa_{\rho}) u)^2.$$
Since $u\in H^2_b(Q, dA)$ and $\gamma \geq \beta$
$$ \int_Q u^2 \rho^{4\frac{\gamma}{\beta}-2} dA<\infty, \ \ 
\text{ and } \ \
\int_Q ((\rho \pa_{\rho}) u)^2\rho^{4\frac{\gamma}{\beta}-2} dA <\infty.$$ 
By the Cauchy-Schwarz inequality, 
\begin{multline*}\int_Q u \big((\rho \pa_{\rho}) u\big) \rho^{4\frac{\gamma}{\beta}-2} dA \leq  \left(\int_Q u^2 \rho^{4\frac{\gamma}{\beta} -2} dA \right)^{1/2} 
\left(\int_Q (\rho \pa_{\rho}u)^2 \rho^{4\frac{\gamma}{\beta}-2} dA \right)^{1/2} < \infty.
\end{multline*}

Putting everything together, we have proven that 
$$\Phi_\gamma (\Psi_\gamma ^* (\Dom (\Delta_{\gamma}) ) ) \subseteq \rho^{2\gamma/\beta} H^2 _b (Q, dA) \cap H^1 _0 (Q, dA).$$

In order to see that for $\beta\leq \gamma'<\gamma<\pi$,  
$$\Phi_\gamma( \Psi_\gamma ^* (\Dom(\Delta_{\gamma}))) \subset \Phi_{\gamma'} ( \Psi_{\gamma'} ^* (\Dom(\Delta_{\gamma'})),$$
we first note that 
$$\Phi_\gamma( \Psi_\gamma ^* (\Dom(\Delta_{\gamma}))) \subset \rho^{2\gamma/\beta} H^2 _b (Q, dA) \subset \rho^{2\gamma'/\beta} H^2 _b (Q, dA).$$
Finally, in order to show that 
$$f \in H^1 _0 (Q, dA_{h_\gamma})\cap \rho^{\gamma/\beta +1} H^2 _b (Q, dA) \implies \Phi_{\gamma'} ^{-1} \Phi_{\gamma} f \in H^1_0 (Q, dA_{h_{\gamma'}}), \quad \gamma' < \gamma,$$ 
simply note that the $L^2$ norm of $\nabla_{h_{\gamma'}}  (\Phi_{\gamma'} ^{-1} \Phi_{\gamma} f )$ can be estimated in the same way as above using the fact that $\gamma' < \gamma$, and therefore $\gamma-\gamma' > 0$.   
\end{proof} 


\subsubsection{The family of operators} 
Finally, let us introduce the family of operators that we will use to prove Polyakov's formula. Let us define $H_\gamma$ as
\begin{equation} H_\gamma := \Phi_\gamma \circ \Psi_\gamma  \circ \Delta_\gamma \circ \Psi_\gamma^{-1} \circ \Phi_\gamma^{-1}
= \Phi_\gamma  \circ \Delta_{h_\gamma} \circ \Phi_\gamma ^{-1}. \label{eq:transopH}\end{equation}

The domains of the family $\{H_\gamma\}_\gamma$ nest
$$\beta \leq \gamma' \leq \gamma \implies \Dom(H_\gamma) \subset \Dom(H_{\gamma'}) \subset \Dom(\Delta)$$
where $\Delta$ is the Laplacian on $Q$.


\section{Short time asymptotic expansion} \label{ss:stae} 
In order to prove the trace class property of the operator ${\mathcal M}_{\left( 1 + \log(r) \right)}e^{-t\Delta_{\alpha}}$ on $S_{\alpha}$ and the trace class property of the operators appearing in the proof of Proposition \ref{prop:dtHo} in \S \ref{s:formula} below, we need 
estimates on the heat kernel. We do not need a sharp estimate; a general estimate in terms of the time variable is enough for our purposes.  
 
\subsection{Heat kernel estimates} \label{hkest} 
The heat kernel estimates we require follow rather quickly from \cite{davies} and \cite{acm}.

\begin{prop} \label{pr-heat2}
Let $S$ denote a finite Euclidean sector. Then the heat kernel of the Dirichlet extension of Laplacian on $S$ satisfies the following estimates
\begin{eqnarray*}
\left| H (t,z,z') \right| &\leq& \frac{C}{t},\\
\left| \pa_t H (t,z,z') \right| &\leq& \frac{C}{t^2},
\end{eqnarray*}
for all $z,z'\in S$, and $t \in (0, T)$, where $C>0$ is a fixed constant which depends only on the constant $T>0$.  
\end{prop}

\begin{proof}
Sectors are both rather mild examples of stratified spaces. Consequently, the heat kernel satisfies the estimate (2.1) on p. 1062 of \cite{acm}.  This estimate is 
\begin{equation} \label{hke1} H(t, z, z') \leq C t^{-1}, \quad \forall z, z'\in S, \quad \forall t \in (0, 1), \end{equation} 
since the dimension $n=2$.   

Next, we apply the results by E.B. Davies in \cite{davies} which hold for the Laplacian on a general Riemannian manifold whose balls are compact if the radius is sufficiently small. These minimal hypotheses are satisfied for sectors.  
By \cite[Lemma 1]{davies}, 
$$|H(t, z, z')|^2 \leq H(t, z, z) H(t, z', z'),$$
for all $z,z'\in S$, and all $t>0$.  If $T <1$, then this estimate together with (\ref{hke1}) gives the first estimate in the Proposition.  In general, by \cite{davies} the function $t \mapsto H(t, z, z)$ is positive, monotone decreasing in $t$, and log convex for every $z$.  For a fixed $T \geq 1$, the estimate (\ref{hke1}) together with the above shows that 
$$|H(t, z, z')|^2 \leq C^2 \quad \forall t \geq 1.$$
So, we simply replace the constant $C$ with the constant $CT$, which we again denote by $C$ and obtain the estimate 
$$|H(t, z, z')|^2 \leq C^2 t^{-2}, \quad \forall t \in (0, T), \quad \forall z \textrm{ and } z' \in S.$$

Next, we apply Theorem 3 of \cite{davies}, which states that the time derivatives of the heat kernel satisfy the estimates
$$\left| \frac{\pa^n}{\pa t^n} H(t,z,z') \right| \leq \frac{n!}{(t-s)^n} H(s,z,z)^{1/2} H(s,z',z')^{1/2}, \quad n \in \N, \quad 0<s<t.$$
Making the special choice $s = t/2$ and $n=1$, we have
$$\left| \pa_t H(t,z,z') \right| \leq \frac{2}{t} H(t/2,z,z)^{1/2} H(t/2,z',z')^{1/2}.$$
Using the estimates for the heat kernel we estimate the right side above which shows that
$$\left| \pa_t H(t,z,z') \right| \leq C t^{-2}, \quad \forall t \in (0,T), \quad \forall z, z' \in  S.$$
\end{proof} 

\begin{remark}  By the heat equation, the estimate for the time derivative of the heat kernel implies the following estimate for the Laplacian of the heat kernel
$$\left| \Delta H(t,z,z') \right| \leq  C t^{-2},$$
for any $0<t<T$, and $z, z' \in  S$, for a constant $C>0$ depending on $T$.
\end{remark}

We now return to the trace class property of the operators in question.

\begin{lemma} \label{l-hsch} 
Let $S$ denote the finite sector with angle $\alpha$ and radius $R$, $S=S_{\alpha,R}$, with $\alpha \in (0, \pi)$. Let $\Delta$ denote the Dirichlet Laplacian on $S$ and $e^{-t\Delta}$ be the corresponding heat operator.  Let $\cM_\psi$ denote the operator multiplication by a function $\psi$.  Let $\xi$ be a smooth function on $S\setminus \{\rho=0\}$ such that $\xi(\rho) = m \log(\rho)$ for a constant $m\in \R$ on some neighborhood of the singular point $\rho=0$. Then, for any $t>0$ the following operators
\begin{enumerate}
\item $\cM_{\xi} e^{-t \Delta}$
\item $\cM_{\xi} \Delta e^{-t \Delta}$
\item $\Delta \cM_{\xi} e^{-t \Delta}$
\item $\cM_\psi e^{-t\Delta}$, where $\psi(\rho,\theta)=O(\rho^{-c})$ as $\rho \to 0$, for $c<1$.
\end{enumerate}
are Hilbert-Schmidt. Moreover, the operators $\cM_{\xi} e^{-t \Delta}$, $\cM_{\xi}\Delta e^{-t \Delta}$, $\Delta \cM_{\xi} e^{-t \Delta}$, $\cM_{\psi} e^{-t \Delta}$, and $\cM_{\psi}\Delta e^{-t \Delta}$ are trace class. 
\label{lemma:HSp}
\end{lemma}

\begin{proof}
Recall that an integral operator is Hilbert-Schmidt if the $L^2$-norm of its integral
kernel is finite. Using the estimates given in Proposition \ref{pr-heat2} we have that
\begin{eqnarray*}\|\cM_{\psi} e^{-t\Delta}\|_2 &\leq& C \int_{S\times S} |\psi(z)|^2 |H (t,z, z')|^2
dA dA' \\
&\leq& \wt{C}(\alpha,R,t) \int_0 ^R \int_0 ^R \rho^{-2c+1} \rho' d\rho d\rho' <\infty
\end{eqnarray*}
since $c<1$. Hence $\cM_\psi e^{-t\Delta}$ is a Hilbert-Schmidt operator. Similarly, 
\begin{eqnarray*}\|\cM_{\xi} e^{-t\Delta}\|_{2} &\leq& C \int_{S\times S} |\log(\rho)|^{2} |H (t,z, z')|^2
dA dA' \\
&\leq& \wt{C}(\alpha,R,t) \int_0 ^R \int_0 ^R |\log(\rho)|^2 \rho \rho' d\rho d\rho' <\infty,
\end{eqnarray*}
since $|\log(\rho)|^2 \rho$ is bounded on $(0, R)$. Thus $\cM_{\xi}  e^{-t \Delta}$ is also Hilbert-Schmidt.
Using the estimates for the kernel of $\Delta e^{-t \Delta}$, we can prove in the same way as above that $\cM_{\xi}\Delta e^{-t \Delta}$ and 
$\cM_{\psi}\Delta e^{-t \Delta}$ are Hilbert-Schmidt.

We shall prove now that $\Delta \cM_{\xi} e^{-t \Delta/2}$ is Hilbert-Schmidt.The integral kernel of $\Delta \cM_{\xi} e^{-t \Delta/2}$ is $\Delta_{z} \big(\xi(z) H(t,z, z')\big)$. By Leibniz's rule, 
$$\Delta_{z} \big(\xi(z) H(t,z, z')\big) = \big(\Delta_{z} \xi(z)\big) H(t,z, z')  + \xi(z)\big(\Delta_{z}  H(t,z, z')\big) + 2 \langle \nabla_z \xi , \nabla_z H\rangle.$$
When considering the integral
$$\int_{S\times S} |\Delta_{z} \big(\xi(z) H(t,z, z')\big)|^2 \ dA(z)dA(z'), $$
using again the estimates on the heat kernel and that the function $\xi$ is smooth away from the singularity, it is clear that the corresponding  terms are all bounded. Near the singularity, for $0< \rho \leq \rho_0$,  $\xi(z) = \log(\rho)$, and $\Delta \log(\rho) = 0$.  Hence, near the singularity, we have
$$\Delta_{z} \big(\xi(z) H(t,z, z')\big) = \xi(z)\big(\Delta_{z}  H(t,z, z')\big) + 2 \rho^{-1} \pa_{\rho} H(t,\rho, \rho', \theta, \theta').$$
The first term corresponds to the operator $\cM_{\xi} \Delta e^{-t \Delta}$ that is Hilbert-Schmidt.  Considering the second term, we note that, for any $t>0$, the heat kernel is in the domain of the Laplace operator.  By Proposition \ref{prop:domdesc2} (c.f. Example 1), this requires that the heat kernel 
$H\in H^2_b(S_{\alpha},\rho d\rho d\theta)$ which implies that 
$\rho^{-1} \pa_{\rho}H(t,\rho,\rho',\theta,\theta')\in L^2(S_{\alpha},\rho d\rho d\theta)$. Thus
\begin{equation*}
\int_{S_{\alpha,\rho_0}\times S} |\rho^{-1} \pa_{\rho}H(t,\rho,\rho',\theta,\theta')|^2 \rho  d\rho d\theta \rho' d\rho' d\theta' 
\leq C(t,\alpha),
\end{equation*} 
where $S_{\alpha,\rho_0}$ denotes the sector with angle $\alpha$ and radius $\rho_0$ and $C(t,\alpha)$ is a constant that depends on $\alpha$ and $t$. Hence, the operator whose integral kernel is $2 \langle \nabla_z \xi , \nabla_z H\rangle$ is Hilbert-Schmidt. 
Since the sum of two Hilbert-Schmidt operator is Hilbert-Schmidt, it follows that $\Delta \cM_{\xi} e^{-t \Delta/2}$ is Hilbert-Schmidt. \\

A way to prove that an operator is trace class is to write it as a product of two Hilbert-Schmidt operators. Since $e^{-t \Delta}$ 
is trace class, in particular it is Hilbert Schmidt. Therefore using the semigroup property of the heat operator we write
$$\cM_{\xi} e^{-t \Delta} = \cM_{\xi} e^{-t \Delta/2} e^{-t \Delta/2}$$
which proves that $\cM_{\xi} e^{-t \Delta}$ is trace class. The trace class property of the other operators listed in this lemma follows in the same way. 
\end{proof}


\subsection{Heat kernel parametrix} \label{ss:hkp}
To prove the existence of the asymptotic expansion of the trace given by equation (\ref{exp-exists}) and to compute it, we replace the heat kernel by a parametrix.  We construct a parametrix for the whole domain in the standard way: first we partition the domain and use the heat kernel of a suitable model for each part, then we combine these using cut-off functions. 
We use the following models for each corresponding part of the domain: 
\begin{enumerate}
\item The heat kernel for the infinite sector with opening angle $\alpha$ for a small neighborhood, $\cN_\alpha$, of the vertex of the sector with opening angle $\alpha$.  Denote this heat kernel by $H_\alpha$.  We note that by \cite[Lemma 6]{vs}, we may use the heat kernel for the infinite sector on this neighborhood.
\item The heat kernel for $\R^2$ for a neighborhood $\cN_i$ of the interior away from the straight edges. Denote this heat kernel by $H_{i}$.
\item The heat kernel for the half-plane, $\R^2_+$, for neighborhoods $\cN_e$ of the straight edges away from the corners. Denote this heat kernel by $H_{e}$.
\item The heat kernel for the unit disk for a small neighborhood, $\cN_a$, of the curved arc away from the corners. Denote this heat kernel by $H_\D$ or by $H_a$ (this is done in order to simplify some equations in the proof).  
\item The curved arc meets the straight segments in two corners. For these corners we consider two disjoint neighborhoods that are denoted by $\cN_{c}$, at these corners we use the heat kernel of the upper half unit disk, $H_{\D_+}$ or by $H_c$ (again, this is done in order to simplify some of the equations), at the corner $(1,0)$.  
\end{enumerate}

Let $*$ represent any of the regions introduced above. We define the gluing functions as cut-off functions $\{\chi_\alpha, \chi_i, \chi_{e}, \chi_a, \chi_{c} \}$ and $\{\wt{\chi}_\alpha, \wt{\chi}_i, \wt{\chi}_{e}, \wt{\chi}_a, \wt{\chi}_{c}\}$. These are smooth functions chosen such that $\{\chi_\alpha, \chi_i, \chi_e, \chi_a, \chi_{c}\}$ form a partition of unity of $S_\alpha$, $\chi_* =1$ on $\cN_*$, and 
$\wt{\chi}_* =1$ on $\supp(\chi_*)$.

Therefore, the parametrix we use is 
\begin{multline} H_p(t,z,z') = 
 \wt{\chi}_\alpha(z) H_\alpha \chi_\alpha (z') + \wt{\chi}_a(z) H_\D \chi_a (z')\\ + \wt{\chi}_{c} (z) H_{\D_+} \chi_{c} (z') + \wt{\chi}_e(z) H_{e} \chi_e (z') + \wt{\chi}_i(z) H_{i} \chi_i (z'). \label{parametrix} \end{multline} 
Above, for the sake of brevity, we have suppressed the argument $(t, z, z')$ of the four model heat kernels.  

The salient point, which is well-known to experts, is that this patchwork para\-me\-trix restricted to the diagonal is asymptotically equal to the true heat kernel on the diagonal with an error of $O(t^{\infty})$ as $t \downarrow 0$.  For these arguments, we refer the reader to Lemma 2.2 of \cite{mr} and \S 4 and Lemma 4.1 of \cite{Aldana-ae}.  
Moreover, it is known that for domains with both corners and curved boundary, the heat trace admits an asymptotic expansion as $t \downarrow 0$, and that this trace has an extra purely local contribution from the angles at the corners.  The proof for domains with \em both \em corners \em and \em curved boundary can be found in \cite[Theorem 2.1]{corners}; see also \cite{mr}. Even though we expect this calculation to be contained in earlier literature we were unfortunately unable to locate it.  Therefore, it is natural to expect that the angles also appear in the variational formula for the determinant.  We shall see that this is indeed the case.

\subsection{Proof of Theorem \ref{th-exp-exists}} \label{ss:pthexpexists}
For a sector, $S_\alpha$, from equation (2.13) of \cite{corners} (cf also \cite{mr}) it follows that the short time asymptotic expansion of the heat trace is given by
$$ \Tr(e^{-t\Delta_\alpha}) = 
\frac{\alpha}{8\pi t} - \frac{\alpha}{8 \sqrt{\pi t}} +  \frac{1}{12} \left( 2\chi(S_\alpha) - 3 \right) + \frac{\pi^2 + \alpha^2}{24 \pi \alpha} + 2 \frac{\pi^2 + \pi^2/4}{24 \pi (\pi/2)}+ O(\sqrt{t}),
$$ 
where $3$ is the number of corners, and the term $2 \frac{\pi^2 + \pi^2/4}{24 \pi (\pi/2)}$ comes from the two corners where the circular arcs meet the straight edges at which the angle is $\pi/2$. The $t^0$ coefficient (also called the constant coefficient) in the short time asymptotic of the heat trace is also $\zeta_{\Delta_{\alpha}}(0)$ : 

\begin{equation} \label{t0-heattrace} \zeta_{\Delta_{\alpha}}(0) = \frac{1}{12} \left( 2\chi(S_\alpha) - 3 \right) + \frac{\pi^2 + \alpha^2}{24 \pi \alpha} + 2 \frac{\pi^2 + \pi^2/4}{24 \pi (\pi/2)} = \frac{\pi^2 + \alpha^2}{24 \pi \alpha} + \frac{1}{8}. \end{equation}

Consequently, it suffices to demonstrate that 
$$\int_{S_\alpha} \log(r) H_{S_\alpha} (t,r,\phi, r, \phi) r dr d\phi$$
admits an expansion as in (\ref{exp-exists}), as $t \downarrow 0$.  

Let the error $E(t, r, \phi, r', \phi')$ be the difference between the true heat kernel and the patchwork construction, 
$$E(t,r,\phi, r', \phi') := H_{S_\alpha} (t,r, \phi, r', \phi') - H_p (t, r, \phi, r', \phi').$$
Then, we have 
$$\left| \int_{S_\alpha} \log(r) E(t, r, \phi, r, \phi) r dr d\phi \right| = O(t^{\infty}), \quad t \downarrow 0,$$
because the model heat kernels decay as $O(t^{\infty})$ as $t \downarrow 0$ in any compact set away from the diagonal.  

Consequently, it suffices to prove that 
$$\int_{S_\alpha} \log(r) H_p (t,r,\phi, r, \phi) r dr d\phi$$
admits a short time asymptotic expansion as in Theorem \ref{th-exp-exists}.  By definition of $H_p$, to demonstrate this, we may proceed locally, by considering the model heat kernels on their respective neighborhoods.  First, note that on $S_\alpha \setminus \cN_\alpha$, $\log(r)$ is a smooth function. 

Therefore, the existence of an asymptotic expansion of the integral 
\begin{equation}\int_{S_\alpha\setminus (\cN_\alpha \cup \cN_{c})} \log(r) H_p (t,r,\phi, r, \phi) r dr d\phi \label{eq:tlswoc}\end{equation}
for small values of $t$ follows from the locality principle of the heat kernel and the exis\-tence of the expansions of the heat kernel of the corresponding models. Although the idea is standard, we briefly explain it.  
\begin{multline*}
\frac{2}{\alpha} \int_{S_\alpha\setminus (\cN_\alpha \cup \cN_{c})}  \log(r)H_p (t,r,\phi, r, \phi)\\ 
=\frac{2}{\alpha} \int_{S_\alpha\setminus (\supp(\chi_\alpha) \cup \supp(\chi_c))}  \log(r) (\chi_i H_i+\chi_e H_e+\chi_a H_\D) \ dA \\
+\frac{2}{\alpha} \int_{\left(\supp(\chi_\alpha)\setminus\cN_\alpha\right) \cup \left(\supp(\chi_c)\setminus \cN_{c}\right)} 
\log(r) \sum_{*\in\{\alpha,i,e,a,c\}} \chi_* H_* \ dA,
\end{multline*}
where $dA$ denotes the area element $r dr d\phi$. Using the existence of the expansion of the heat kernel for small times in the interior and the smooth boundary away from the corners, we have that the asymptotic expansion of the integral exists. In addition, we can compute the constant coefficient of the expansion of the trace using the expansion of the heat kernels. This is: 
\begin{multline*} 
\frac{2}{\alpha} \int_{S_\alpha\setminus (\supp(\chi_\alpha) \cup \supp(\chi_c)}  \log(r) (\chi_i H_i+\chi_e H_e+\chi_a H_\D) \ dA\\
= \frac{2}{\alpha} \frac{1}{4\pi t} \int_{S_\alpha\setminus (\supp(\chi_\alpha) \cup \supp(\chi_c)}  \log(r) \left(\chi_i + \chi_e +\chi_a\right) dA\\\ 
+ \frac{2}{\alpha} \frac{1}{8 \sqrt{\pi t}} \int_{\partial(S_\alpha)\setminus \partial(\supp(\chi_\alpha) \cup \supp(\chi_c)) }  \log(r) \left(\chi_i + \chi_e +\chi_a\right) \ ds\\ 
+ \frac{2}{\alpha} \frac{1}{24\pi} \int_{S_\alpha\setminus (\supp(\chi_\alpha) \cup \supp(\chi_c)}  \log(r) \left(\chi_i + \chi_e +\chi_a\right) \text{Scal}_g \ dA\\ + 
\frac{2}{\alpha}\frac{1}{12 \pi} \int_{\partial(S_\alpha)\setminus \partial(\supp(\chi_\alpha) \cup \supp(\chi_c)) }  \log(r) \left(\chi_i + \chi_e +\chi_a\right)\kappa_g \ ds + O(t^{1/2}).
\end{multline*}

Observing that the scalar curvature is zero, the logarithm vanishes on the boundary of $S_\alpha$ where $r=1$, and the geodesic curvature of the straight edges is zero, we have that constant terms vanish:
\begin{multline*}\frac{2}{\alpha} \frac{1}{24\pi} \int_{S_\alpha\setminus (\supp(\chi_\alpha) \cup \supp(\chi_c)}  \log(r) \left(\chi_i + \chi_e +\chi_a\right) \text{Scal}_g \ dA\\ + 
\frac{2}{\alpha}\frac{1}{12 \pi} \int_{\partial(S_\alpha)\setminus \partial(\supp(\chi_\alpha) \cup \supp(\chi_c)) }  \log(r) \left(\chi_i + \chi_e +\chi_a\right)\kappa_g \ ds = 0.\end{multline*}

For the integral 
$$\frac{2}{\alpha} \int_{\left(\supp(\chi_\alpha)\setminus\cN_\alpha\right) \cup \left(\supp(\chi_c)\setminus \cN_{c}\right)} 
\log(r) \sum_{*\in\{\alpha,i,e,a,c\}} \chi_* H_* \ dA,$$
we note that in both cases the points in $\supp(\chi_\alpha)\setminus\cN_\alpha$ and $\supp(\chi_c)\setminus \cN_{c}$ are either interior points or points in the smooth boundary of $S_\alpha$. It follows then from the locality principle of the heat kernels, that this case is the same case as above. Therefore 
there exists an asymptotic expansion of the integral given in (\ref{eq:tlswoc}) for small values of time.  Moreover, this expansion does not contain $\log(t)$ terms, and its constant term vanishes.

The existence of the asymptotic expansion of the integral over $\cN_\alpha$ is proven in \S \ref{Scarslaw}. In that section we compute as well the contributions of this integral to the coefficients $a_{2,0}$, and $a_{2,1}$, defined in equation (\ref{exp-exists}). 

Unlike the neighborhood $\cN_\alpha$, there is no ``purely local'' contribution from the other two corners in the sector, apart from the contribution due to the short time expansion of the heat trace given in (\ref{t0-heattrace}).   
In order to prove this, we need to consider the heat kernel of the unit half disk; let $H_{\D_+}$ denote this heat kernel, with the Dirichlet boundary condition. Let $H_\D$ denote the heat kernel for the unit disk with Dirichlet boundary condition.  Using the method of images, the heat kernel for the half disk can be written in terms of the heat kernel for the unit disk as follows: 
\begin{equation}
H_{\D_+} (r, \theta, r', \theta', t) = H_\D (r, \theta, r', \theta', t) - H_\D (r, \theta, r', -\theta', t).
\label{eq:hkhd-it-hkd}
\end{equation}
We will use the fact that the unit disk is a manifold with boundary to prove that these corners do not contribute to our formula.  To accomplish this, we need to consider the associated heat space for the unit disk, in the sense of \cite[Chapter 7]{tapsit}.

The heat space for the disk can be constructed following \cite{MaVer} \S 3.1.  We shall see that the polyhomogeneity of the heat kernel on this space follows from Theorem 1.2 of \cite{MaVer}.  This may not be immediately apparent, because in \cite{MaVer}, the authors consider compact manifolds with edges. A compact manifold with boundary is a particular case of a compact manifold with edges in which the fiber of the cone is a point, $F=\{p\}$, and the lower dimensional stratrum is the boundary, $B= \pa M$.  For more details in this simplified case we also refer to \cite{grieser} and \cite{tapsit}.

\subsubsection{The heat space}  
The heat space associated to the unit disk in $\R^2$ is a manifold with corners obtained by performing two parabolic blow-ups of submanifolds of $\D \times \D \times \R^+$.  Let 
$$\cD_0 := \{ (p, p, 0) \in \D \times \D \times \R^+, \  p \in \D \}.$$
In order to construct the heat space we need to first perform parabolic blow up of 
$$\cD_b:=\cD_0 \cap (\pa \D \times \pa \D \times \R^+).$$ 
The notation for this blown-up space is 
$$[\D\times\D \times \R^+ ; \cD_b, dt].$$
The notation $dt$ indicates that the blowup is parabolic in the direction of the conormal bundle, $dt$.  In chapter 7 of \cite{tapsit} (see also \cite{etdeo1}), it is shown that there is a unique minimal differential structure with respect to which smooth functions on $\D^2 \times \R^+$ and parabolic polar coordinates around $\cD_b$ are smooth in the space $[\D \times \D \times \R^+ ; \cD_b, dt]$.  We recall that the parabolic polar coordinates around $\cD_b$ are $R = \sqrt{s^2 + (s')^2 + t}$ and $\Theta = (t/R^2, s/R, s'/R)$ on $\D^2 \times \R^+$, where $s$ and $s'$ are boundary defining functions for $\pa \D$ in each copy of $\D$.  As a set, this space is equivalently given by the disjoint union  
$$[\D^2 \times \R^+; \cD_b, dt] = \left( (\D^2 \times \R^+) \setminus \cD_b \right) \sqcup (PN^+(\cD_b)/\R^+),$$
where $PN^+(\cD_b)/\R^+$ the interior parabolic normal bundle of $\cD_b$ in $\D^2 \times \R^+$.  This can also be defined using equivalence classes of curves in analogue to the $b$-blowup in the $b$-heat space of \cite{tapsit} Chapter 7; specifically see p. 274--275 of \cite{tapsit}.  For a schematic diagram of the first blow-up, we refer to Figure 2 of \cite{MaVer}.  

Next, the diagonal away from the boundary is blown up at $t=0$.  We note that although the heat space is itself unchanged under the order of blowing up (see Proposition 3.13 of \cite{etdeo1}), the heat kernel is sensitive to which order the blow up is performed (see exercise 7.19 of \cite{tapsit}). In the notation of Melrose (see \S 4 and \S 7 of \cite{tapsit}), the heat space is then 
$$\D_h ^2 := [\D \times \D \times [0, \infty) ; \cD_b, dt; \cD_0, dt ].$$
Specifically, let $\cD_1$ denote the lift of $\cD_0$ to the intermediate space, $[\D^2 \times \R^+; \cD_b, dt]$.  The second step is to blow up $[\D^2 \times \R^+; \cD_b, dt] $ along $\cD_1$, parabolically in the $t$ direction.  As a set, this space is given by the disjoint union  
$$[\D^2 \times \R^+; \cD_b, dt; \cD_0, dt] = \left( [\D^2 \times \R^+; \cD_b, dt]  \setminus \cD_1 \right) \sqcup (PN^+(\cD_1)/\R^+),$$
where $PN^+(\cD_1)/\R^+$ the interior parabolic normal bundle of $\cD_1$ in $[\D^2 \times \R^+; \cD_b, dt]$.  This can also be defined using equivalence classes of curves in analogue to the $b$-blowup in the $b$-heat space of \cite{tapsit} Chapter 7, as explained above.

The heat space is a manifold with corners which has five codimension one boundary hypersurfaces, also known as boundary faces.  For a schematic diagram of this heat space space, we refer to Figure 3 of \cite{MaVer}.  The left and right boundary faces, $\cL$ and $\cR$ are given by the lifts to $\D_h ^2$ of $\pa\D \times \D \times [0, \infty)$ and $\D \times \pa \D \times [0, \infty)$, respectively.  The remaining three boundary faces are at the lift of $\{t=0\}$.  Denote by $\cB$ the face created by blowing up $\cD_b$, and by $\cD$ the face created by blowing up $\cD_0$.   Let $\beta : \D_h ^2 \to \D \times \D \times [0, \infty)$ denote the blow-down map.  Then the last boundary face, the temporal boundary\footnote{In the terminology of \cite{MaVer}, $\cB$ is known as the front face, ff, $\cD$ is known as the temporal diagonal, td, and $\cT$ is known as the temporal face, tf.}  denoted by $\cT$ is given by the closure of 
$$\beta^{-1} ( \D \times \D \times \{0\}) \setminus \left( \cB \cup \cD\right).$$
We denote the boundary defining functions correspondingly by $\rho_{\cL}$, $\rho_{\cR}$, $\rho_{\cB}$, $\rho_{\cD}$, and $\rho_{\cT}$.  Then we note that $t$ lifts to $\D_h^2$ as 
$$\beta^*(t) = \rho_{\cT} \rho_{\cB} ^2 \rho_{\cD} ^2.$$

\subsubsection{Polyhomogeneous conormal distributions on manifolds with corners}  
The heat space is a \em manifold with corners.  \em  An important class of distributions on manifolds with corners is the class of polyhomogeneous conormal distributions, which we abbreviate as \em pc distributions.  \em  We recall how these are defined in general.  Let $X$ be an $n$-dimensional manifold with corners.  By definition (see \S 2A of \cite{etdeo1}), $X$ is locally modelled diffeomorphically near each point by a neighborhood of the origin in the product $(\R^+)^k \times \R^{n-k}$.  Here by locally modelled we mean analogous to the definition of an $n$-dimensional Riemannian manifold being locally modelled by neighborhoods of $\R^n$.  Let $\{M_i\}_{i=1} ^J$ denote the codimension one boundary faces, which we simply refer to as boundary faces.  Let $\cV_b$ be the space of smooth vector fields on $X$ which are tangent to all boundary faces.  

For a point $q \in \pa X$ contained in a corner of maximal codimension $k$, choose coordinates $x^1, \ldots, x^k, y$ near $q$, where $x^i$ are defining functions for the boundary hypersurfaces $M_{i^1}, \ldots , M_{i^k}$ intersecting the corner at $q$, and $y$ is a set of coordinates along this codimension $k$ corner.  Then $\cV_b$ is in this context spanned over $\cC^\infty(X)$ near $q$ by $\{ x^1 \pa_{x^1}, \ldots , x^k \pa_{x^k}, \pa_{y^\alpha} \}$.  The conormal space is 
$$\cA^0 (X) = \{ u : V_1 \ldots V_l u \in L^\infty(X), \, \forall V_i \in \cV_b, \, \textrm{ and } \forall l \}.$$
To motivate the notion of polyhomogeneity, consider first the case in which there is only boundary face, $\pa X$, defined by $x$.  Then we say that $u$ is polyhomogeneous if $u$ admits an expansion 
$$u \sim \sum_{\Re s_j \to \infty} \sum_{p=0} ^{p_j} x^{s_j} (\log x)^p a_{j,p} (x, y), \quad a_{j,p} \in \cC^\infty(X).$$
Here the first index is over $\{s_j \}_{j \in \N} \subset \C$ whereas the second sum is over a finite set (for each $j$) of non-negative integers.  When $X$ has many possibly intersecting codimension one boundary components, then a polyhomogeneous conormal distribution is required to have such expansions at the interior of each boundary face with product type expansions at the corners.  To be more precise, beginning with the highest codimension corners, which have no boundary, one demands the existence of such an expansion, and then one proceeds inductively to the lower codimension corners and finally to the boundary faces.  

\begin{lemma} \label{le:hkdiskpc} 
The heat kernel, $H_\D$, lifted to $\D_h ^2$ is a polyhomogeneous conormal distribution. 
\end{lemma} 

\begin{proof} 
The polyhomogeneity and conormality of $\beta^* (H_\D)$ both follow Theorem 1.2 of \cite{MaVer}.  Specifically, as noted above, the unit disk is an example of an edge manifold, and in this case, the heat kernel with Dirichlet boundary condition is the Friedrichs heat kernel.  
\end{proof}     

Recall equation (\ref{eq:hkhd-it-hkd}) where the heat kernel for the upper half disk is given by the method of images.  We define the involution $f:\D\times \D \times [0, \infty) \to \D\times \D \times [0, \infty)$ by 
$$f(r, \theta, r', \theta', t) = (r, \theta, r', -\theta', t).$$
Then, the reflected term is simply $H_\D \circ f$.  Moreover, we note that $f^2$ is the identity map, and thus $f=f^{-1}$.  Let us denote 
$$\cD_0 ' = \{ (r, \theta, r, -\theta, 0) : (r, \theta) \in \D \} \subset \D \times \D \times [0, \infty),$$
and we observe that 
$$\cD_0 ' = f(\cD_0).$$
Then, it follows immediately from Lemma \ref{le:hkdiskpc} that $H_\D \circ f$ lifts to a polyhomogeneous conormal distribution on 
$$[ \D \times \D \times [0, \infty) ; \cD_0 ' \cap \pa \D \times \pa \D, dt; \cD_0 ' , dt].$$
We therefore immediately obtain 

\begin{cor} \label{cor:hkhdiskpc} 
Let 
\begin{multline*}\wt \D_h ^2 :=\\ 
[\D \times \D \times [0,\infty) ; \cD_0 \cap \cD_0' \cap \pa \D^2, dt; \cD_0 \cap \pa \D^2, dt; \cD_0' \times \pa \D^2, dt; \cD_0 \cap \cD_0', dt; \cD_0, dt; \cD_0', dt],
\end{multline*}
where $\pa \D^2$ denotes $\pa \D \times \pa \D$, and we have slightly abused the notation by not including the time variable when it is clear from the context. 
Then, the function 
$$H_\D - H_\D \circ f$$
lifts to $\wt \D_h ^2$ to a polyhomogeneous conormal distribution. Moreover, the product, 
$$\log(r) \left( H_\D - H_\D \circ f \right)$$
also lifts to $\wt \D_h ^2$ to a polyhomogeneous conormal distribution.  
\end{cor} 

\begin{proof} 
By the preceding lemma, $H_\D$ lifts to be polyhomogeneous conormal on $\D_h ^2$ and therefore also on $\wt \D_h ^2$.  In particular, performing additional blowups does not introduce any problems for $H_\D$.  By the observation that $f^2$ is the identity map, and $f(\cD_0) = \cD_0 '$, the same argument shows that $H_\D \circ f$ also lifts to be polyhomogeneous conormal.  The function $\log(r)$ is already polyhomogeneous conormal on $\D \times \D \times [0, \infty)$, and thus it remains polyhomogeneous conormal when lifted to $\wt \D_h ^2$. 
\end{proof}

\begin{lemma}\label{le:othercorners} Let $\cN_{c}$ denote the union of two neighborhoods of radius $\eps<1/3$ about the corners in $S_\alpha$ where the circular arcs meet the straight edges.  Then the trace, 
$$ \int_{\cN_{c}} \log(r) H_{\D_{+}} (t, r, \phi, r, \phi) r dr d\phi$$
has an asymptotic expansion as $t \downarrow 0$ which contains only integer and half-integer powers of $t$, and no $\log(t)$ terms.  Let $\alpha(\epsilon)$ denote the coefficient of $t^0$ in this expansion.  Then 
$$\lim_{\epsilon \to 0} \alpha(\epsilon) =  0.$$ 
\end{lemma} 

\begin{proof} 
By symmetry, it suffices to compute the trace near the point $(1,0)$.  
The heat kernel for the upper half disk can be written as
$$H_{\D_+} = H_\D - H_\D \circ f$$
By Corollary \ref{cor:hkhdiskpc}, the product 
$$\log(r) \left( H_\D - H_\D \circ f\right)$$
lifts to a polyhomogeneous conormal distribution on $\wt \D_h ^2$.  We compute the lift of 
$$r = 1-(1-r) = 1-s,$$
is given by 
$$\beta^*(r) = 1 - \beta^*(s) = 1 - \rho_\cL \rho_{\cB}.$$
Then, $\log(r) = \log(1-(1-r))$, and so we compute its lift
$$\beta^*( \log(r)) = \beta^*(\log(1-(1-r))) = \log(1-\rho_\cB \rho_{\cL}). $$
This is a smooth function near $\rho_{\cB}$ and $\rho_{\cL}$ and admits an asymptotic expansion there, 
$$\log(1-\rho_\cB \rho_{\cL})= \sum_{k \geq 1} - \frac{(\rho_{\cB} \rho_{\cL})^k}{k}, \textrm{ near $\cL$ and $\cB$.}$$
We know from \cite{MaVer} that the lifts of $H_\D$ and $H_\D \circ f$ to $\wt \D_h ^2$ contain integer and half-integer powers of the boundary defining functions, but they do not contain any log terms.  Hence, blowing down, or equivalently computing the trace near the lift of the point $(1,0)$, by the pushforward theorem there is an expansion as $t \downarrow 0$ which contains only integer and half-integer powers of $t$, and in particular, no $log(t)$ terms.  As a consequence, only the coefficient of $t^0$ may enter into our Polyakov formula, hence it is the only coefficient of interest to us.  We estimate this coefficient.  

Let $\cN_{\eps}$ be the intersection of $S_{\alpha}$ with a disc of radius $\eps$ centered at $(1,0)$. We then use the existence of the asymptotic expansion to write 
$$\int_{\cN_{\eps}} \log(r) H_{\D_+} (t, r, \phi, r, \phi) r dr d\phi  \sim \alpha (\epsilon) t^0 + R(\epsilon, t), \quad t \downarrow 0.$$
Note that 
$$|| \log(r)||_\infty = O(\eps) \textrm{ for all points} (r, \theta) \in \cN_\eps.$$
Hence, we estimate 
$$\left| \int_{\cN_\eps} \log(r) H_{\D_+} (t, r, \phi, r, \phi) r dr d\phi \right| \leq O(\epsilon) \int_{\cN_\eps} H_{\D_+} (t, r, \phi, r, \phi) r dr d\phi.$$ 
Now, on the right we have the asymptotic expansion of $H_{\D_+}$ near this corner,  
$$\int_{\cN_\eps} H_{\D_+} (t, r, \phi, r, \phi) r dr d\phi $$
$$\sim \frac{|\cN_\eps|}{4\pi t} - \frac{|\pa \cN_\eps \cap \pa \D_+|}{8 \sqrt{\pi t} } +\frac{|\pa \cN_\eps \cap \pa \D|}{12 \pi} + \frac{\pi^2 - (\pi/2)^2}{12 \pi^2} + O(\sqrt{t}), \quad t \downarrow 0.$$
Above, $|\cN_\eps|$, $|\pa \cN_\eps \cap \pa \D_+|$, $|\pa \cN_\eps \cap \pa \D|$ denote area and perimeters, respectively.  We note that the curvature along the boundary is one, and the angle at which the circular arc meets the straight edge is $\pi/2$.  These two observations lead to the computation above of the $t^0$ term.  Consequently, we have the estimate, 
$$\left| \int_{\cN_\eps} \log(r) H_{\D_+} (t, r, \phi, r, \phi) r dr d\phi \right| $$
$$\leq O(\epsilon) \left(  \frac{|\cN_\eps|}{4\pi t} - \frac{|\pa \cN_\eps \cap \pa \D_+|}{8 \sqrt{\pi t} } + \frac{|\pa \cN_\eps \cap \pa \D|}{12 \pi} + \frac{\pi^2 - (\pi/2)^2}{12 \pi^2} + O(\sqrt{t}) \right), \quad t \downarrow 0.$$

Letting $\epsilon \downarrow 0$, for any $t>0$, the right side vanishes.  Moreover, letting $\epsilon = t$, then as $t=\epsilon \downarrow 0$, the right side also vanishes.  This requires the coefficient, $\alpha(\epsilon)$, to vanish as $\epsilon \downarrow 0$, because the term $\alpha(\epsilon) t^0$ is independent of $t$.  

Finally, we note that a similar argument cannot be applied to the corner at the origin in the original sector, that is the corner of opening angle, $\alpha$, at which the conformal factor has a logarithmic singularity.  First and foremost, we cannot bring out the $L^\infty$ norm of the log there. 
\end{proof}



\section{Variational Polyakov formula}  \label{s:formula}    
Let $A$ be an integral operator on $L^2(Q,h_{\gamma})$ with kernel $K_{A}(z,z')$. The transformed operator $\Phi_\gamma  A \Phi_\gamma ^{-1}$ to the Hilbert space $L^2(Q, g)$ by the conformal transformation $\Phi_\gamma f = e^{\sigma_\gamma} f$ has integral kernel
$e^{\sigma_\gamma(z)}K_{A}(z,z')e^{\sigma_\gamma(z')}$. This follows from the transformation of the area element 
and 
\begin{eqnarray*}
(\Phi_\gamma A \Phi_\gamma ^{-1} f)(z) &=& \Phi_\gamma   \left(\int_Q K_A (z,z') e^{-\sigma_{\gamma}(z')} f(z') dA_{h_{\gamma}}(z') \right)\\
&=&  e^{\sigma_\gamma(z)}  \int_Q K_A (z,z') e^{-\sigma_\gamma(z')} f(z') e^{2\sigma_{\gamma}(z')} dA\\
&=& \int_Q e^{\sigma_\gamma(z)} K_A (z,z') e^{\sigma_{\gamma}(z')} f(z') dA (z')
\end{eqnarray*}
for $f\in L^2(Q,g)$.

Thus
\begin{eqnarray*}
\tr_{L^2(Q,g)} \left( \Phi_\gamma A \Phi_\gamma ^{-1} \right)
&=& \int_Q K_A (z,z) e^{2\sigma_{\gamma}(z)} dA (z)\\
&=& \int_Q K_A (z,z) dA_{h_{\gamma}}(z) =\tr_{L^2(Q,h_{\gamma})}\left( A\right).
\end{eqnarray*}

\subsection{Differentiation of the operators}
As we saw in equation (\ref{eq:transopH}), the domains of the family $\{H_\gamma\}_\gamma$ nest. 
In order to compute the derivative with respect to the angle at $\gamma= \alpha$, one would like to apply both $H_\gamma$ and $H_\alpha$ to the elements in the domain of $H_\alpha$.  There are subtleties which arise, but we can remedy them. 

\begin{lemma} \label{beta-choice} Let $0< \beta \leq \alpha<\pi$, and $\beta\leq \gamma<\pi$. Then the following one-sided derivatives 
$$\left.\frac{d H_\gamma}{d\gamma^-}\right|_{\gamma=\alpha} \text{ for } \beta < \alpha, \quad \text{ and } \quad 
\left.\frac{d H_\gamma}{d\gamma^+}\right|_{\gamma=\alpha} \text{ for } \beta \leq \alpha$$
are well defined. 
In both cases we have 
\begin{equation} \frac{\partial H_{\gamma}}{\partial \gamma ^\pm} = \dot{H_\gamma} =
\left(\frac{\pa \sigma_{\gamma}}{\pa \gamma}\right) H_{\gamma} + \Phi_\gamma
\left(\frac{\pa \Delta_{h_\gamma}}{\pa \gamma} \right) \Phi_\gamma ^{-1}  - \Phi_\gamma  \Delta_{h_\gamma}
\left(\frac{\pa \sigma_{\gamma}}{\pa \gamma}\right) \Phi_\gamma ^{-1}.  \label{dHa1} \end{equation}
\end{lemma}

\begin{proof}
The formal expression for $\dot{H_\gamma}$ follows from a straightforward computation. For the left derivative, we have that $\gamma, \ \beta < \alpha$.  Since $\Dom(H_\alpha) \subset \Dom(H_\gamma)$ for each $\gamma < \alpha$, we can apply both the operators $H_\alpha$ and $H_\gamma$ to all elements of the domain of $H_\alpha$ and let $\gamma \uparrow \alpha$.  The derivative 
$\left.\frac{d H_\gamma}{d\gamma^-}\right|_{\gamma=\alpha}$ is therefore computed in this way and given by (\ref{dHa1}).  We can then let $\beta \uparrow \alpha$.  

For the right derivative $\gamma > \alpha$, and we let $\beta := \alpha$.  In this case we cannot apply both operators $H_\gamma$ and $H_\alpha$ to all elements of $\Dom(H_\alpha)$ because there might be functions $f\in \Dom(H_\alpha)\setminus \Dom(H_{\gamma})$.  However, for such a function there is a sequence $\{f_n\}_n$ in $C^{\infty}_0(Q,g)$ with $f_n \to f$ in $\Dom(H_\alpha)$, since smooth and compactly supported functions are dense in the domain of the operator.  
Then, for $f\in \Dom(H_\alpha) \setminus \Dom(H_\gamma)$ we define 
\begin{equation} \label{rightderiv} \left.\frac{d H_\gamma}{d\gamma^+}\right|_{\gamma=\alpha} f :=  \lim_{n \to \infty } \left. \frac{d H_\gamma}{d\gamma^+}\right|_{\gamma=\alpha} f_n \end{equation} 
and we shall see that this limit is well defined.  For any $n \in \N$  
\begin{eqnarray*} 
\left.\frac{d H_\gamma}{d\gamma^+}\right|_{\gamma=\alpha} f_n &=& 
\frac{1}{\alpha} \left(1 + \log(\rho)\right) \Delta_{\alpha}f_n + -2 \frac{1}{\alpha} \left(1 + \log(\rho)\right) \Delta_{\alpha} f_n - \Delta_{\alpha}
\left(\frac{1}{\alpha} \left(1 + \log(\rho)\right)f_n \right)\\
&=& 
-\frac{1}{\alpha} \left(1 + \log(\rho)\right) \Delta_{\alpha} f_n
-\Big((\frac{1}{\alpha} \left(1 + \log(\rho)\right)\Delta_{\alpha}f_n \\
& & - 2g(\nabla_\alpha (\frac{1}{\alpha} \left(1 + \log(\rho)\right), \nabla_\alpha f_n) + \frac{1}{\alpha}{\mathcal M}_{\Delta_{\alpha} (1 + \log(\rho))} f_n \Big)\\
&=& -\frac{2}{\alpha} \left(1 + \log(\rho)\right) \Delta_{\alpha}f_n + \frac{2}{\alpha} \rho^{-1}\pa_{\rho}f_n . 
\end{eqnarray*}
The expression simplifies as above upon the observation that $\Delta_\alpha (1+\log(\rho)) = 0$.  Since $\alpha=\beta$, $\Dom(H_\alpha) = \Dom(\Delta) =H^{1}_{0}(Q,dA)\cap \rho^2 H^2 _b (Q, dA)$, and by assumption $f_ n \to f$ in $\Dom(H_\alpha)$.  Consequently, 
\begin{equation}
\Delta_\alpha f_n \to \Delta_\alpha f \textrm{ and }  \rho^{-1} \pa_\rho f_n \to \rho^{-1} \pa_\rho f, \textrm{ in } L^2 (Q,g).
\label{eq:cdLaux}
\end{equation}
By the Cauchy-Schwarz inequality, 
$$\int_Q \left| (\log \rho) (\Delta_\alpha f_n - \Delta_\alpha f) \right| dA \leq || \log(\rho) ||_{L^2 (Q, dA)} || \Delta_\alpha f_n - \Delta_\alpha f ||_{L^2(Q, dA)},$$
which tends to $0$ as $n \to \infty$ by the assumption that $f_n \to f$ in $\Dom(H_\alpha)$.  We therefore have the $L^1(Q, g)$ convergence 
$$(\log \rho) \Delta_\alpha f_n \to (\log \rho) \Delta_\alpha f.$$
This convergence together with the $L^2 (Q, g)$ convergence given in equation (\ref{eq:cdLaux}) above (which implies $L^1$ convergence because $Q$ is compact) shows that 
\begin{eqnarray}\lim_{n \to \infty } \left. \frac{d H_\gamma}{d\gamma^+}\right|_{\gamma=\alpha} f_n &=& \lim_{n \to \infty} -\frac{2}{\alpha} \left(1 + \log(\rho)\right) \Delta_{\alpha}f_n + \frac{2}{\alpha} \rho^{-1}\pa_{\rho}f_n \notag \\
&=& -\frac{2}{\alpha} \left(1 + \log(\rho)\right) \Delta_{\alpha}f + \frac{2}{\alpha} \rho^{-1}\pa_{\rho}f.  \label{rightderiv2}
\end{eqnarray} 
The above limit is in $L^1 (Q, g)$ and is well-defined for all $f \in \Dom(H_\alpha)$ because it is independent of the choice of approximating sequence $f_n \in C^\infty _0$.  This shows that we may indeed define the right derivative in (\ref{rightderiv}), and it is equal to (\ref{rightderiv2}).
\end{proof}

\begin{remark}
Although the definitions of $\sigma_\gamma$, $h_\gamma$, $Q$, and $H_\gamma$ depend on the choice of $\beta$, the final variational formula
is independent of this choice since, in the end, everything is pulled back to the original sector $S_\alpha$, and $\beta$ drops out of the equations. We only require this parameter to rigorously differentiate the trace; the sector $Q=S_{\beta}$ and the choice of $\beta$ are part of an auxiliary construction. 
\end{remark}
\begin{prop}\label{prop:dtHo} Let $H_{\gamma}$ be as in equation (\ref{eq:transopH}). Then the derivative of the
transformed heat operators is
\begin{eqnarray*}{\frac{d}{d \gamma}} \tr_{L^{2}(Q,g)}(\Phi_{\gamma} e^{-t\Delta_{h_\gamma}}
\Phi_{\gamma} ^{-1} ) &=& -t \ \tr_{L^{2}(Q,g)} (\dot{H}_{\gamma} e^{-t H_{\gamma}})\\
&=& -t \ \tr_{L^{2}(Q,h_{\gamma})} (\dot{\Delta}_{h_\gamma} e^{-t\Delta_{h_\gamma}}),
\end{eqnarray*}
where $\dot{\Delta}_{h_\gamma} \equiv \left.{\frac{\partial}{\partial \gamma}}\
\Delta_{h_\gamma}\right\vert_{\gamma}= -2(\partial_{\gamma}\sigma_{\gamma})\Delta_{h_\gamma}$.
\end{prop}

\begin{proof}
Although the proof of this proposition is standard in the boundaryless case, we include some details to show that the statement also holds in our case.
Following the same computation as in \cite[Lemma 5.1]{Aldana-ae} and \cite{Mueller},
\begin{equation*}
{\frac{d}{d \gamma}} \tr_{L^{2}(Q,g)}(\Phi_{\gamma} e^{-t\Delta_{h_\gamma}}
\Phi_{\gamma} ^{-1} ) = \tr_{L^{2}(Q,g)} \left( {\frac{d}{d \gamma}} e^{-t
H_{\gamma}}\right). \end{equation*}
Let $\gamma_{2}> \gamma_{1}$.  Duhamel's
principle is well known and often used in the settings of both manifolds with boundaries and conical singularities; see \cite{chee}.
We apply this principle in terms of the operators
$$e^{-t H_{\gamma_{1}}} - e^{-t H_{\gamma_2}} =  \int_{0}^{t} - e^{-s H_{\gamma_{1}}} H_{\gamma_{1}} e^{-(t-s)
H_{\gamma_{2}}} + e^{-s H_{\gamma_{1}}} H_{\gamma_{2}} e^{-(t-s) H_{\gamma_{2}}} \ ds.$$

Notice that the product $H_{\gamma_{1}} e^{-(t-s)H_{\gamma_{2}}}$ is well defined since $e^{-(t-s)H_{\gamma_2}}$ maps $L^2(Q,g)$ onto $\Dom(H_{\gamma_2})$
and $\Dom(H_{\gamma_2}) \subset \Dom(H_{\gamma_1})$. Then for $f\in L^2(Q,g)$, $e^{-(t-s)H_{\gamma_2}} f \in \Dom(H_{\gamma_1})$.

Dividing by $\gamma_{1}-\gamma_{2}$ the previous equation and letting $\gamma_{2}\to \gamma_{1}$, we obtain
$$\left.{\frac{d}{d \gamma}}
\ e^{-t H_{\gamma}}\right\vert_{\gamma=\gamma_{1}} = - \int_{0}^{t} e^{-s
H_{\gamma_{1}}} \left({\left.{\frac{d}{d \gamma}} H_{\gamma}
\right\vert_{\gamma=\gamma_{1}}}\right) e^{-(t-s) H_{\gamma_{1}}}\ ds.$$

Therefore since the heat operators are trace class
\begin{equation}
{\frac{d}{d \gamma}} \tr_{L^{2}(Q,g)}(\Phi_{\gamma} e^{-t\Delta_{h_\gamma}}\Phi_{\gamma} ^{-1} )
= -t \ \tr_{L^{2}(Q,g)} \left(\dot{H}_{\gamma} e^{-tH_{\gamma}}\right).
\end{equation}

We computed $\frac{\partial}{\partial \gamma} H_{\gamma}$ in equation (\ref{dHa1}). Substituting its value into our calculation above, we obtain
\begin{eqnarray*}& & \tr_{L^{2}(Q,g)} \left(\dot{H}_{\gamma} e^{-tH_{\gamma}}\right)\\
& & = \tr_{L^{2}(Q,g)} \left(
(
\left( \pa_{\gamma} \sigma_{\gamma}\right) H_{\gamma} + \Phi_{\gamma}
\left(\pa_{\gamma} \Delta_{h_\gamma}\right) \Phi_{\gamma} ^{-1} - \Phi_{\gamma} \Delta_{h_\gamma}
\left(\pa_{\gamma} \sigma_{\gamma}\right) \Phi_{\gamma} ^{-1}
) e^{-tH_{\gamma}}\right)\\
& & = \tr_{L^{2}(Q,g )} \left(\Phi_{\gamma}
\left( (\pa_{\gamma} \sigma_{\gamma}) \Delta_{h_\gamma}
e^{-t \Delta_{h_\gamma}} + (\pa_{\gamma} \Delta_{h_\gamma}) e^{-t \Delta_{h_\gamma}} -
\Delta_{h_\gamma} (\pa_{\gamma} \sigma_{\gamma}) e^{-t \Delta_{h_\gamma}} \right) \Phi_{\gamma} ^{-1}
\right)\\
& & = \tr_{L^{2}(Q,h_{\gamma})}\left( (\pa_{\gamma}\sigma_{\gamma}) \Delta_{h_\gamma} e^{-t\Delta_{h_\gamma}} +
\dot{\Delta}_{h_\gamma} e^{-t \Delta_{h_\gamma}} - \Delta_{h_\gamma}(\pa_{\gamma}\sigma_{\gamma})e^{-t\Delta_{h_\gamma}}\right)\\
& &= \tr_{L^{2}(Q,h_{\gamma})}\left(\dot{\Delta}_{h_\gamma} e^{-t \Delta_{h_\gamma}}\right),
\end{eqnarray*}
where we have used that the operators 
$\left(\pa_{\gamma} \sigma_{\gamma}\right) H_{\gamma} e^{-tH_{\gamma}}$,
$\Phi_{\gamma}  \left(\pa_{\gamma} \Delta_{h_\gamma}\right) \Phi_{\gamma} ^{-1} e^{-tH_{\gamma}}$, and
$\Phi_{\gamma} \Delta_{h_\gamma} \left(\pa_{\gamma} \sigma_{\gamma}\right) \Phi_{\gamma} ^{-1}  e^{-tH_{\gamma}}$
are trace class in ${L^{2}(Q,g)}$; see Lemma \ref{lemma:HSp}.  Since the operators are all trace class, the first and third terms cancel due to commutation of the operators when taking the trace. 
\end{proof}

\begin{proof}[Proof of Theorem \ref{var-det}]
In order to prove Theorem \ref{var-det}, we differentiate the spectral zeta function with respect to the angle $\gamma$ as in 
equation (\ref{eq:derzetang}). 

We start by noticing the equality of the following traces:
$$\Tr_{L^{2}(S_{\gamma},g)}(e^{-t \Delta_{\gamma}})
= \Tr_{L^{2}(Q,h_{\gamma})}(e^{-t \Delta_{h_{\gamma}}})
= \Tr_{L^{2}(Q,g)} (e^{-t H_{\gamma}}).$$
Then, from Proposition \ref{prop:dtHo} we have
\begin{multline*}
\left. \frac{\partial}{\partial \gamma} \ \Tr_{L^{2}(S_{\gamma},g)}(e^{-t \Delta_{\gamma}})
\right|_{\gamma = \alpha}
= -t \ \tr_{L^{2}(Q,h_\alpha)}\left(\dot{\Delta}_{h_\alpha} e^{-t \Delta_{h_\alpha}}\right)\\
= 2 t\ \Tr_{L^{2}(Q,h_\alpha)}\left( 
\left(\frac{1}{\alpha} + \frac{1}{\beta} \log \rho\right)
\Delta_{h_\alpha} e^{-t \Delta_{h_\alpha}}\right)
\end{multline*}
where we have replaced $(\delta \sigma_{\alpha})$ by its value
$\left(\frac{1}{\alpha} + \frac{1}{\beta} \log \rho\right)$, and we have used that the Laplacian changes conformally in dimension $2$.
On the other hand,
$$\frac{\pa}{\pa t} \Tr_{L^{2}(Q,h_\alpha)}
\left( (\delta \sigma_{\alpha}) e^{-t \Delta_{h_\alpha}}\right)
= - \Tr_{L^{2}(Q,h_\alpha)}\left( (\delta \sigma_{\alpha})
\Delta_{h_\alpha} e^{-t \Delta_{h_\alpha}}\right).$$
The convergence above follows from the invariance of the trace and the estimates contained in \S \ref{hkest}, in particular Lemma \ref{l-hsch}.

Thus 
$$\left. \frac{\partial}{\partial \gamma} \ \Tr_{L^{2}(S_{\gamma},g)}(e^{-t \Delta_{\gamma}})
\right|_{\gamma = \alpha}
= -2 t\ \frac{\pa}{\pa t} \Tr_{L^{2}(Q,h_\alpha)}\Big( (\delta \sigma_{\alpha}) e^{-t \Delta_{h_\alpha}}\Big).$$

Notice that using the change of variables in equation (\ref{eq:cts}) we obtain
$$\Tr_{L^{2}(Q,h_\alpha)}\left(\left(\frac{1}{\alpha} + \frac{1}{\beta} \log \rho\right)
e^{-t \Delta_{h_\alpha}}\right) 
= \Tr_{L^2 (S_{\alpha},g)} \left( \frac{1}{\alpha} \left( 1 + \log(r) \right) e^{- t \Delta_{\alpha}}  \right).$$

Now, going back to the computation of $\delta \zeta_{\Delta_{\alpha}}' (0)$ and replacing the corresponding terms we have
$$\left. \frac{\partial}{\partial \gamma} \zeta_{\Delta_{\gamma}} (s)\right|_{\gamma=\alpha}
= -\frac{2}{\Gamma(s)} \int_0 ^\infty t^{s}
\frac{\pa}{\pa t} \Tr_{L^{2}(S_\alpha,g)}\Big((\delta \sigma_{\alpha}) e^{-t \Delta_{h_\alpha}} \Big)
dt. $$
Recall that upon changing variables, $\delta \sigma_{\alpha}(r,\phi) = \frac{1}{\alpha} \left( 1 + \log(r) \right) $. The next step is to integrate by parts. In order to be able to integrate by parts, we require appropriate estimates of
the trace for large values of $t$ and an asymptotic expansion of it for small values of $t$.

The large values of $t$ are not problematic since
$$\tr_{L^{2}(S_\alpha,g)} \Big( (\delta \sigma_{\alpha}) e^{-t \Delta_{\alpha}} \Big) = O(e^{-c'_{\alpha}t}), \text{ as } t\to \infty,$$
for some constant $c'_{\alpha}>0$.  This statement follows from a standard argument; see for example \cite[Lemma 5.2]{Aldana-ae}. Let $t>1$ and write
\begin{align*}
(\delta\sigma_{\alpha}) e^{-t \Delta_{\alpha}} = (\delta\sigma_{\alpha})
e^{-{\frac{1}{2}} \Delta_{\alpha}} e^{-(t-{\frac{1}{2}}) \Delta_{\alpha}}.
\end{align*}
The operator $(\delta\sigma_{\alpha}) e^{-{\frac{1}{2}} \Delta_{\alpha}}$ is trace class. Since
the spectrum of the operator $\Delta_{\alpha}$ is contained in $ [c_{\alpha},\infty)$ for some $c_{\alpha}>0$, for $t>1$ we have
$$\Vert e^{-(t-{\frac{1}{2}}) \Delta_{\alpha}} \Vert_{L^{2}(S_\alpha,g)}\leq e^{-c_{\alpha}(t-{\frac{1}{2}})}$$
Thus for any $t>1$, the trace satisfies the desired estimate:
\begin{multline*}
\vert \tr((\delta\sigma_{\alpha}) e^{-t \Delta_{\alpha}}) \vert
\leq
\Vert (\delta\sigma_{\alpha}) e^{-{\frac{1}{2}} \Delta_{\alpha}}
e^{-(t-{\frac{1}{2}}) \Delta_{\alpha}}\Vert_{1}\\
\leq \Vert (\delta\sigma_{\alpha}) e^{-{\frac{1}{2}} \Delta_{\alpha}} \Vert_{1} \Vert
e^{-(t-{\frac{1}{2}}) \Delta_{\alpha}}\Vert_{L^{2}(S_\alpha,g)}
\ll e^{-c'_{\alpha}t},
\end{multline*}
where $\Vert \cdot \Vert_{1}$ denotes the trace norm of the operator and $\Vert \cdot \Vert_{L^{2}(S_\alpha,g)}$ denotes the operator norm 
in $L^{2}(S_\alpha,g)$. 

As for the small values of $t$, the existence of an asymptotic expansion is established in Theorem \ref{th-exp-exists}.    Consequently, integration by parts gives 
\begin{equation*} \left. \frac{\partial}{\partial_{\gamma}} \zeta_{\Delta_{\gamma}} (s)\right|_{\gamma=\alpha} = \frac{2 s}{\Gamma(s)}
\int_0 ^{\infty}  t^{s-1} \tr_{L^{2}(S_\alpha,g)} \Big( (\delta \sigma_{\alpha}) e^{-t\Delta_{\alpha}} \Big) dt,
\end{equation*}
Now, we insert the asymptotic expansion for the trace proven in Theorem \ref{th-exp-exists} to obtain
\begin{multline*} \left. \frac{\partial}{\partial_{\gamma}} \zeta_{\Delta_{\gamma}} (s)\right|_{\gamma=\alpha} = \frac{2}{\alpha} \frac{s}{\Gamma(s)}
\int_0 ^{1}  t^{s-1} \Big( a_0 t^{-1} + a_{1}t^{-\frac12} + a_{2, 0} \log(t) + a_{2, 1} + f(t) \Big) dt \\
+ \frac{s}{\Gamma(s)}\int_1 ^{\infty}  t^{s-1} \tr_{L^{2}(S_\alpha,g)} \Big(2 (\delta \sigma_{\alpha}) e^{-t\Delta_{h_\alpha}} \Big) dt, 
\end{multline*}
where $f(t)=O(t^{\frac12})$. 
 Thus 
$$\left. \frac{\partial}{\partial_{\gamma}} \zeta_{\Delta_{\gamma}} (s)\right|_{\gamma=\alpha} = 
\frac{2}{\alpha} \frac{s}{\Gamma(s)} 
 \Big( \frac{a_0}{s-1} + \frac{a_1}{s-\frac12} - \frac{a_{2, 0}}{s^2} + \frac{a_{2, 1}}{s} + e(s) \Big)$$
where $e(s)$ is analytic in $s$ for $\re(s)>-1/2$. The Taylor expansion at $s=0$ of the reciprocal Gamma function $\frac{1}{\Gamma(s)}$ has the form $\frac{1}{\Gamma(s)} = s + \gamma_e s^2 + O(s^3)$ which implies $\frac{s}{\Gamma(s)} = s^2 + \gamma_e s^3 + O(s^4)$. Thus, differentiating  with respect to $s$ and evaluating at $s=0$, we obtain:
$$\left. \frac{\partial}{\partial s} 
 \frac{\partial}{\partial_{\gamma}} \zeta_{\Delta_{\gamma}} (s)\right|_{\gamma=\alpha, s=0}
= \frac{2}{\alpha} \left(-\gamma_{e} a_{2,0} + a_{2,1}\right),$$
where $a_{2,0}$ and $a_{2,1}$ are defined by (\ref{exp-exists}). This finishes the proof of Theorem \ref{var-det}.
\end{proof}


\section{The quarter circle} \label{s:qcrd}
We have proven that the derivative of the logarithm of the determinant of the Laplacian in the angular direction on a finite 
Euclidean sector is given  in terms of the coefficients $a_{2,0}$ and $a_{2,1}$ in the small time expansion in (\ref{variation}) in Theorem \ref{var-det}. To complete the proof of Theorem \ref{th-exp-exists}, we shall simultaneously (1) complete the proof  that this small time expansion exists and (2) compute the contribution from the corner of opening angle $\alpha$.
To motivate and elucidate the arguments used in the rather arduous general case, we first consider the simplest case, when $\alpha = \pi/2$.  

\subsection{Proof of Theorem \ref{ctpio2}} 
Let $\alpha=\pi/2$, then the infinite sector with angle $\alpha$ is the quadrant $C=\{(x,y)\in \R^2 , x,y\geq 0\}$.
The Dirichlet heat kernel in this case can be obtained as the product of the Dirichlet heat kernel on the
half line $[0,\infty)$ with itself. For $x_1,x_2 \in [0,\infty)$ the Dirichlet heat kernel is given by
$$p_{hl}(t,x_1,x_2)=\frac{1}{\sqrt{4\pi t}} (e^{-\frac{(x_1 - x_2)^2}{4t}} - e^{-\frac{(x_1 + x_2)^2}{4t}}).$$

Let $u=(x_1,y_1)$, $v=(x_2,y_2)$ be in $C$, we have
\begin{multline*}
p_{C}(t,u,v)= p_{hl}(t,x_1,x_2) p_{hl}(t,y_1,y_2)\\
= \frac{1}{4\pi t} (e^{-\frac{|u-v|^2}{4t}} + e^{-\frac{|u + v|^2}{4t}} -
e^{-\frac{(x_1 - x_2)^2 + (y_1+y_2)^2}{4t}} - e^{-\frac{(x_1 + x_2)^2 + (y_1 - y_2)^2}{4t}}).
\end{multline*}
Writing this in polar coordinates with $u=re^{i\phi}$, $v=r'e^{i\phi'}$ we obtain
\begin{eqnarray*}
p_{C}(t,u,v) &=& \frac{e^{-\frac{r^2 + r'^2}{4t}}}{4\pi t} (e^{\frac{r r'}{2t}\cos(\phi'-\phi)} + e^{-\frac{r r'}{2t}\cos(\phi'-\phi)}\\
& & -e^{\frac{r r'}{2t}\cos(\phi'+\phi)} - e^{-\frac{r r'}{2t}\cos(\phi'+\phi)} )\\
&=& \frac{e^{-\frac{r^2 + r'^2}{4t}}}{2\pi t} \left( \cosh\left( \frac{rr'\cos(\phi'-\phi)}{2t} \right) -
\cosh\left( \frac{rr'\cos(\phi'+\phi)}{2t} \right) \right). \label{eq:hkCuad}
\end{eqnarray*}

Let $R>0$, and recall the factor $\frac{2}{\alpha} = \frac{4}{\pi}$ in this case.   Let  ${\chi}_{S_{\pi/2,R}}$ be the characteristic function of the finite sector $S_{\pi/2,R}$
\begin{multline*}
\frac{4}{\pi}\Tr\left( {\mathcal M}_{{\chi}_{S_{\pi/2,R}}\left( 1 + \log(r) \right)} e^{-t\Delta_{\pi/2}}\right)
= \int_{0}^{R}\int_{0}^{\pi/2} \frac{4}{\pi} ( 1 + \log(r)) p_{C}(t,r,\phi,r,\phi)\ r\ dr\ d\phi\\ 
=\int_{0}^{R}\int_{0}^{\pi/2} \frac{4}{\pi} ( 1 + \log(r))
\frac{e^{-\frac{r^2}{2t}}}{4\pi t} (e^{\frac{r^2}{2t}} + e^{-\frac{r^2}{2t}}
-e^{\frac{r^2}{2t}\cos(2\phi)} - e^{-\frac{r^2}{2t}\cos(2\phi)} )\ r d\phi dr \\
= \frac{1}{\pi^2 t} \int_{0}^{R}\int_{0}^{\pi/2} (1 + \log(r))
(1 + e^{-\frac{r^2}{t}} -e^{-\frac{r^2}{2t}}e^{\frac{r^2}{2t}\cos(2\phi)} - e^{-\frac{r^2}{2t}}e^{-\frac{r^2}{2t}\cos(2\phi)} )\ r d\phi dr.
\end{multline*}

We split this integral into two terms, \begin{eqnarray*}
T_1(t) &=& \frac{1}{\pi^2 t} \int_{0}^{R}\int_{0}^{\pi/2} (1 + \log(r))(1 + e^{-\frac{r^2}{t}})\ r d\phi dr\\
&=& \frac{1}{2\pi t}\int_{0}^{R} (1 + \log(r))(1 + e^{-\frac{r^2}{t}})\ r dr\\
&=&
\frac{1}{2\pi t} \left( \int_{0}^{R} r dr + \int_{0}^{R} \log(r)\ r dr +
\int_{0}^{R} e^{-\frac{r^2}{t}}\ r dr
+ \int_{0}^{R} \log(r) e^{-\frac{r^2}{t}}\ r dr \right),
\end{eqnarray*}

\begin{eqnarray*}
T_2(t) &=& -\frac{1}{\pi^2 t} \int_{0}^{R}\int_{0}^{\pi/2} (1 + \log(r))
(e^{-\frac{r^2}{2t}}e^{\frac{r^2}{2t}\cos(2\phi)} + e^{-\frac{r^2}{2t}}e^{-\frac{r^2}{2t}\cos(2\phi)} )\ r d\phi dr\\
&=& - \frac{1}{\pi^2 t} \int_{0}^{R} (1 + \log(r)) e^{-\frac{r^2}{2t}} \int_{0}^{\pi/2}
(e^{\frac{r^2}{2t}\cos(2\phi)} + e^{-\frac{r^2}{2t}\cos(2\phi)}) \ d\phi \ r dr.
\end{eqnarray*}

\begin{claim} 
The integral $(T_1 + T_2)(t)$ has an asymptotic expansion as $t\to 0$ of the form 
$$(T_1 + T_2)(t) = 
 \frac{1}{2\pi t}\left( R + \frac{R^{2}\log(R)}{2} - \frac{R^{2}}{4} \right)  
-\frac{ R \log R}{\pi \sqrt{\pi t}} + \frac{\log(t)}{8\pi}
- \frac{1}{4\pi} - \frac{\gamma_e}{8\pi} + O(t^{1/2})$$
\end{claim} 

\begin{proof} 

By inspection, the first two terms in $T_1(t)$ contribute only to the $t^{-1}$ coefficient, and that contribution is 
$$\frac{1}{2\pi t} \left( R + \frac{R^2 \log(R)}{2} - \frac{R^2}{4} \right).$$
So, we look at the expansion in $t$ of

\begin{equation}
\wt{T}_1(t) = \frac{1}{2\pi t} \left(\int_{0}^{R} e^{-\frac{r^2}{t}}\ r dr + \int_{0}^{R} \log(r) e^{-\frac{r^2}{t}}\ r dr \right).
\label{eq:FiT1}
\end{equation}
We compute
$$\frac{1}{2\pi t} \int_{0}^{R} e^{-\frac{r^2}{t}}\ r dr = \frac{1}{4\pi} \int_{0}^{R^2/t} e^{-u}\ du = \frac{1}{4\pi} - \frac{1}{4\pi}e^{-R^2/t},$$
and

\begin{multline*}
\frac{1}{2\pi t} \int_{0}^{R} \log(r) e^{-\frac{r^2}{t}}\ r dr = \frac{1}{8\pi}\int_{0}^{R^2/t} \log(tu) e^{-u} \ du\\ =
\frac{1}{8\pi} \int_{0}^{\infty} \log(u) e^{-u} \ du - \frac{1}{8\pi} \int_{R^2/t}^{\infty} \log(u) e^{-u} \ du + 
\frac{1}{8\pi}\int_{0}^{R^2/t} \log(t) e^{-u} \ du\\ =
\frac{-\gamma_e}{8\pi} - \frac{1}{8\pi} \int_{R^2/t}^{\infty} \log(u) e^{-u} \ du + \frac{\log(t)}{8\pi}(1-e^{-R^2/t})
\end{multline*}
where $\gamma_e$ is the Euler constant.

 Since we are interested in the behavior for fixed $R$ as $t \downarrow 0$, we may assume $R^2 > t $ so that 
 $$0< \log(u) < u, \quad \forall u > R^2/t.$$
 Then, we estimate 
 $$\int_{R^2/t} ^\infty \log(u) e^{-u} du \leq \int_{R^2/t} ^\infty u e^{-u} du = \frac{R^2 e^{-R^2/t}}{t} + e^{-R^2/t}.$$ 
 This is vanishing rapidly as $t \downarrow 0$ for any fixed $R>0$.  

Therefore for $\wt{T}_1(t)$ we obtain
$$\wt{T}_1(t) = \frac{1}{4\pi} - \frac{\gamma_e}{8\pi} +\frac{\log(t)}{8\pi} + O(t^{\infty}), \quad t\downarrow 0.$$

Hence, $T_1(t)$ has the asymptotic expansion 
\begin{equation}
T_{1}(t)=  \frac{1}{2\pi t}\left( R + \frac{R^{2}\log(R)}{2} - \frac{R^{2}}{4} \right) + \frac{\log(t)}{8\pi}
+ \frac{1}{4\pi} - \frac{\gamma_e}{8\pi} + O(t^{\infty}), \quad t \downarrow 0.
\end{equation}

Let us consider now the second term, $T_2(t)$:
\begin{eqnarray*}
T_2(t) &=& -\frac{1}{\pi^2 t} \int_{0}^{R}\int_{0}^{\pi/2} (1 + \log(r))
(e^{-\frac{r^2}{2t}}e^{\frac{r^2}{2t}\cos(2\phi)} + e^{-\frac{r^2}{2t}}e^{-\frac{r^2}{2t}\cos(2\phi)} )\ r d\phi dr\\
&=& - \frac{1}{\pi^2 t} \int_{0}^{R} (1 + \log(r)) e^{-\frac{r^2}{2t}} \int_{0}^{\pi/2}
(e^{\frac{r^2}{2t}\cos(2\phi)} + e^{-\frac{r^2}{2t}\cos(2\phi)}) \ d\phi \ r dr.
\end{eqnarray*}
The modified Bessel function of the first kind of order zero admits the following integral representation

$$I_{0}(a) = \frac{1}{\pi} \int_0^{\pi} e^{a\cos(\phi)} \ d\phi$$
for $a\in \R$, $a\geq 0$. After a change of variables
$$\int_0^{\pi/2} e^{a\cos(2\phi)} \ d\phi = \frac{\pi}{2} I_{0}(a).$$
Since $\cos(\pi-x) = -\cos(x)$, we obtain
$$T_2(t)= -\frac{1}{\pi t} \int_{0}^{R} (1 + \log(r)) e^{-\frac{r^2}{2t}} I_{0}\left(\frac{r^2}{2t}\right)  \ r dr.$$
We know how to compute these integrals using techniques inspired by \cite{vs}.  Let us write $T_{2,1}$ for the integral with $1$, and $T_{2,2}$ for the integral with $\log(r)$,
so $T_2 = T_{2,1} + T_{2,2}$.  We start by changing variables, letting $u = r^2/2t$, 
$$ T_{2,1} =
- \frac{1}{\pi t}\int_{0}^{R} r e^{-r^2/2t} I_{0}\left(\frac{r^2}{2t}\right) dr = - \frac{1}{\pi} \int_{0}^{R^2/2t} e^{-u} I_{0}(u) du.
$$
Let $I_1(u)$ be the modified Bessel function of first kind of order one.
By \cite[(3) p. 79]{wa} with $\nu = 1$,
\begin{equation} \label{wa79-3} u I_1 '(u) + I_1 (u) = u I_0 (u). \end{equation}
By \cite[(4) p. 79]{wa} with $\nu = 0$,
\begin{equation} \label{wa79-4} u I_0 '(u) = uI_1 (u). \end{equation}
We use these to calculate
\begin{eqnarray*}
\frac{d}{du} \left( e^{-u} u (I_0 (u) + I_1 (u) ) \right) &=& e^{-u} \left(-u I_0 (u) - uI_1 (u) + I_0 (u) + I_1 (u) + u I_0'(u) + u I_1 '(u) \right)\\
&=& e^{-u} \left(  - uI_1 (u) + I_0 (u) + u I_0'(u) \right),  \quad \quad \text{ by (\ref{wa79-3})}\\
&=& e^{-u} I_0 (u), \quad \quad \quad \text{ by (\ref{wa79-4})}.
\end{eqnarray*}
Next, define
\begin{equation} \label{magic} g(u) := e^{-u} u (I_0 (u) + I_1 (u)), \end{equation}
and note that we have computed
$$g'(u) = e^{-u} I_0 (u).$$
We therefore have
$$- \frac{1}{\pi} \int_{0}^{R^2/2t} e^{-u} I_{0}(u) du = -\frac{1}{\pi} \left( g(R^2/2t) - g(0) \right).$$
These Bessel functions are known to satisfy (see \cite{wa})
$$I_0 (0) = 1, \quad I_1 (0) = 0.$$
It follows that $g(0) = 0$, and we therefore obtain that
$$- \frac{1}{\pi} \int_{0}^{R^2/2t} e^{-u} I_{0}(u) du = -\frac{1}{\pi} g(R^2/2t).$$ 
For large arguments, the Bessel functions admit the following asymptotic expansions (see \cite{wa})
$$I_j (x) = \frac{e^x}{\sqrt{2 \pi x}} \left(1 - \frac{1}{2x} \left(j^2 - \frac{1}{4} \right) + \sum_{k=2} ^\infty c_{j,k} x^{-k} \right), \quad x \gg 0, \quad j=0,1.$$
We therefore compute the expansion of $g$ as follows
$$g(u) = \frac{\sqrt{u}}{\sqrt{2\pi}} \left( 2 - \frac{1}{4u} + \sum_{k=2} ^\infty (c_{0,k} + c_{1,k} ) u^{-k} \right), \quad u \gg 1.$$
Consequently, for $u = R^2/2t$ we have
$$g(R^2/2t) = \frac{R}{\sqrt{4\pi t}} \left( 2 - \frac{t}{2R^2} + \sum_{k=2} ^\infty (c_{0,k} + c_{1,k} ) \left( \frac{2t}{R^2} \right)^k \right), \quad t \ll 1.$$
It follows that for small $t$, $T_{2,1}(t)$ has the following asymptotic expansion
$$T_{2,1}(t) = - \frac{R}{\pi \sqrt{4\pi t}} \left( 2 - \frac{t}{2R^2} + \sum_{k=2} ^\infty (c_{0,k} + c_{1,k} ) \left( \frac{2t}{R^2} \right)^k \right), \quad t \ll 1.$$
Therefore 
$$T_{2,1}(t) = - \frac{R}{\pi \sqrt{\pi t}} + O(t^{1/2}) \quad \textrm{ as } \quad t \to 0.$$

Now, let us look at $T_{2,2}$.  Changing variables again $u = r^2/2t$ we obtain
\begin{eqnarray*}
T_{2,2} &=& - \frac{1}{\pi t}\int_{0}^{R} r \log(r) e^{-r^2/2t} I_{0}\left(\frac{r^2}{2t}\right) dr =
- \frac{1}{\pi} \int_{0}^{R^2/2t} \log(\sqrt{2tu}) e^{-u} I_{0}(u) du\\
&=& - \frac{1}{2\pi} \int_{0}^{R^2/2t} \log(u) e^{-u} I_{0}(u) du - \frac{\log(2t)}{2\pi} \int_{0}^{R^2/2t} e^{-u} I_{0}(u) du.\\
\end{eqnarray*}

For the first integral we use (\ref{magic}) and integrate by parts,
$$\int_0 ^{R^2/2t} \log (u) e^{-u} I_0 (u) du = \log (u) g(u) |_0 ^{R^2/2t} - \int_0 ^{R^2/2t} e^{-u} (I_0 (u) + I_1 (u) ) du.$$
Since $g'(u) = e^{-u} I_0 (u)$, we have
$$ \int_0 ^{R^2/2t} e^{-u} (I_0 (u) + I_1 (u) ) du = g(R^2/2t) - g(0) + \int_0 ^{R^2/2t} e^{-u} I_1 (u) du.$$
Note that $I_0 '(u) = I_1 (u)$.  Therefore, we integrate by parts again,
\begin{eqnarray*} \int_0 ^{R^2/2t} e^{-u} I_1 (u) du &=& e^{-u} I_0 (u) |_0 ^{R^2/2t} - \int_0 ^{R^2/2t} -e^{-u} I_0 (u) du,\\
&=&  e^{-u} I_0 (u) |_0 ^{R^2/2t} + g(u) |_0 ^{R^2/2t} .\end{eqnarray*}
Putting these calculations together, we have
$$\int_0 ^{R^2/2t} \log(u) e^{-u} I_0 (u) du
= \log (u) g(u) |_0 ^{R^2/2t} -  2 \left( g(R^2/2t) - g(0) \right) - e^{-u} I_0 (u) |_0 ^{R^2/2t}.$$
Therefore, we have calculated
\begin{eqnarray*}
- \frac{1}{2\pi} \int_{0}^{R^2/2t} \log(u) e^{-u} I_{0}(u) du &=&
 \frac{1}{2\pi} \left( -  \log (u) g(u) |_0 ^{R^2/2t} + \right. \\
& & \left. 2 \left( g(R^2/2t) - g(0) \right) + e^{-u} I_0 (u) |_0 ^{R^2/2t} \right); \\
- \frac{\log(2t)}{2\pi} \int_{0}^{R^2/2t} e^{-u} I_{0}(u) du &=&  - \frac{\log(2t)}{2\pi} \left( g(R^2/2t) - g(0) \right).
\end{eqnarray*}

Since $g(0) = 0$ and $I_0(0)=1$, we have
\begin{multline*}
T_{2,2} (t) \\
= \frac{1}{2\pi} \left( - \log(R^2/2t) g(R^2/2t) + 2g(R^2/2t) + e^{-R^2/2t} I_0 (R^2/2t) -1 -\log (2t) g(R^2/2t) \right)\\
=  \frac{1}{2\pi} \left( - 2 \log (R) g(R^2/2t) + 2g(R^2/2t) + e^{-R^2/2t} I_0 (R^2/2t) -1 \right).
\end{multline*}

We use the asymptotic expansion of $I_0(u)$ for $u \to \infty$ to compute
$$e^{-R^2/2t} I_0 (R^2/2t) = \frac{\sqrt{t}}{R \sqrt{\pi}} \left(1 + \frac{t}{4R^2} + \sum_{k=2} ^\infty c_{0,k}  \left( \frac{2t}{R^2} \right)^k \right), \quad t \ll 1.$$
We therefore obtain that the asymptotic expansion of $T_{2,2}(t)$ is

\begin{multline*}
-\frac{1}{2\pi} -\frac{ R \log R}{\pi \sqrt{4 \pi t}}  \left( 2 - \frac{t}{2R^2} + \sum_{k=2} ^\infty (c_{0,k} + c_{1,k} ) \left( \frac{2t}{R^2} \right)^k \right) \\
+\frac{R}{\pi\sqrt{4\pi t}} \left( 2 - \frac{t}{2R^2} + \sum_{k=2} ^\infty (c_{0,k} + c_{1,k} ) \left( \frac{2t}{R^2} \right)^k \right)\\
 + \frac{\sqrt{t}}{2\pi R \sqrt{\pi}} \left(1 + \frac{t}{4R^2} + \sum_{k=2} ^\infty c_{0,k}  \left( \frac{2t}{R^2} \right)^k \right), \quad t \ll 1.
\end{multline*}

Putting the contributions of $T_{1}$ and $T_{2}$ together, we obtain 

\begin{multline*}
T_{1} + T_2 (t)=  \frac{1}{2\pi t}\left( R + \frac{R^{2}\log(R)}{2} - \frac{R^{2}}{4} \right)  
-\frac{ R \log R}{\pi \sqrt{\pi t}} + \frac{\log(t)}{8\pi}
- \frac{1}{4\pi} - \frac{\gamma_e}{8\pi} + O(t^{1/2}),
\end{multline*}
which completes the proof of the claim. \end{proof} 

To determine the variational Polyakov formula, we combine the ingredients from the claim together with the contribution of the the other parts of the sector. Recalling the parametrix construction  in \S \ref{ss:hkp}, and that we use $*$ for the index in $\{\alpha,i,e,a,c\}$, we have that
$$\Tr\left( \mathcal M_{(1 + \log(r))} e^{-t\Delta_{\pi/2}}\right) =
\int_0^1 \int_0^{\pi/2} \left( 1 + \log(r) \right) H_{p}(r,\phi,r,\phi,t) \ r d\phi dr + O(t^{\infty})$$
$$= \Tr\left(e^{-t\Delta_{\pi/2}}\right) + \int_0^1 \int_0^{\pi/2} \log(r) \left(\sum_{*} \chi_{*}(r,\phi) H_{*}(r,\phi,r,\phi,t)\right) \ r d\phi dr+ O(t^{\infty})$$
\begin{multline*}
\sum_{*} \int_0^1 \int_0^{\pi/2} \log \cdot \chi_{*} \cdot H_{*}\ dA = 
\int_{\cN_{\pi/2}} \log \cdot \chi_{\pi/2} \cdot p_c \ dA + \int_{\cN_c} \log \cdot \chi_{c} \cdot H_{S_{\pi/2}} \ dA \\
+ \int_{S_{\pi/2}\setminus (\cN_{\pi/2} \cup \cN_c) } \log \cdot \left(\sum_{*} \chi_{*} \cdot H_{*}\right) \ dA
\end{multline*}
where to simplify the notation, $dA=r d\phi dr$. Now we have to look for the coefficients $a_{2,0}$ and $a_{2,1}$ in the short time asymptotic expansion (\ref{exp-exists}). 
Recall that the constant term in the asymptotic expansion of the heat trace, $\Tr(e^{-t\Delta_{\pi/2}})$ was computed in equation (\ref{t0-heattrace}) and in this case it is 
$$\zeta_{\Delta_{\pi/2}}=-\frac{1}{12} + 3 \left( \frac{\pi^2 + \pi^2/4}{24 \pi \pi/2} \right) = \frac{11}{48}.$$
Recalling the factor of $2/\alpha$ with $\alpha = \pi/2$ in this case, the total contribution from the trace of the heat kernel is 
$$\frac{4}{\pi} \ \zeta_{\Delta_{\pi/2}}=\frac{11}{12\pi}.$$
Since this term also includes the purely local corner contribution from the origin, which is already contained in the calculation of 
$$\int_{0}^{R}\int_{0}^{\pi/2} \frac{4}{\pi} ( 1 + \log(r)) p_{C}(t,r,\phi,r,\phi)\ r\ dr\ d\phi$$
in Claim 1, we need to remove this part, which is, since $\alpha = \pi/2$,  
$$\frac{2}{\alpha} \left( \frac{\pi^2 - \pi^2/4}{24 \pi (\pi/2)} \right) = \frac{1}{4\pi}.$$
So we have
$$\frac{11}{12\pi} - \frac{1}{4\pi} = \frac{2}{3\pi}.$$

As we proved in \S \ref{ss:pthexpexists} above, the integrals over $\cN_c$ and $S_{\pi/2}\setminus (\cN_{\pi/2}\cup \cN_c)$ do not contribute to the coefficients $a_{2,0}$ and $a_{2,1}$. 

Consequently, putting all the terms which contribute to the formula together, gives 
$$\frac{\log(t)}{8\pi} - \frac{1}{4\pi} - \frac{\gamma_e}{8\pi} + \frac{2}{3\pi}.$$ 

The variational Polyakov formula for the quarter circle is consequently 
\begin{equation} \label{variation-qc}
\left.\frac{\pa}{\pa \gamma} \big(-\log(\det(\Delta_{S_\gamma}))\big)\right|_{\gamma=\pi/2} = 
\frac{-\gamma_e}{4\pi} + \frac{5}{12 \pi}.
\end{equation} 
\qed


\section{Carslaw-Sommerfeld heat kernel} \label{Scarslaw} 
In this section we use the explicit form of the heat kernel on an infinite angular sector with opening angle $\alpha$ given by Carslaw in \cite{Carslaw} to prove the existence of the asymptotic expansion of $\Tr(\cM_{\chi_{\cN_\alpha}\log(r)}e^{-t\Delta_{\alpha}})$. At the same time we compute the contribution of this part to the total Polyakov formula. This will complete the proofs of Theorems \ref{th-exp-exists} and \ref{allsectors}. 

In \cite{Carslaw}, Carslaw gave the following formula for the heat kernel on an infinite angular sector with opening angle $\alpha$:
\begin{equation}
\wt{H}_{\alpha}(r,\phi,r',\phi',t) = \frac {e^{-(r^2 + r'^2)/4t}}{8\pi \alpha t}
\int_{A_\phi} e^{rr'\cos(z -\phi)/2t} \frac{e^{i\pi z/\alpha}}{e^{i\pi z/\alpha}-e^{i\pi \phi'/\alpha}} dz
\label{eq:CarslawHK}
\end{equation}
where $A_\phi$ is the contour in the $\C_z$-plane that is the union of the two following contours: one contained in
$\{z | \phi -\pi < \re(z) < \phi +\pi, \im(z) > 0\}$ going from $\phi +\pi +i\infty$ to $\phi -\pi +i\infty$,
and the other one contained in
$\{z | \phi -\pi < \re(z) < \phi +\pi, \im(z) < 0\}$ going from $\phi -\pi -i\infty$ to $\phi+\pi - i\infty$.  In Figure \ref{fig:hkc1} we reproduce original Carslaw's contour from \cite{Carslaw}.  
\begin{figure} \includegraphics[height=170pt]{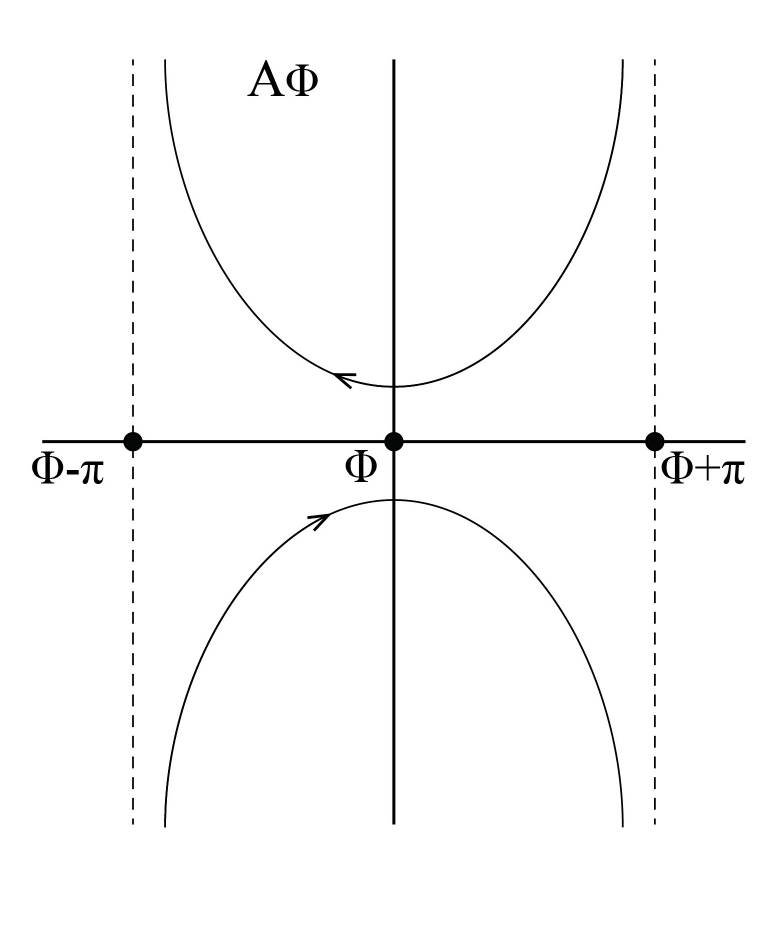} \caption{Contour $A_{\phi}$ in the $\C_z$ plane.}   \label{fig:hkc1} \end{figure} 

As noted there, this contour can be deformed into a different contour, depicted in Figure \ref{fig:hkc2}, that is composed of the following curves:
\begin{enumerate}
\item $\ell_1 = \{\phi -\pi + i y, y\in \R\}$ oriented from $-i\infty$ to $i\infty$,
\item $\ell_2 = \{\phi +\pi + i y, y\in \R\}$ oriented from $i\infty$ to $-i\infty$, and
\item small circles around the poles in the interval $z \in ]\phi -\pi,\phi +\pi[$. Since we will be considering $\phi$ close to $\phi'$, 
poles on the lines will not appear. 
\end{enumerate}
Notice that at the lines $\ell_1$ and $\ell_2$, $\cos(z-\phi) <0$ since
$$\cos(z-\phi) = \cos(x -\phi + iy) = \cos(\pm \pi +iy) = -\cosh(y)<0$$
the integrals over the straight lines converge and will vanish in the limit as $t\to 0$ (c.f. \cite{Carslaw} (iii) on p. 367).  
\begin{figure} \includegraphics[height=170pt]{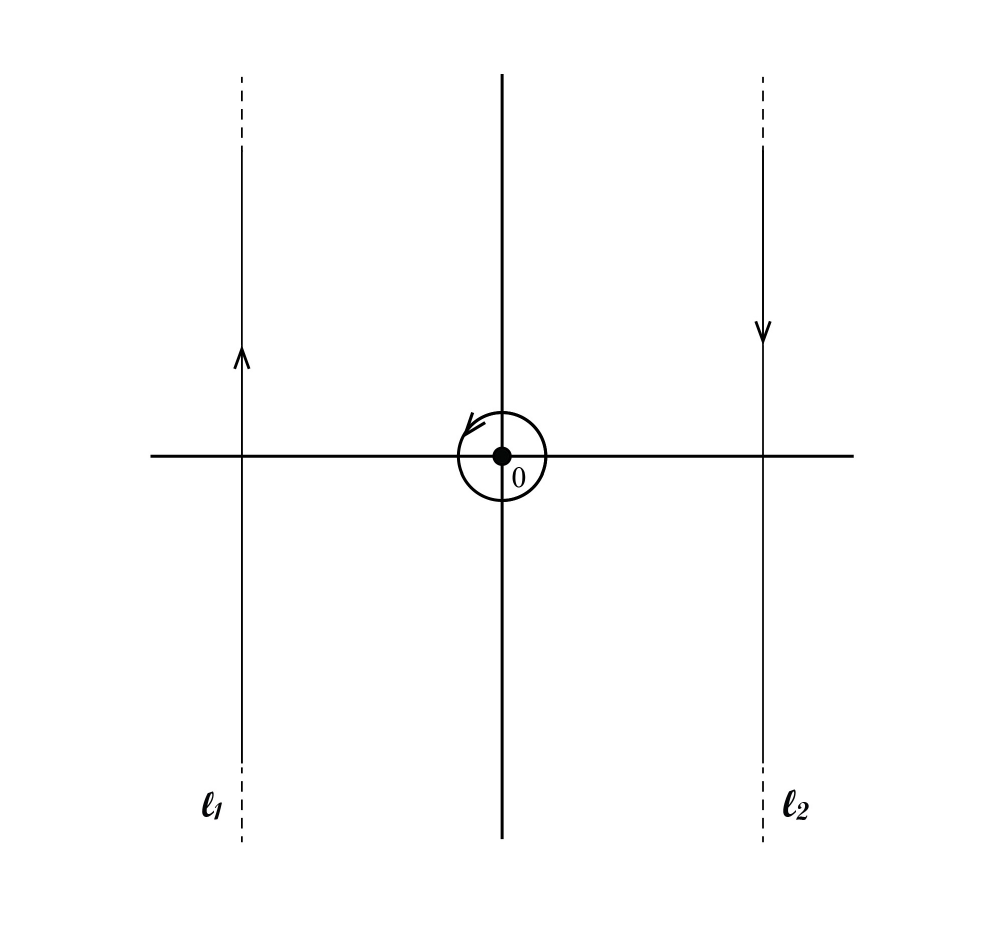} \caption{Deformed contour. To simplify the picture we assume only one pole at $z=0$.} \label{fig:hkc2}\end{figure} 

Unfortunately, this kernel does not correspond to the Dirichlet Laplacian since it does not satisfy the boundary
condition. To remedy this, we use the method of images as in \cite{AuSa}.  We first re-write (\ref{eq:CarslawHK}) with a change of coordinates, 
$w := z - \phi$, and write $A_0$ for the contour $A_\phi$ defined above with $\phi=0$ in the $\C_w$ plane, then
$$\wt{H}_{\alpha}(r,\phi,r',\phi',t) = \frac {e^{-(r^2 + r'^2)/4t}}{8\pi \alpha t}
\int_{A_0} e^{rr'\cos(w)/2t} \frac{1}{1-e^{i\pi (\phi'-\phi -w)/\alpha}} dw.$$

This is the so-called ``direct term'' corresponding to $\phi'-\phi$.  By the method of images, to obtain the Dirichlet heat kernel, we must incorporate the term co\-rres\-ponding to $\phi' + \phi$, this is 
$$\wt{H}_{\alpha}(r,-\phi,r',\phi',t) = \frac {e^{-(r^2 + r'^2)/4t}}{8\pi \alpha t}
\int_{A_0} e^{rr'\cos(w)/2t} \frac{1}{1-e^{i\pi (\phi'+\phi -w)/\alpha}} dw,$$
and it is called the ``reflected term''. Consequently, the Dirichlet heat kernel is 
\begin{eqnarray}
H_{\alpha}(r,\phi,r',\phi',t) &=& \frac {e^{-(r^2 + r'^2)/4t}}{8\pi \alpha t}
\left(\int_{A_0} e^{rr'\cos(w)/2t} \frac{1}{1-e^{i\pi (\phi'-\phi -w)/\alpha}} dw \right.\notag \\
& & \left. -\int_{A_0} e^{rr'\cos(w)/2t} \frac{1}{1-e^{i\pi (\phi'+\phi -w)/\alpha}} dw
\right) \label{eq:DFhk} 
\end{eqnarray}


\subsection{Contribution from the poles}
Let us define the following functions:
$$f_1(z) = \frac{e^{rr'\cos(z)/2t}}{1-e^{i \pi (\phi'-\phi -z)/\alpha}}, \quad f_2 (z) = \frac{e^{rr'\cos(z)/2t}}{1-e^{i \pi (\phi'+\phi -z)/\alpha}},$$
The first thing to do is to compute the residues at the poles of $f_1$ and $f_2$ within the interval $]-\pi, \pi[$, for $\phi'$ and $\phi$ close to each other but different. 
There are two reasons for this assumption. The first reason is that we would like to have a general expression for the heat kernel close to the diagonal, not only at the diagonal. The second reason is more serious, and arises due to the possibility of non-commuting limits.  For example, to determine the terms in the heat kernel arising from the residues at the poles, the correct order of computations is first to compute with the heat kernel for $\phi'$ and $\phi$ close, and then afterwards set $\phi'=\phi$.  In some cases, if one first sets $\phi'=\phi$ and then attempts to compute, the result is incorrect.  
In general the function $f_1$ has poles at the points 
$$(\phi' - \phi - z)\pi/\alpha = 2k\pi \iff \phi' - \phi - z = 2k\alpha \iff z= \phi' - \phi +2k\alpha, \, k \in \Z.$$
Similarly, $f_2$ has poles at the points 
$$(\phi' + \phi - z)\pi/\alpha = 2j\pi \iff \phi' + \phi - z = 2j\alpha \iff z= \phi' + \phi +2j\alpha, \, j \in \Z.$$
We first assume without loss of generality $\phi' > \phi$, later when we want to compute the trace we make $\phi' = \phi$. 

Then, the poles of $f_1$ and $f_2$ which lie in the interval $]-\pi, \pi[$ are those with 
\begin{equation}
k, j \in \Z, \quad \frac{-\pi}{2\alpha} < \frac{\phi' - \phi}{2\alpha} + k < \frac{\pi}{2\alpha}, 
\quad \textrm{ and } \quad \frac{-\pi}{2\alpha} < \frac{\phi' + \phi}{2\alpha} + j < \frac{\pi}{2\alpha},
\label{eq:condjkpoles}
\end{equation}
respectively.  
\subsubsection{Pole contribution from the direct term}  
We compute the residues at the poles of $f_1$: 
\begin{eqnarray*}
{\text{Res}}_{z=\phi'-\phi+2 k\alpha} \frac{e^{rr'\cos(z)/2t}}{1-e^{i\pi(\phi'-\phi-z)/\alpha}}
&=& \lim_{z \to \phi'-\phi +2 k \alpha}
\frac{(z-(\phi'-\phi - 2k \alpha))e^{rr'\cos(z)/2t}}{1-e^{i\pi(\phi'-\phi-z)/\alpha}}\\
&=& \frac{\alpha}{i\pi} e^{rr'\cos(\phi'-\phi + 2k\alpha)/2t}
\end{eqnarray*}
Therefore, the integrals over the contours surrounding these poles are, by the Residue Theorem, 
$$2\alpha e^{rr'\cos(\phi'-\phi + 2k\alpha)/2t}.$$  

The poles which are contained in the interval $]-\pi, \pi[$ \em depend \em on the value of the angles $\phi$ and $\phi'$. That is why, in order to have a comprehensive formula close to the diagonal, we restrict their range by assuming $\phi'$ and $\phi$ are close.  We compute the range of all these poles.  Since we are assuming $\phi' > \phi$, it follows that $\phi' - \phi > 0$. 
Without loss of generality, we may assume for a short moment that $\phi=0$ and $\phi'\leq \alpha/4$, since we are interested in the case when $\phi$ and $\phi'$ are close. The equation for $k$ becomes
$$\frac{-\pi}{2\alpha} - \frac{\phi'}{2\alpha} <  k < \frac{\pi}{2\alpha} - \frac{\phi'}{2\alpha}, \quad \textrm{ with } \quad 
0<\frac{\phi'}{2\alpha}\leq \frac{1}{8}.$$
Consequently, the smallest pole of $f_1$ occurs at 
\begin{equation} \label{kmin} k_{min} = \ceil*{ \frac{-\pi}{2\alpha} }. \end{equation} 

For the largest pole of $f_1$ we have two cases: $\frac{\pi}{2\alpha}\not\in \Z$ and otherwise.
If $\frac{\pi}{2\alpha}\in \Z$, then $k_{max} = \frac{\pi}{2\alpha} - 1$. If, on the contrary, $\frac{\pi}{2\alpha} \not\in \Z$ then 
$$\frac{\pi}{2\alpha} = \floor*{\frac{\pi}{2\alpha}} + \delta, \textrm{ for some } \delta\in ]0,1[.$$
In this case, we shall and may assume in addition that $\phi'/2\alpha < \delta$. This will be, in terms of $\phi$ and $\phi'$,
$\phi'-\phi < \delta 2\alpha$. Therefore the largest pole occurs at $\floor*{\frac{\pi}{2\alpha}}$. Summarizing we obtain: 
\begin{equation} \label{kmax} k_{max} = \floor*{\frac{\pi}{2\alpha}} \textrm{ if } \frac{\pi}{2\alpha} \not\in \Z, \textrm{ otherwise } k_{max} = \frac{\pi}{2\alpha} - 1. \end{equation} 
Therefore the contribution to the heat kernel is:
$$\frac {e^{-(r^2 + r'^2)/4t}}{8\pi \alpha t}\sum_{k\in [k_{\min},k_{\max}]} 2\alpha e^{rr'\cos(\phi'-\phi + 2k\alpha)/2t}.$$


To compute the Polyakov formula contributions arising from these poles, we restrict to the diagonal by setting $\phi' = \phi$, $r'=r$ in the expression above. We then multiply by $\log(r)$ and integrate over a finite sector of radius $R$:
\begin{multline*}
\int_{0}^{R}\int_{0}^{\alpha} \sum_{k\in [k_{\min},k_{\max}]} \frac {e^{-(r^2 )/2t}}{8\pi \alpha t} \log(r) 2\alpha e^{r^2\cos(2k\alpha)/2t}
\ d\phi \ r \ dr \\
= \sum_{k\in [k_{\min},k_{\max}]} \frac{\alpha}{4\pi t}  \int_{0}^{R} e^{-r^2(1-\cos(2k\alpha))/2t}  \log(r)\ r \ dr .
\end{multline*}
We compute each of these integrals separately. 

If $\cos(2k\alpha) = 1$, 
$$\frac{\alpha}{4\pi t} \int_{0}^{R} e^{-r^2(1-\cos(2k\alpha))/2t}  \log(r)\ r \ dr = \frac{\alpha}{4\pi t}\int_{0}^{R} \log(r)\ r \ dr$$
then the coefficients of $t^0$ and $\log(t)$ as $t\downarrow 0$ vanish; there is no contribution from such $k$. 
We note that 
$$\cos(2k\alpha) = 1 \iff \exists \ell \in \Z \textrm{ with } k = \frac{\ell\pi}{\alpha}.$$ 

Assuming this is not the case, we use substitution in the integral, letting 
$$u = r^2 (1-\cos(2k\alpha))/(2t), \quad du = r dr (1-\cos(2k\alpha))/t.$$
Thus we consider
\begin{equation}
\frac{\alpha}{4\pi (1-\cos(2k\alpha))} \int_0 ^{R^2 (1-\cos(2k\alpha))/(2t)} e^{-u} \log(\sqrt{2tu} (1-\cos(2k\alpha))^{-1/2}) du. \label{auxaedt}
\end{equation}

Next, using the same argument as in the computation of $\wt{T}_1$ in the case of the quarter circle, we compute
\begin{multline*}
\int_0 ^{R^2 (1-\cos(2k\alpha))/(2t)} e^{-u} \log(2tu (1-\cos(2k\alpha))^{-1}) du\\ = \int_0 ^\infty e^{-u} \log(u) du 
-\int_{R^2 (1-\cos(2k\alpha))/(2t)}^\infty e^{-u} \log(u) \ du \\ + \left(\log(2/(1-\cos(2k\alpha)) +  \log(t) \right)(1-e^{-R^2 (1-\cos(2k\alpha))/(2t)})
\end{multline*}
In the same way as before the integral in the middle vanishes rapidly as $t\downarrow 0$. It follows from a straightforward calculation that the constant term in the asymptotic expansion as $t\to 0$ in the integral in (\ref{auxaedt}) is
\begin{equation} \label{f1-contribution1} \frac{\alpha}{8\pi (1-\cos(2k\alpha))} \left( -\gamma_e + \log\left( \frac{2}{1-\cos(2k\alpha)}\right) \right), \end{equation} 
and the $\log(t)$ term is 
\begin{equation} \label{f1-logcontribution} \frac{\alpha \log(t)}{8\pi(1-\cos(2k\alpha))}. \end{equation} 
Let $W_{\alpha}$ be defined by 
$$W_{\alpha} = \left\{ k \in \left( \Z \bigcap \left[k_{min},  k_{max}\right]\right) \setminus \left\{ \frac{\ell\pi}{\alpha} \right\}_{\ell \in \Z}\right\}.$$
Hence, the total contribution to the variational Polyakov formula will come from 
\begin{multline} \label{poles-f1-polyakov} \frac{\alpha}{8\pi(1-\cos(2k\alpha))} \sum_{k\in W_\alpha} 
\left(-\gamma_e + \log\left( \frac{2}{1-\cos(2k\alpha)}\right) + \log(t) \right) \end{multline} 

Recalling the factor of $\frac{2}{\alpha}$ and equation (\ref{variation}), the total contribution to the variational Polyakov formula is: 
$$\sum_{k\in W_{\alpha}} \frac{1}{4 \pi (1-\cos(2k\alpha))} \left(-2\gamma_e + \log\left( \frac{2}{1-\cos(2k\alpha)}\right)\right).$$

\subsubsection{Pole contribution from the reflected term} 
The residues at the poles of $f_2$ are: 
\begin{eqnarray*}
{\text{Res}}_{z=\phi'+\phi+2 j\alpha} \frac{e^{rr'\cos(z)/2t}}{1-e^{i\pi(\phi'+\phi-z)/\alpha}}
&=& \lim_{z \to \phi'+\phi +2 j \alpha}
\frac{(z-(\phi'+\phi - 2j \alpha))e^{rr'\cos(z)/2t}}{1-e^{i\pi(\phi'+\phi-z)/\alpha}}\\
&=& \frac{\alpha}{i\pi} e^{rr'\cos(\phi'+\phi + 2j\alpha)/2t}
\end{eqnarray*}
Therefore, the integrals over the contours surrounding these poles are, by the Residue Theorem, 
$$2 \alpha e^{rr'\cos(\phi'+\phi + 2j\alpha)/2t}.$$

Note that the location of the poles such that $z \in ]-\pi, \pi[$ \em depend \em on the value of $\phi$.  In particular, the set 
$$V_{\phi}:=]\frac{-\pi-2\phi}{2\alpha}, \frac{\pi-2\phi}{2\alpha}[\cap \Z$$
 depends on $\phi$. 
At first glance, this would seem to be problematic.  However, we shall see that by first integrating over $\phi \in [0, \alpha]$, a wonderful simplification occurs; this is made precise by the following Lemma.  

\begin{lemma} \label{l:marvsimpl}
For any $\alpha \in ]0,\pi[$, 
$$\int_0^\alpha \sum_{j \in ]\frac{-\pi-2\phi}{2\alpha}, \frac{\pi-2\phi}{2\alpha}[\cap \Z}  e^{r^2\cos(2\phi + 2j\alpha)/2t}\ d\phi
= \frac12 \int_{-\pi}^{\pi} e^{r^2\cos(\varphi)/2t}\ d\varphi = \pi I_0 (r^2/2t),$$
where $I_0$ is the modified Bessel function. 
\end{lemma}

\begin{proof} 
The proof goes by cases.  For different values of $\alpha$, we look the values of $j$ which satisfy the equation
$$-\pi <2\phi +2\alpha j < \pi, \quad \textrm{ with } \quad 0\leq \phi \leq \alpha$$
The sets $V_{\phi}$ are constant on intervals, so we split the integral over $[0,\alpha]$ into the integral over these subintervals; then we change variables $\varphi=2\phi +2j\alpha$, rearrange, and obtain the final result. 

We consider the following cases, and note that it is straightforward to verify that for any $\alpha \in ]0,\pi[$, precisely one of these cases holds:  
\begin{enumerate}
\item $\alpha=\frac{\pi}{2k+1}$,
\item $\alpha=\frac{\pi}{2k}$,
\item $\alpha=\frac{\pi}{2k-2\varepsilon} >\frac{\pi}{2k}$, with $k\geq 1$ and $\frac12>\varepsilon>0$, and,
\item $\alpha=\frac{\pi}{2k+1-2\varepsilon} >\frac{\pi}{2k+1}$, with $k\geq 1$ and $\frac12 >\varepsilon>0$. 
\end{enumerate}

 \noindent {\it Case $\alpha=\frac{\pi}{2k+1}$}: Here,  
$$j \in V_\phi \iff -k-\frac12 -\frac{\phi}{\alpha} < j < k +\frac12 -\frac{\phi}{\alpha}.$$ 

Then the set $V = V_\phi$ takes three different values: 
\begin{itemize}
\item On $[0,\alpha/2[$, $V=\{-k,\dots,k\}$,
\item at $\{\alpha/2\}$, $V=\{-k,\dots,k-1\}$
\item on $]\alpha/2,\alpha[$, $V=\{-k-1,\dots,k-1\}$
\item at $\alpha$, $V=\{-k-1,\dots,k-2\}$
\end{itemize}
Then, we have  
\[  \int_0^\alpha \sum_{j \in V_{\phi}} e^{r^2\cos(2\phi + 2j\alpha)/2t}\ d\phi \]
\[= \int_{0}^{\alpha/2}\sum_{j=-k}^k e^{r^2\cos(2\phi + 2j\alpha)/2t}\ d\phi + \int_{\alpha/2}^{\alpha}\sum_{j=-k-1}^{k-1} e^{r^2\cos(2\phi + 2j\alpha)/2t}\ d\phi \]
$$= \frac12 \sum_{j=-k}^k\int_{2j\alpha}^{2j\alpha+\alpha} e^{r^2\cos(\varphi)/2t}\ d\varphi +
\frac12 \sum_{j=-k-1}^{k-1} \int_{\alpha + 2j\alpha}^{2\alpha+2j\alpha} e^{r^2\cos(\varphi)/2t}\ d\varphi $$ 
$$= \frac12 \sum_{j=-k}^k\int_{2j\alpha}^{\alpha(2j+1)} e^{r^2\cos(\varphi)/2t}\ d\varphi +
\frac12 \sum_{j=-k-1}^{k-1} \int_{\alpha(2j+1)}^{\alpha(2j+2)} e^{r^2\cos(\varphi)/2t}\ d\varphi $$
$$= \frac12 \int_{(-2k-1)\alpha}^{-2k\alpha} e^{r^2\cos(\varphi)/2t}\ d\varphi + \frac12 \sum_{j=-k}^{k-1}\int_{2j\alpha}^{\alpha(2j+1)} e^{r^2\cos(\varphi)/2t}\ d\varphi + \frac12 \int_{2k\alpha}^{(2k+1)\alpha} e^{r^2\cos(\varphi)/2t}\ d\varphi $$
$$= \frac12 \int_{-\pi}^{\pi} e^{r^2\cos(\varphi)/2t}\ d\varphi. $$

\noindent {\it Case $\alpha=\frac{\pi}{2k}$}: In this case, $j \in V_\phi$ must satisfy $-k -\frac{\phi}{\alpha} < j < k -\frac{\phi}{\alpha}$.   
The set $V= V_\phi$ again takes three different values:
\begin{itemize}
\item At $\{0\}$, $V=\{-k+1,\dots,k-1\}$, 
\item on $]0,\alpha[$, $V=\{-k,\dots,k-1\}$
\item at $\{\alpha\}$, $V=\{-k,\dots,k-2\}$
\end{itemize}
The proof in this case follows straightforward.\\

\noindent {\it Case $\alpha=\frac{\pi}{2k-2\varepsilon} >\frac{\pi}{2k}$, with $k\geq 1$ and $\frac12>\varepsilon>0$}: In this case, $j \in V_\phi$ must satisfy 
$-k + \varepsilon -\frac{\phi}{\alpha} < j < k -\varepsilon -\frac{\phi}{\alpha}$. Then the set $V$ takes three different values:
\begin{itemize}
\item On $[0,\alpha\varepsilon]$, $V=\{-k+1,\dots,k-1\}$,
\item on $]\alpha\varepsilon,(1-\varepsilon)\alpha[$, $V=\{-k,\dots,k-1\}$
\item on $[(1-\varepsilon)\alpha,\alpha]$, $V=\{-k,\dots,k-2\}$
\end{itemize} In this case we compute
\begin{multline*}
 \int_0^\alpha \sum_{j \in V_{\phi}} e^{r^2\cos(2\phi + 2j\alpha)/2t}\ d\phi 
= \int_{0}^{\varepsilon\alpha}\sum_{j=-k+1}^{k-1} e^{r^2\cos(2\phi + 2j\alpha)/2t}\ d\phi \\  + 
\int_{\varepsilon\alpha}^{(1-\varepsilon)\alpha}\sum_{j=-k}^{k-1} e^{r^2\cos(2\phi + 2j\alpha)/2t}\ d\phi +
\int_{(1-\varepsilon)\alpha}^{\alpha}\sum_{j=-k}^{k-2} e^{r^2\cos(2\phi + 2j\alpha)/2t}\ d\phi\\
= \frac12 \sum_{j=-k+1}^{k-1} \int_{2j\alpha}^{2\varepsilon\alpha+2j\alpha} e^{r^2\cos(\varphi)/2t}\ d\varphi  +
\frac12 \sum_{j=-k}^{k-1} \int_{2\alpha(j+\varepsilon)}^{2\alpha(1-\varepsilon +j)} e^{r^2\cos(\varphi)/2t}\ d\varphi\\ +
\frac12 \sum_{j=-k}^{k-2}\int_{2(1-\varepsilon)\alpha+2j\alpha}^{2\alpha+2j\alpha} e^{r^2\cos(\varphi)/2t}\ d\varphi
\end{multline*}

Let $J(\varphi)$ denote $e^{r^2\cos(\varphi)/2t}$, then 
\begin{multline*}
 \int_0^\alpha \sum_{j \in V_{\phi}} e^{r^2\cos(2\phi + 2j\alpha)/2t}\ d\phi = \frac12 \sum_{j=-k+1}^{k-2} \int_{2j\alpha}^{2\alpha+2j\alpha} J\ d\varphi\\
+ \frac12 \left(\int_{2\alpha(k-1)}^{2\alpha(k-1+\varepsilon)} J\ d\varphi + \int_{2\alpha(-k+\varepsilon)}^{2\alpha(-k+1-\varepsilon)} J\ d\varphi
+ \int_{2\alpha(k-1+\varepsilon)}^{2\alpha(k-1+1-\varepsilon)} J\ d\varphi + \int_{2\alpha(-k+1-\varepsilon)}^{2\alpha(-k+1)}J\ d\varphi \right)\\
= \frac12 \left(\int_{2\alpha(-k+\varepsilon)}^{2\alpha(-k+1)} J\ d\varphi +
\int_{2\alpha(-k+1)}^{2\alpha(k-1)} J\ d\varphi +
\int_{2\alpha(k-1)}^{2\alpha(k-\varepsilon)} J\ d\varphi \right)= \frac12 \int_{-\pi}^{\pi} e^{r^2\cos(\varphi)/2t}\ d\varphi.
\end{multline*}

\noindent {\it Case $\alpha=\frac{\pi}{2k+1-2\varepsilon} >\frac{\pi}{2k+1}$, with $k\geq 1$ and $\frac12 >\varepsilon>0$}: The equation becomes 
$-k -\frac12 + \varepsilon -\frac{\phi}{\alpha} < j < k +\frac12 -\varepsilon -\frac{\phi}{\alpha}$. Then the set $V$ takes three different values:
\begin{itemize}
\item On $[0,\alpha(\frac12-\varepsilon)[$, $V=\{-k,\dots,k\}$,
\item on $]\alpha(\frac12-\varepsilon),\alpha(\frac12+\varepsilon)]$, $V=\{-k,\dots,k-1\}$
\item on $](\frac12+\varepsilon)\alpha,\alpha]$, $V=\{-k-1,\dots,k-1\}$
\end{itemize}
Here we have
\begin{multline*}
 \int_0^\alpha \sum_{j \in V_{\phi}} e^{r^2\cos(2\phi + 2j\alpha)/2t}\ d\phi 
= \int_{0}^{(\frac12-\varepsilon)\alpha}\sum_{j=-k}^{k} e^{r^2\cos(2\phi + 2j\alpha)/2t}\ d\phi \\ +
\int_{(\frac12-\varepsilon)\alpha}^{(\frac12+\varepsilon)\alpha}\sum_{j=-k}^{k-1} e^{r^2\cos(2\phi + 2j\alpha)/2t}\ d\phi +
\int_{(\frac12+\varepsilon)\alpha}^{\alpha}\sum_{j=-k-1}^{k-1} e^{r^2\cos(2\phi + 2j\alpha)/2t}\ d\phi\\
= \frac12\sum_{j=-k}^{k} \int_{2j\alpha}^{(1-2\varepsilon+2j)\alpha} e^{r^2\cos(\varphi)/2t}\ d\varphi +
\frac12\sum_{j=-k}^{k-1} \int_{(1-2\varepsilon+2j)\alpha}^{(1+2\varepsilon+2j)\alpha} e^{r^2\cos(\varphi)/2t}\ d\varphi \\ +
\frac12\sum_{j=-k-1}^{k-1} \int_{(1+2\varepsilon+2j)\alpha}^{2(j+1)\alpha} e^{r^2\cos(\varphi)/2t}\ d\varphi 
=\frac12\sum_{j=-k}^{k-1} \int_{2j\alpha}^{2(j+1)\alpha} e^{r^2\cos(\varphi)/2t}\ d\varphi \\ +
 \frac12 \left(\int_{2k\alpha}^{(1-2\varepsilon+2k)\alpha} e^{r^2\cos(\varphi)/2t}\ d\varphi + 
\int_{(1+2\varepsilon-2k-2)\alpha}^{-2k\alpha}e^{r^2\cos(\varphi)/2t}\ d\varphi \ \right) \\
= \frac12 \int_{(-2k-1+2\varepsilon)\alpha}^{(2k+1-2\varepsilon)\alpha} e^{r^2\cos(\varphi)/2t}\ d\varphi 
= \frac12 \int_{-\pi}^{\pi} e^{r^2\cos(\varphi)/2t}\ d\varphi.
\end{multline*}
Recalling the formula for the modified Bessel function of the second type,
$$I_0 (x) = \frac{1}{\pi} \int_0 ^\pi e^{x \cos(\theta)} d\theta,$$
we see that 
$$\frac{1}{2} \int_{-\pi} ^\pi e^{r^2 \cos(\varphi)/2t} d\varphi =  \int_0 ^\pi e^{r^2 \cos(\varphi)/2t} d\varphi = \pi I_0 (r^2/2t).$$
This completes the proof of the lemma.
\end{proof}
 
To compute the contribution to the Polyakov formula from these poles, we recall that the residues at the poles of $f_2$, restricted to the diagonal, give $2\alpha e^{r^2 \cos(2\phi + 2j\alpha)/2t}$.  Furthermore, there is a factor of $\frac{e^{-r^2/2t}}{8\alpha \pi t}$, and finally, the reflected term is subtracted in the definition of the heat kernel.  The preceding Lemma takes care of the integration with respect to $\phi$, and so it remains to analyze 
$$-\frac{2 \alpha}{8 \alpha t} \int_0 ^R e^{-r^2/2t} \log(r) I_0 (r^2/2t) r dr = \frac{-1}{4 t} \int_0 ^R e^{-r^2/2t} \log(r) I_0 (r^2/2t) r dr = \frac{\pi}{4}T_{2,2} (t),$$ 
where $T_{2,2} (t)$ was defined in \S \ref{s:qcrd} .
There, we computed the $t^0$ term in the expansion of $T_{2,2} (t)$ to be $ -\frac{1}{2\pi}$. There is no $\log(t)$ term coming from $T_{2,2} (t)$.  
We therefore have a contribution from the reflected term by 
$$-\frac{\pi}{4} \frac{1}{2\pi} = -\frac{1}{8}.$$
Recalling the factor of $2/\alpha$, the contribution to the variational Polyakov formula from these poles is simply 
\begin{equation} \label{reflected} -\frac{1}{4\alpha}. \end{equation} 

\subsection{Contribution from the integrals over the lines}
The line $\ell_1$ can be parameterized by $\ell_1(s) = -\pi+is$, $-\infty <s< \infty$,
and $\ell_2(s) = \pi + is$, now with $s$ going from $\infty$ to $-\infty$.  Write
$$\int_{\ell_1 \cup \ell_2} (f_1(z) - f_2(z)) dz = L_1 + L_2.$$

Note that if $\alpha = \pi/n$, for some $n\in \N$, then $f_1$ is periodic of period $2\pi$, 
$$f_1(z+2\pi)= \frac{e^{rr'\cos(z+2\pi)/2t}}{1-e^{i n \pi (\phi'-\phi -z -2\pi)/\pi}} = \frac{e^{rr'\cos(z)/2t}}{1-e^{i n (\phi'-\phi -z -2\pi)}}
$$
$$ = \frac{e^{rr'\cos(z)/2t}}{1-e^{i n (\phi'-\phi -z)}} = f_1(z).$$

Therefore $f_1$ takes the same values in the lines $\ell_1$ and $\ell_2$.  Since they have contrary orientation, the
integrals sum to zero.  The same holds for $f_2$, since 
$$f_2(z+2\pi)= \frac{e^{rr'\cos(z+2\pi)/2t}}{1-e^{i n \pi (\phi'+\phi -z -2\pi)/\pi}} = \frac{e^{rr'\cos(z)/2t}}{1-e^{i n (\phi'+\phi -z -2\pi)}}
$$
$$ = \frac{e^{rr'\cos(z)/2t}}{1-e^{i n (\phi'+\phi -z)}} = f_2(z).$$

In the general case, consider first $f_1$: 
\begin{eqnarray*}
L_1 &=& \int_{\ell_1\cup \ell_2} f_1(z) dz\\ 
&=& i \int_{-\infty}^{\infty} \left(\frac{e^{-rr'\cosh(s)/2t}}{1-e^{i\frac{\pi}{\alpha}(\pi+\phi'-\phi)}e^{\frac{\pi}{\alpha}s}}
- \frac{e^{-rr'\cosh(s)/2t}}{1-e^{i\frac{\pi}{\alpha}(-\pi+\phi'-\phi)}e^{\frac{\pi}{\alpha}s}}\right)\ ds.\\
\end{eqnarray*}

Restring to the diagonal, $r=r'$ and $\phi=\phi'$, we re-write 
$$L_1 = i \int_{-\infty} ^\infty e^{-r^2 \cosh(s)/(2t)} \left( \frac{1}{1-e^{\pi s/\alpha} e^{i\pi^2/\alpha}} - \frac{1}{1-e^{\pi s/\alpha} e^{-i\pi^2/\alpha}} \right) ds $$
$$= i \int_{-\infty} ^\infty e^{-r^2 \cosh(s)/(2t)} \left( \frac{e^{\pi s/\alpha} (2i \sin(\pi^2/\alpha))}{1+e^{2\pi s/\alpha} - e^{\pi s/\alpha} (2\cos(\pi^2/\alpha))} \right) ds $$
$$= -2 \sin(\pi^2/\alpha) \int_{-\infty} ^\infty e^{-r^2 \cosh(s)/(2t)} \frac{1}{e^{-\pi s/\alpha} + e^{\pi s/\alpha} - 2 \cos(\pi^2/\alpha)} ds $$
$$= - \sin(\pi^2/\alpha) \int_{-\infty} ^\infty \frac{e^{-r^2 \cosh(s)/(2t)}}{\cosh(\pi s/\alpha) - \cos(\pi^2/\alpha)} ds.$$ 

Including the factor of $\frac{e^{-r^2/2t}}{8\alpha \pi t}$, as well as the $\log(r)$, we compute 
$$\frac{1}{8\alpha \pi t} \int_0 ^R \int_0 ^\alpha e^{-r^2(1+\cosh(s))/2t} \log(r) r dr d\phi = \frac{1}{8\pi t} \int_0 ^R  e^{-r^2(1+\cosh(s))/2t} \log(r) r dr.$$
Next, we do a substitution letting 
$$u = \frac{r^2 (1+\cosh(s))}{2t}, \quad du = \frac{r(1+\cosh(s))}{t} dr,$$
so this becomes 
$$\frac{1}{16 \pi (1+ \cosh(s))} \int_0 ^{R^2 (1+\cosh(s))/2t} e^{-u} \log(2tu/(1+\cosh(s))) du.$$
It follows from our previous estimates that the integral from $R^2(1+\cosh(s))/2t$ to $\infty$ is rapidly vanishing as $t \downarrow 0$.  Hence, we may simply compute 
$$\frac{1}{16 \pi (1+ \cosh(s))} \int_0 ^\infty e^{-u} \log(2tu/(1+\cosh(s))) du $$
$$= \frac{1}{16\pi(1+\cosh(s))} \left( \log\left( \frac{2}{1+\cosh(s)} \right) + \log(t) - \gamma_e \right).$$
Thus, we have for $L_1$ in the case that $\alpha \neq \frac{\pi}{n}$ for any $n \in \N$, a contribution coming from 
$$- \sin(\pi^2/\alpha) \int_{-\infty} ^\infty \frac{\log\left( \frac{2}{1+\cosh(s)}\right)  - \gamma_e}{16\pi(1+\cosh(s))(\cosh(\pi s /\alpha) - \cos(\pi^2/\alpha))} ds $$
$$- \log(t) \sin(\pi^2/\alpha) \int_{-\infty} ^\infty \frac{1}{16\pi(1+\cosh(s))(\cosh(\pi s /\alpha) - \cos(\pi^2/\alpha))} ds.$$
Recalling the factor of $2/\alpha$, this gives a contribution to the variational Polyakov formula
$$- \frac{2}{\alpha} \sin(\pi^2/\alpha) \int_{-\infty} ^\infty \frac{\log\left( \frac{2}{1+\cosh(s)}\right)  - \gamma_e}{16\pi(1+\cosh(s))(\cosh(\pi s /\alpha) - \cos(\pi^2/\alpha))} ds $$
$$+ \frac{2 \gamma_e}{\alpha} \sin(\pi^2/\alpha) \int_{-\infty} ^\infty \frac{1}{16\pi(1+\cosh(s))(\cosh(\pi s /\alpha) - \cos(\pi^2/\alpha))} ds.$$
In forthcoming work, we shall compute these integrals.  

Fortunately, there will be no contribution to our formula coming from $f_2$.  To see this, we compute analogously
\begin{eqnarray*}
L_2 &=& - i
\int_{-\infty}^{\infty} \left(\frac{e^{-rr'\cosh(s)/2t}}{1-e^{i\frac{\pi}{\alpha}(\pi+\phi'+\phi)}e^{\frac{\pi}{\alpha}s}}
- \frac{e^{-rr'\cosh(s)/2t}}{1-e^{i\frac{\pi}{\alpha}(-\pi+\phi'+\phi)}e^{\frac{\pi}{\alpha}s}}\right)\ ds\\
\end{eqnarray*}

Restricting to the diagonal, we obtain 
$$L_2 = \sin(\pi^2/\alpha)  \int_{-\infty} ^\infty \frac{e^{-r^2 \cosh(s)/(2t)} }{\cosh(s\pi/\alpha + 2\pi i \phi/\alpha) - \cos(\pi^2/\alpha)} ds. $$
As observed by Kac, the when one integrates $L_2$ over the domain, that is with respect to $r dr d\phi$, the result vanishes; see p. 22 of \cite{kac}.  It is not immediately clear there \em why \em the integral vanishes, because the computation is omitted.  Moreover, our setting is not identical, because we are integrating with respect to $\log(r) r dr d\phi$ rather than $r dr d\phi$.  However, upon closer inspection, it becomes apparent that the reason the integral of $L_2$ over the domain vanishes is due to integration with respect to the angular variable, $d\phi$.  For the sake of completeness, since this computation is only stated but not demonstrated in \cite{kac}, we compute the integral with respect to the angular variable $\phi$,
$$\int_0 ^\alpha \frac{1}{\cosh(s\pi/\alpha + 2\pi i \phi/\alpha) + C} d\phi, \quad C:= -\cos(\pi^2/\alpha).$$
We do the substitution 
$$\theta = s\pi/\alpha + 2\pi i \phi/\alpha,$$
and this becomes 
$$\frac{\alpha}{2\pi i} \int_{s\pi/\alpha} ^{s\pi/\alpha + 2\pi i} \frac{1}{\cosh(\theta) +  C} d\theta.$$
The integral is 
\begin{equation} \label{arctan} \left . - \frac{2 \arctan \left( \frac{(C-1) \tanh(\theta/2)}{\sqrt{1-C^2}} \right)}{\sqrt{1-C^2}} \right|_{\theta=s\pi/\alpha} ^{s\pi/\alpha + 2\pi i}. \end{equation}  
It suffices to compute that the value of the hyperbolic tangent is the same at both endpoints, 
$$\tanh\left( \frac{s\pi/\alpha + 2\pi i}{2} \right) = \frac{\sinh(i\pi + s\pi/(2\alpha))}{\cosh(i \pi + s\pi/(2\alpha))} = \frac{-\sinh(s\pi/(2\alpha))}{-\cosh(s\pi/(2\alpha))} = \tanh(s\pi/(2\alpha)).$$
This follows from the fact that $e^{\pm i \pi } = -1$, and so 
$$\sinh(i\pi + \theta) = - \sinh(\theta), \quad \cosh(i\pi + \theta) = - \cosh(\theta).$$
Consequently, since the $\tanh$ has the same values at the two endpoints, the whole quantity (\ref{arctan}) vanishes.  It follows that $L_2$ will make no contributions to our formula.



\subsection{The total expressions}
We begin with the total expression for the heat kernel on an infinite sector of opening angle $\alpha \in ]0, \pi[$ with Dirichlet boundary condition:  
\begin{multline*}
H_{\alpha}(r,\phi,r',\phi',t) = \frac {e^{-(r^2 + r'^2)/4t}}{8\pi \alpha t} 
\left( \sum_{k=k_{min}}^{k_{max}} 2\alpha e^{rr'\cos(\phi'-\phi + 2k\alpha)/2t}\right.\\ + 
\sum_{V_{\phi,\phi'}} 2 \alpha e^{rr'\cos(\phi'+\phi + 2j\alpha)/2t}\\
- \sin(\pi^2/\alpha) \int_{-\infty} ^\infty \frac{e^{-r r'\cosh(s)/2t}}{\cosh(\frac{\pi}{\alpha}s + i\frac{\pi}{\alpha}(\phi'-\phi)) - \cos(\pi^2/\alpha)} ds \\
\left. + \sin(\pi^2/\alpha) \int_{-\infty} ^\infty \frac{e^{-r r'\cosh(s)/2t}}{\cosh(\frac{\pi}{\alpha}s + i\frac{\pi}{\alpha}(\phi'+\phi)) - \cos(\pi^2/\alpha)} ds\right)
\end{multline*}
where $k_{min} = \ceil*{ \frac{-\pi}{2\alpha} }$, and $k_{max} = \floor*{\frac{\pi}{2\alpha}}$ if $\frac{\pi}{2\alpha} \not\in \Z$,  otherwise 
$k_{max} = \frac{\pi}{2\alpha} - 1$. For $0 < \phi'-\phi < \min \{ (\frac{\pi}{2\alpha} - \floor*{\frac{\pi}{2\alpha}}) 2\alpha, \alpha/2\}$, if 
$\frac{\pi}{2\alpha} \not\in \Z$, and $0 < \phi'-\phi < \alpha/2$ otherwise. 
The sets $$V_{\phi,\phi'}:=]\frac{-\pi-\phi-\phi'}{2\alpha}, \frac{\pi-\phi-\phi'}{2\alpha}[\cap \Z$$
are the same as the sets $V_{\phi}$ described in the proof of Lemma \ref{l:marvsimpl}.  

The total expression Polyakov's formula is obtained by putting together the previous computations, recalling the factor of $\frac{2}{\alpha}$, and including contribution of the constant coefficient of the heat trace. We combine all these ingredients to determine the coefficients $a_{2,0}$ and $a_{2,1}$ in the expansion (\ref{exp-exists}) and conclude with the variational Polyakov formula for all sectors.

Recall that the constant coefficient of the heat trace, which is $\zeta_{\Delta_{\alpha}}(0)$ in equation (\ref{t0-heattrace}), was computed according to \cite[equation (2.13)]{corners}. Including the factor of $\frac{2}{\alpha}$, the contribution to the Polyakov formula from the heat trace is 
$$\frac{2}{\alpha} \zeta_{\Delta_{\alpha}}(0) = \frac{2}{\alpha} \left( -\frac{1}{12} + \frac{\pi^2 + \alpha^2}{24 \pi \alpha} + 2 \frac{\pi^2 + \pi^2/4}{24 \pi (\pi/2)} \right).$$
This simplifies to 
$$\frac{\pi}{12\alpha^2} + \frac{1}{12\pi} + \frac{1}{4\alpha}.$$ 
Consequently, when we combine with the contribution of the reflected term (\ref{reflected}) the $\frac{1}{4\alpha}$ term vanishes. Adding the contributions of the direct term and of the line $L_1$ we obtain
\begin{multline*}
\left.\frac{\pa}{\pa \gamma} \big(-\log(\det(\Delta_{\gamma}))\big)\right|_{\gamma=\alpha}
= \frac{\pi}{12\alpha^2} + \frac{1}{12\pi} \\
+ \sum_{k\in W_{\alpha}} \Big(\frac{-\gamma_e}{2 \pi (1-\cos(2k\alpha))} + \frac{1}{4\pi(1-\cos(2k\alpha))} \log\left( \frac{2}{1-\cos(2k\alpha)}\right) \Big)+ ,\\
- \frac{2}{\alpha} \sin(\pi^2/\alpha) \int_{-\infty} ^\infty \frac{\log\left( \frac{2}{1+\cosh(s)}\right)  - \gamma_e}{16\pi(1+\cosh(s))(\cosh(\pi s /\alpha) - \cos(\pi^2/\alpha))} ds\\
+ \frac{2 \gamma_e}{\alpha} \sin(\pi^2/\alpha) \int_{-\infty} ^\infty \frac{1}{16\pi(1+\cosh(s))(\cosh(\pi s /\alpha) - \cos(\pi^2/\alpha))} ds,
\end{multline*}
where the set $W_{\alpha}$ is defined in the statement of Theorem \ref{allsectors}. Notice that if the angle $\alpha$ is of the form $\alpha = \frac{\pi}{n}$, for some $n \in \N$, then the terms with the integrals are omitted from the formula. 
\qed


\section{Determinant of the Laplacian on rectangles} \label{s:rect} 
In this section we prove Theorem \ref{thm:mdr}.   
Consider a rectangle of width $1/L$ and length $L$.  The spectrum of the Euclidean Laplacian on this rectangle with Dirichlet boundary condition can easily be computed using separation of variables, and it is 
$$\left\{ \frac{m^2 \pi^2}{L^2} + \frac{n^2 \pi^2}{w^2} \right\}_{m, n \in \N}.$$
Consequently the spectral zeta function has the following expression:
\begin{eqnarray*}\zeta_L(s) &=& \sum_{m,n \in \N}  \left( \frac{1}{\pi^2 m^2 /L^2 + \pi^2 n^2 L^2} \right)^s\\
&=&(\pi )^{-2s} \sum_{m, n \in \N} \frac{1}{|L|^{2s} | mz + n|^{2s}}, \qquad z = \frac{i}{L^2}.\end{eqnarray*}

\begin{proof}[Proof of Theorem \ref{thm:mdr}] 
We would like to use the computations in \cite[p. 204--205]{OPS}, and so we relate the above expression for the zeta function to the corresponding expression in \cite{OPS} for the torus by
$$\zeta_L(s) = \frac{(\pi)^{-2s}}{2} \left( \sum_{(m,n) \in \Z\times \Z \setminus (0,0)} \frac{1}{|L|^{2s} | mz + n|^{2s}} - 2 L^{-2s} \sum_{n \in \N} \frac{1}{n^{2s}} - 2 L^{2s} \sum_{m \in \N} \frac{1}{m^{2s}} \right).$$
By \cite[p. 204--205]{OPS},
$$G(s) := \sum_{(m,n) \in \Z\times \Z \setminus (0,0)} \frac{1}{|L|^{2s} | mz + n|^{2s}}$$
satisfies
$$G(0) = -1, \quad G'(0) = -\frac{1}{12} \log \left( (2 \pi)^{24} \frac{ (\eta(z) \bar{\eta}(z) )^{24}}{(L)^{24}} \right),$$
where $\eta$ is the Dedekind $\eta$ function.  Consequently,
$$\zeta_L(s) = \frac{1}{2 \pi^{2s}} \left( G(s) - 2 L^{-2s} \zeta_R(2s) - 2 L^{2s} \zeta_R(2s) \right),$$
where $\zeta_R (s)$ denotes the Riemann zeta function $\zeta_R (s) = \sum_{n \in \N} n^{-s}.$
Since the Riemann zeta function satisfies
$$\zeta_R (0) = -\frac{1}{2}, \quad \zeta_R '(0) = - \log \sqrt{2\pi},$$
we compute
\begin{eqnarray*}
\zeta_L'(0) &=& \frac{1}{2} G'(0) - \log \pi + 2 \log(2\pi) = - \log\left( \frac{2\pi |\eta(z)|^2}{L} \right) - \log \pi + 2 \log (2\pi)\\
&=&\log(2) - \log(|\eta(z)|^2/L).\end{eqnarray*}
Consequently we obtain the formula for the determinant
$$\det \Delta_L = e^{-\zeta_L'(0)} = \frac{|\eta(z)|^{2}}{2L}=\frac{|\eta(i/L^2)|^{2}}{2L} =:f(L).$$

Since the rectangle is invariant under $L \mapsto L^{-1}$, we also have 
$$f(L) = \frac{1}{2} \eta(iL^2)^2 L.$$
We briefly recall the definition and some classical identities for the Dedekind $\eta$ function.  First, we have
$$\eta(\tau) = q^{1/12} \prod_{n=1} ^\infty (1-q^{2n}), \quad q = e^{\pi i \tau}, \quad \im(\tau) > 0.$$

We use the following identity from \cite[p. 12]{gunther}, 
$$\log \eta(i/y) - \log \eta(iy) = \frac{1}{2} \log(y), \quad y \in \R^+.$$
Then, we compute for 
$$-\log(\det \Delta_L) = \zeta_L '(0) = -2\log(\eta(i/L^2)) + \log(L) + \log(2),$$
$$-i \frac{\eta'(i/y)}{\eta(i/y) y^2} - i \frac{\eta'(iy)}{\eta(iy)} = \frac{1}{2y} \implies 4 \eta'(i) = i \eta(i).$$
This shows that
\begin{equation} \label{square-min} \frac{d}{dL} \zeta_L '(0)  = \frac{4 i \eta'(i/L^2)}{\eta(i/L^2) L^3} + \frac{1}{L} \implies \frac{d}{dL} \zeta_{L=1} '(0)  \frac{4 i \eta'(i) + \eta(i)}{\eta(i)}=0. \end{equation}
Since
$\frac{d}{dL} \det \Delta_L = \left(\frac{d}{dL} \log(\det \Delta_L) \right) \det \Delta_L$,
and $\det \Delta_L >0$, we have that
$$\left.\frac{d}{dL}\det \Delta_L \right|_{L=1}=0.$$

Next, we show that $f(L)$ is monotonically increasing on $(0,1)$.  By symmetry under $L \mapsto L^{-1}$, this will complete the proof that the zeta regularized determinant on a rectangle of dimensions $L \times 1/L$ is uniquely minimized by the square of side length one.  

To prove this, we begin by recalling equation (1.13) from Hardy \& Ramanujan \cite[eqn (1.13)]{harram}, 
$$\eta(\tau) = \frac{q^{1/12}}{1 + \sum_{n=1} ^\infty p(n) q^{2n}}, \quad q = e^{\pi i \tau}, \quad \im(\tau) >0.$$
Above, $p(n)$ is the number theoretic partition function on $n$.  We therefore compute that 
$$2 f(L) = \eta(iL^2)^2 L = \frac{L e^{-\pi L^2/6}}{\left( 1+ \sum_{n=1} ^\infty p(n) e^{-2 \pi L^2 n} \right)^2}.$$
It is clear to see that the denominator is a monotonically decreasing function of $L$.  We compute that the numerator, 
$$L e^{-\pi L^2/6} \textrm{ is monotonically increasing on } L \in \left( 0, \sqrt{ \frac{3}{\pi}} \right).$$
Thus the quotient is monotonically increasing on that interval as well.  

Let us write
$$2 f(L) = F(L) \wt{G}(L), \quad F(L) = L e^{-\pi L^2/6}, \quad \wt{G}(L) = \left( 1+ \sum_{n=1} ^\infty p(n) e^{-2 \pi L^2 n} \right)^{-2}.$$
Then we have that $F, \wt{G}>0$ on $L>0$, and $\wt{G}'(L) > 0$ on $L > 0$.  We also have that $F'(L) > 0 $ 
for $0<L<\sqrt{3/\pi}$, $F'(\sqrt{3/\pi}) = 0$, and $F'(L) < 0$ for $\sqrt{3}{\pi} < L < 1$.  We wish to prove that 
$$(F\wt{G})' >0 \textrm{ on } \left[ \sqrt{ \frac{3}{\pi}}, 1 \right).$$
This is immediately true at the left endpoint by the preceding observations.  Thus, it is enough to show that 
$$\left| \frac{F'}{F} \right| < \left| \frac{\wt{G}'}{\wt{G}} \right| \textrm{ on } \left( \sqrt{ \frac{3}{\pi}}, 1 \right).$$

We already know that the equality $|F'/F| = |\wt{G}'/\wt{G}|$ holds at $L=1$.  Thus, after computing $|F'/F|$, we must show that 
$$\frac{\wt{G}'}{\wt{G}} > \frac{\pi L}{3} - \frac{1}{L}, \quad \sqrt{ \frac{3}{\pi}} < L < 1.$$
We compute 
$$\wt{G}'(L) = -2(1+\sum_{n \geq 1} p(n) e^{-2 \pi L^2 n})^{-3} (-4\pi L \sum_{n \geq 1} n p(n) e^{-2 \pi L^2 n}).$$
Thus 
$$\frac{\wt{G}'}{\wt{G}} = \frac{ 8 \pi L \sum n p(n) e^{-2 \pi L^2 n}}{1+ \sum p(n) e^{-2 \pi L^2 n}},$$
and we are bound to prove that 
$$\frac{\wt{G}'}{\wt{G}}= \frac{ 8 \pi L \sum n p(n) e^{-2 \pi L^2 n}}{1+ \sum p(n) e^{-2 \pi L^2 n}} > \left| \frac{F'}{F} \right| = \frac{\pi L}{3} - \frac{1}{L}, \quad \sqrt{ \frac{3}{\pi}} < L < 1.$$
Consequently, re-arranging the above inequality, we are bound to prove that 
$$A(L) > B(L), \quad \sqrt{ \frac{3}{\pi}} < L < 1,$$
where 
$$A(L) = \sum_{n \geq 1} n p(n) e^{-2 \pi L^2 n}, \quad B(L) = \left( \frac{1}{24} - \frac{1}{8\pi L^2} \right) \left( 1 + \sum_{n \geq 1} p(n) e^{-2\pi L^2 n} \right).$$ 

To prove that $A(L) > B(L)$ for $\sqrt{3}{\pi} < L < 1$, we first observe that $A(L) >0$ for all $L>0$.  Moreover, $A(L)$ is clearly a monotonically decreasing function of $L$.  We have calculated that $f'(1) = 0$, and $2f(L) = F(L) \widetilde G (L)$, which shows that 
$$\frac{\widetilde G' (1)}{\widetilde G(1)} = -\frac{F'(1)}{F(1)} = \left| \frac{F'(1)}{F(1)} \right| = \frac{\pi}{3} - 1.$$
Hence, $A(1) = B(1)$.  It is plain to see that $B(\sqrt{3/\pi}) = 0$.  Thus since $A$ is monotonically decreasing on $(\sqrt{3}/\pi, 1)$, and $A(\sqrt{3/\pi}) > B(\sqrt{3/\pi})$, it suffices to show that $B$ is strictly increasing on $(\sqrt{3/\pi}, 1)$.  If this is the case, then the graphs of $A$ and $B$ can only cross at most once on  $(\sqrt{3/\pi}, 1]$.  Since we know that at the left endpoint of this interval, we have $A>B$, and at the right endpoint, we have $A=B$, this shows that on the open interval $(\sqrt{3/\pi}, 1)$, $A>B$.  

We therefore compute 
$$B'(L) = \frac{2}{8\pi L^3} \left( 1 + \sum p(n) e^{-2 \pi L^2 n} \right) + \left( \frac{1}{24} - \frac{1}{8 \pi L^2} \right) \sum -4 \pi L n p(n) e^{-2 \pi L^2 n}$$
$$= \frac{1}{4\pi L^3}  \left( 1 + \sum p(n) e^{-2 \pi  L^2 n} \right) + \left( \frac{1}{6} - \frac{1}{2\pi L^2} \right) (-\pi L) \sum n p(n) e^{- 2 \pi  L^2 n}.$$
On $(\sqrt{3/\pi}, 1)$, 
$$\frac{1}{4\pi L^3}  \left( 1 + \sum p(n) e^{-2 \pi L^2 n} \right) >0,$$ 
whereas 
$$ \left( \frac{1}{6} - \frac{1}{2\pi L^2} \right) (-\pi L) \sum n p(n) e^{-\pi L^2 n} < 0.$$
Thus, it suffices to prove that 
$$\frac{1}{4 \pi^2 L^4}  \left( 1 + \sum p(n) e^{-2 \pi L^2 n} \right) >  \left( \frac{1}{6} - \frac{1}{2\pi L^2} \right) \sum n p(n) e^{-2 \pi L^2 n}, $$
for $L \in (\sqrt{3/\pi}, 1)$.
We have on this interval
$$\frac{1}{4 \pi^2 L^4}  \left( 1 + \sum p(n) e^{-2 \pi L^2 n} \right) > \frac{1}{4 \pi^2} .$$
So, it will be enough to prove that 
$$\frac{1}{4 \pi^2} >  \left( \frac{1}{6} - \frac{1}{2\pi L^2} \right) \sum n p(n) e^{-2 \pi L^2 n}.$$
On this interval
$$\frac{1}{6} - \frac{1}{2\pi L^2} \leq \frac{1}{6} - \frac{1}{2\pi} = \frac{\pi - 3}{6\pi}.$$
So, it is enough to prove that 
$$\frac{6\pi}{(\pi - 3)} \frac{1}{4 \pi^2} =\frac{3}{(\pi-3) 2\pi} >  \sum n p(n) e^{-2 \pi L^2 n}.$$
For one final simplification, since the sum on the right is a monotonically decreasing function of $L$, it will suffice to prove this inequality holds for the smallest possible $L = \sqrt{3/\pi}$.  Thus, it is enough to prove that 
$$\frac{3}{(\pi-3) 2\pi}  >  \sum n p(n) e^{-6 n}.$$
Now, we recall a recent estimate of the partition function (p. 114 of \cite{wdap}) 
$$p(n) \leq \frac{e^{c\sqrt{n}}}{n^{3/4}}, \quad c= \pi \sqrt{2/3} < 2.6, \quad \forall n \geq 1.$$
It is straightforward to verify that for all $n \geq 2$ we have
$$n^{1/4} e^{2.6 \sqrt{n}} \leq e^{2n}.$$
Thus we estimate 
$$ \sum n p(n) e^{-6 n} = e^{-6} + \sum_{n \geq 2} n p(n) e^{-6n} \leq e^{-6} + \sum_{n \geq 2} e^{-4n} = \frac{1}{e^6} + \frac{1}{e^8 - e^4} < 0.003.$$
On the other hand
$$\frac{3}{(\pi-3) 2\pi}>3.$$
This completes the proof.  
\end{proof}


\noindent {\bf Concluding remarks.}
Isospectral polygonal domains are known to exist \cite{gww2}, and one can construct many examples by folding paper \cite{chap}.  A natural question is:  how many polygonal domains may be isospectral to a fixed polygonal domain?  Osgood, Phillips and Sarnak used the zeta-regularized determinant to prove that the set of isospectral metrics on a given surface of fixed area is compact in the smooth topology \cite{OPS2}.  Can one generalize this result in a suitable way to domains with corners? Is it possible to define a flow, as \cite{OPS} did, which deforms any initial $n$-gon towards the regular one over time and increases the determinant? How large is the set of isospectral metrics on a surface with conical singularities?  These and further related questions will be the subject of future investigation and forthcoming work. 


\appendix
\section{Carslaw's formula for the Dirichlet heat kernel of the quadrant} 
In the case $\alpha=\pi/2$, the Dirichlet heat kernel for the (infinite) quadrant in polar coordinates was given in equation (\ref{eq:hkCuad}) 
\begin{multline*}
p_{C}(r,r', \phi, \phi',t)\\
=\frac{e^{-\frac{r^2 + r'^2}{4t}}}{2\pi t} \left( \cosh\left( \frac{rr'\cos(\phi'-\phi)}{2t} \right) -
\cosh\left( \frac{rr'\cos(\phi'+\phi)}{2t} \right) \right). 
\end{multline*}

We shall verify that this coincides with the formula in (\ref{eq:DFhk}) with $\alpha=\pi/2$.  The Dirichlet heat kernel by the method of Carslaw is 
\begin{multline}
H_{C}(r,\phi,r',\phi',t) = \frac {e^{-(r^2 + r'^2)/4t}}{8\pi \alpha t}
\left(\int_{A_0} e^{rr'\cos(w)/2t} \frac{1}{1-e^{i 2(\phi'-\phi -w)}} dw \right. \\
\left. -\int_{A_0} e^{rr'\cos(w)/2t} \frac{1}{1-e^{i2(\phi'+\phi -w)}} dw \right)
\end{multline}

We determine the poles of 
$$f_1(w) =  \frac{e^{rr'\cos(w)/2t}}{1-e^{i2(\phi'-\phi-w)}}, \text{ and } f_2(w) = \frac{e^{rr'\cos(w)/2t}}{1-e^{i2(\phi'+\phi-w)}}$$
located in $]-\pi, \pi[$.   In general, the poles of $f_1$ are at the points  $\phi'-\phi +\pi j$ for some $j \in \Z$.  
By symmetry, we may assume without loss of generality that $\phi' > \phi$, and that $\phi'-\phi \leq \pi/2$.  Then, the only $j \in \Z$ such that 
$\phi' - \phi + \pi j \in ]-\pi, \pi[$ are $j=0$ and $j=-1$.  We compute the residues at these poles:
\begin{eqnarray*}
{\text{Res}}_{z=\phi'-\phi+\pi j} \frac{e^{rr'\cos(z)/2t}}{1-e^{i2(\phi'-\phi-z)}}
&=& \lim_{z\to \phi'-\phi +\pi j}
\frac{(z-(\phi'-\phi - \pi j))e^{rr'\cos(z)/2t}}{1-e^{i2(\phi'-\phi-z)}}\\
&=& \frac{1}{2i} e^{rr'\cos(\phi'-\phi + \pi j)/2t}
\end{eqnarray*}

For $f_2$, the poles are in general at $w=\phi'+\phi +\pi j$, for $j \in \Z$.  Those poles within the interval $]-\pi , \pi[$, assuming without loss of generality $\phi' \geq \phi$ are again those with $j = -1$, and $j= 0$. 
The residues at these poles are:
\begin{eqnarray*}
{\text{Res}}_{z=\phi'+\phi+\pi j} \frac{e^{rr'\cos(z)/2t}}{1-e^{i2(\phi'+\phi-z)}}
&=& \lim_{z\to \phi'+\phi+\pi j}
\frac{(z-(\phi'+\phi + \pi j))e^{rr'\cos(z)/2t}}{1-e^{i2(\phi'+\phi-z)}}\\
&=& \frac{1}{2i} e^{rr'\cos(\phi'+\phi + \pi j)/2t}
\end{eqnarray*}

Since the angle is $\pi/2$, the integrals over the lines vanish, so putting everything together we obtain:
\begin{eqnarray*}
H_{C}(r,\phi,r',\phi',t) &=& \frac {e^{-(r^2 + r'^2)/4t}}{4 \pi t}
(e^{rr'\cos(\phi'-\phi)/2t} + e^{rr'\cos(\phi'-\phi -\pi)/2t}\\
& & - e^{rr'\cos(\phi'+\phi)/2t} - e^{rr'\cos(\phi'+\phi -\pi)/2t})\\
&=& \frac {e^{-(r^2 + r'^2)/4t}}{4 \pi t}(e^{rr'\cos(\phi'-\phi)/2t} + e^{-rr'\cos(\phi'-\phi)/2t}\\
& & - e^{rr'\cos(\phi'+\phi)/2t} - e^{-rr'\cos(\phi'+\phi)/2t})\\
&=& p_{C}(r,r', \phi, \phi',t).
\end{eqnarray*}

It is also interesting to verify that for the case of the quarter circle, although the Polyakov formula given in Theorem \ref{allsectors} is quite complicated, it is nonetheless consistent with the result of Theorem \ref{ctpio2}.  Especially, this is interesting because the proof of Theorem \ref{ctpio2} is independent of the proof of Theorem \ref{allsectors}. 

For the quarter circle, the only contribution from the poles of $f_1$ corresponds to $k=-1$, and this gives 
$$-\frac{\gamma_e}{4\pi}.$$
The contribution from the poles of $f_2$ is simply 
$$-\frac{1}{2\pi}.$$
The heat trace gives a contribution of 
$$-\frac{1}{3\pi} + \frac{1}{3\pi} + \frac{1}{12\pi} = \frac{11}{12\pi}.$$
Putting all of these together, we have
$$-\frac{\gamma_e}{4\pi} + \frac{5}{12\pi},$$
which indeed coincides with our calculation in (\ref{variation-qc}).


\end{document}